\documentclass[11pt,a4paper]{amsart}
\usepackage{amsmath,amsfonts,amsthm,amssymb,graphicx, svg} 
\usepackage{caption}
\newtheorem{theorem}{Theorem}
\newtheorem{lemma}{Lemma}

\newtheorem{definition}{Definition}
\newtheorem{proposition}{Proposition}

\theoremstyle{remark}
\newtheorem{remark}{Remark}

\usepackage[utf8]{inputenc}
\usepackage[T1]{fontenc}
\usepackage{color}
\usepackage{hyperref}
\usepackage[sort,nocompress]{cite}
 \usepackage[normalem]{ulem}

\usepackage{mathtools}

\usepackage[mathscr]{euscript}
\usepackage{enumerate}

\def\i{\mathrm{in}}
\def\out{\mathrm{out}}

\usepackage{tikz,circuitikz}

\usepackage{enumerate}
\usepackage{comment}

\newcommand{\mntext} {}
\newcommand{\newnewtext}{}

\newcommand{\mnewtext}{}
\newcommand{\mltext}{}

\newcommand{\newnewtex}{} 

\newcommand{\commented}[1]{}

\newcommand{\e}{\varepsilon}

\renewcommand{\d}{\mathrm{d}}
\newcommand{\eps}{\varepsilon}
\newcommand{\p}{\partial}

\newcommand{\Sp}{\mathbb{S}} 

\newcommand{\Z}{\mathbb{Z}} 
\newcommand{\R}{\mathbb{R}} 
\newcommand{\bD}{\mathbb{D}} 

\renewcommand{\H}{\mathcal{H}} 
\renewcommand{\P}{\mathcal{P}}
\newcommand{\I}{I} 
\newcommand{\T}{\mathscr{T}}
\newcommand{\X}{\mathcal{X}}
\newcommand{\Y}{\mathcal{Y}}
\newcommand{\WF}{\mathrm{WF}} 
\newcommand{\ext}{\mathrm{ext}} 
\newcommand{\loc}{\mathrm{loc}}

\newcommand{\supp}{\mathrm{supp}} 

\newcommand{\gpert}{ {\widetilde{g}} }
\newcommand{\Source}{\mathscr{S}}
\newcommand{\sct}{B}
\newcommand{\Char}{\mathrm{char}}

\renewcommand{\div}{\mathrm{div}}
\newcommand{\nullinf}{\mathcal{I}}

\newcommand{\gpenrose}{{g_{\hat N}}} 
\newcommand{\M}{{M }} 
\newcommand{\gM}{{g_M }} 
\newcommand{\N}{{N }} 
\newcommand{\gN}{{g_N }} 
\newcommand{\Next}{{N_\ext}} 
\newcommand{\gext}{{g_\ext}} 

\newcommand{\E}{{\mathcal E}}

\newcommand{\Nextended}{\Next}

\newcommand{\Wminus}{{\Omega_{\mathrm{in}}}}
\newcommand{\Wplus}{{\Omega_{\mathrm{out}}}}

\newcommand{\goodN}{\mathcal{N}} 
\newcommand{\U}{\mathcal{U}} 

\newcommand{\NP}{\mathrm{NP}}
\newcommand{\SP}{\mathrm{SP}}

\newcommand{\extension}[1]{} 
\newcommand{\noextension}[1]{} 
\newcommand{\observation}[1]{} 
\newcommand{\generalizations}[1]{} 
\newcommand{\motivation}[1]{} 
\newcommand{\nomotivation}[1]{} 
\newcommand{\canberemoved}[1]{} 

\newcommand{\muutos}[1]{{#1}}

\newcommand{\norm}[1]{\left\Vert #1\right\Vert}

\newcommand{\vece}{{\vec\varepsilon}}

\def \mcl {{\mathcal L}}

\def \mcu {{\mathcal U}}

\def \bpf {\begin{proof}}
\def \epf {\end{proof}}

\renewcommand{\tilde}[1]{\widetilde{#1}}
\renewcommand{\hat}{\widehat}

\newcommand{\mtext}{}
\def \bfo {\begin {eqnarray*} }
\def \efo {\end {eqnarray*} }
\def \ba {\begin {eqnarray*} }
\def \ea {\end {eqnarray*} }
\def \beq {\begin {eqnarray}}
\def \eeq {\end {eqnarray}}

\title[Inverse scattering in Lorentzian manifolds]{Inverse scattering problems for non-linear wave equations
on Lorentzian manifolds}

\author[Alexakis,  Isozaki, Lassas, Tyni]{Spyros Alexakis,  Hiroshi Isozaki, Matti Lassas, Teemu Tyni}
\address{Spyros Alexakis, Department of Mathematics, University of Toronto, Canada.\vspace{-3mm}}
\address{Hiroshi Isozaki,  College of Mathematics,
 University of Tsukuba,  Japan.\vspace{-3mm}}
\address{Matti Lassas,  Department of Mathematics and Statistics, University of Helsinki, Finland.\vspace{-3mm}}
\address{Teemu Tyni,
Research Unit of Applied and Computational Mathematics, University of Oulu, Finland.}
\date{January 14, 2025}

\begin{document}

\begin{abstract}
We show that an inverse scattering problem for a semilinear wave equation can be solved
on a manifold having an asymptotically Minkowskian infinity, that is,
scattering functionals determine the topology, differentiable structure, and
the conformal type of the manifold. Moreover, the metric and the coefficient
of the non-linearity are determined up to a multiplicative transformation.
The manifold on which the inverse problem is considered is allowed to be an open, globally hyperbolic manifold which may have non-trivial topology
or several infinities (i.e., ends) of which at least one has
to be of the asymptotically Minkowskian type. 
To formulate the inverse problems we define
a new type of data, non-linear scattering functionals, which are defined also in the cases
where the classically defined scattering operators are not well-defined. This makes it possible to solve inverse problems also in cases where some of the incoming waves lead to a blow-up of the scattered solution.
We use non-linear interaction of waves  as a beneficial tool that helps to solve the inverse problem. The corresponding inverse problem for the linear wave equation still remains unsolved.

\end{abstract}
\maketitle

 \tableofcontents

\section{Introduction: inverse scattering problems}

In this work we consider scattering problems for non-linear wave equations on Lorentzian manifolds.
Examples of scattering problems include radar, seismic and optical  imaging 
\cite{Scattering1B,Scattering3,Scattering4,StefanovJAMS}
and theory of experiments in particle physics, see \cite{Isozakibook,Scattering9}.
In typical linear cases, scattering problems  \cite{Scattering5,Scattering7,Scattering8,Scattering10}
are formulated in asymptotically flat space-times $(\R^{3+1},g)$ for linear wave equations $\square_g u =0$ as the task of recovering information about the space-time metric $g$ from measurements of the scattering operator $S: u_-\mapsto u_+$.
Here $u_\pm(s,\theta)=\lim_{t\to\pm\infty} |t|^{(3-1)/2} u(t+s,t\theta)$, $s\in\R$, are the past and future radiation fields.
The geometrical formulation of inverse scattering problems have been studied e.g. in \cite{GuillarmouSalo,Joshi,Joshi2,scatt1} for a linear wave equation with a time-independent metric. 
Inverse scattering problems for non-linear wave equations with a known metric and unknown non-linear term has been  studied in \cite{Stefanov1,scatt2,scatt3,scatt4}.

The work \cite{KLU2018} introduced a method now known as the {higher order linearization method} to inverse problems for nonlinear equations.
The method is based on self-interaction of waves in the presence of non-linearities, and has successfully been used to solve many inverse problems from local measurements (see e.g.\ \cite{NLeB,NLe6,NLe1,NLwe1} for elliptic problems and \cite{LassasICM,LLPT1,LLPT2,LUW} for wave equations).
Such methods are not available for linear wave equations, where the existing uniqueness results 
are limited to manifolds that are close to subsets of Minkowski space \cite{AFO,AFO2}, or have a time-independent or real-analytic metric \cite{Eskin}.
For the latter case the results are based on Tataru's unique continuation theorem \cite{Tataru1}, and these results have been shown to fail for general metric tensors that are not analytic in the time variable \cite{Alinhac}.

In \cite{KLU2018,NLWave2, FLO,KLOU}
inverse problems with near-field measurements 
have been studied  for non-linear wave equations, e.g., of the form
\beq\label{non-linear eq A}
\sum_{j,k=0}^3 g^{jk}(x)\frac {\p^2 u}{\p x^j\p x^k}(x)+
a(x)u(x)^\kappa=f(x)
\eeq
where $x=(x^0,x^1,x^2,x^3)=(t,y)\in \R^{1+3}$,  $x^0=t$ is the time-variable and $g^{jk}(x)$ is a Lorentzian metric. 
For this equation one can consider an open set $V\subset \R^{1+3}$ and the {\it source-to-solution map} $L:f\to u|_V$ that maps a (sufficiently small) source $f$ supported in the set $V$ to the restriction $u|_V$ of the corresponding solution $u$ of \eqref{non-linear eq A}.
This map is similar to the Dirichlet-to-Neumann map or the Cauchy data set, often used in the study inverse boundary value problems 
 \cite{kenig2014calderon,kenig2014recent,sun1997inverse,uhlmann2009electrical,sylvester1987global}
 for conductivity equation,
\cite{ferreira2009limiting,guillarmou2011calderon,kenig2014calderon,kenig2014recent,Novikov1,Novikov2,sylvester1987global}
for Schr\"odinger equation and \cite{guillarmou2011calderon} for Laplace-Beltrami equation
and \cite{oksanen-realprincipal,salo2012inverse,QE1,QE2} for general real principal type operators and systems.
 On counterexamples, see \cite{Daude-nonuniqueness,LiimatainenO}.
The work \cite{KLU2018} studied the question of which the (local) source-to-solution operator uniquely determine the  metric $g^{jk}(x)$ in a maximal set  $W\subset \R^{1+3}$ to which causal signals can propagate from $V$ and return to $V$.
Using non-linearity as a beneficial tool, inverse problems have been solved for several hyperbolic equations, for example \cite{NLg5,NLg7,NLg11,NLg4,NLg13}  for wave equation
and \cite{Stefanov-nonlinear} of Westervelt
equation, and \cite{NLg3,NLg4,Zhou1,Zhou2,Zhou3} for hyperbolic systems of equations. 
Some physically motivated works are \cite{KLOU} and \cite{NLWave2,FLO,NLWave4}, where inverse problems for the Einstein and Yang-Mills equations were considered, respectively.
Related questions have also been studied in \cite{NLe1,Krupchyk1,Krupchyk2,Krupchyk3,LaiZ} for elliptic equations.
There are also results for the Bolzmann equation \cite{Boltz1,Boltz2,LaiBoltzmann} 
and non-linear parabolic equations \cite{para1}.

In inverse scattering problems on Lorentzian manifolds, a motivating problem are wave scattering on a black hole \cite{Baskin-scattering1,Baskin-scatterin2,Hintz1} and the Doppler ultrasound imaging in moving medium \cite{Doppler-imaging} (on non-linear models used in ultrasound imaging, see \cite{NLg1,Stefanov-nonlinear,Uhlmann-acoustics}).
Our space-times will allow the causal structure of black hole exteriors, and in fact will allow multiple (asymptotically Minkowskian) ``ends'', and also multiple event horizons. However we should make clear that {we expect that the space-times we consider cannot be solutions to Einstein's equations with reasonable matter models. ({For those with a complete maximal Cauchy surface,} this is because their ADM energy will be zero.) Nonetheless, traditional notions of general relativity such as null infinities ${\mathcal I}^+, {\mathcal I}^-$, and event horizons {continue to make}  sense. So we adopt this terminology here.

The present paper has two goals:
The first one is to consider inverse scattering problems in a more general setting than is traditionally done and  introduce a new type of limited
scattering data, the {\it scattering functionals} (see Definition \ref{def: Scattering functionals} below).
{\mnewtext  In the study of scattering (and inverse scattering) for quasilinear equations (such as the Einstein equations) with incoming data coming from past null infinity on encounters several technical problems; {among many other difficulties,
one has}  the possibility of the causal structure of the space-time changing because of the quasilinearity. (The dynamical formation of black holes being one such example.) Even restricting to semi-linear equations, one may not be able to define a scattering map due to the finite-time breakdown of the solutions. Yet}
the traditional definition of the scattering operator \cite{Zworski1,Isozakibook,Melrose,LaxPhillips}, see also
\cite{Baez,Baez2,Taujanskas} on non-linear problems, requires
global existence of solutions.
To consider inverse problems in {\mnewtext  such potentially} unstable physical systems  where global existence of solutions may fail, we define  the  {scattering functionals} {\mnewtext  on incoming data that are suitably small}. These form a more limited data set which exists even in the case when 
the solutions {\mnewtext  may } not  exist globally -- The scattering functionals share the same feature as the generalized curvature used in synthetic non-smooth geometry (see \cite{Villani}):
The scattering functionals are equivalent to the scattering operator when it is well-defined, but the
scattering functionals are defined also in more general settings.


{\newnewtext Below, we will consider  Lorentzian manifolds that are assumed to be
globally hyperbolic. Note that this is a standard condition that guarantees that hyperbolic
equations have local solutions with compactly supported sources.
We will show that using 
the scattering functionals, it is possible to reconstruct the topology of the manifold and the conformal class of its metric 
under the assumption} that the manifold has at least one asymptotically Minkowskian infinity, and that the whole space-time is causally connected to this infinity in the future and the past.
The scattering functional are associated to using small amplitude  in-going waves and observing their scattering on a \emph{portion} of future null infinity.
We remark that the produced waves may well blow up somewhere in space-time; the key idea is that we observe their scattering data at only a part of the future light-like infinity ${\mathcal I}^+$ \emph{prior} to any potential blow-up.

The second goal of the paper is to show that the inverse scattering problems can be reduced to inverse problems for near field observations, including
the case when the perturbation of the metric is not compactly supported. 
For linear equations, inverse scattering problems and inverse problems with near field measurements are in many cases equivalent \cite{Ber}, while
for the non-linear equations the relation of these problems is not understood.
In this paper our aim is to develop a {\it  general method to  reduce inverse scattering problems to near-field problems}.
This method could be used to solve many classical inverse scattering problems for non-linear models.
Our approach is based on {\mnewtext  the Penrose compactification} of the space-time.
There, the non-compact Lorentzian manifold $(\R^{3+1},g)$ is mapped to $(\hat N,\gpenrose)$, where $\hat N$ is a compact subset of $\Sp^3\times\R$ (see Figure \ref{fig:penrose2}, Right).
The compactified manifold $(\hat N,\gpenrose)$ can be extended {smoothly} to
a larger, open subset of $\Sp^3\times\R$. It turns out that the past and future radiation fields can be thought of as resulting from waves produced by sources in the non-physical extension.
Hence, one could say that the scattering problem is reduced to a problem where the {\it measurements are done beyond the infinity of the physical space}.
 
Summarizing, we study the following inverse scattering problem:

\medskip

{\it Inverse scattering problem:} 
In a globally hyperbolic space-time with at least one infinity that is an asymptotically Minkowski space, does the scattering functionals  for a non-linear wave equation  uniquely determine  the topology and
differentiable structure of the underlying space-time and (the conformal type of) the Lorentzian metric?

\subsection{A simplified scattering problem}

{To warm up, let us consider the nonlinear wave equation
\begin{equation}\label{eq:nonlinear wave equation simple}
\square_g u(t,y) + 
{{a}}(t,y)u(t,y)^\kappa = 0,\quad x=(t,y)\in \R^{1+3}
\end{equation}
where $g$ is a globally hyperbolic Lorentzian metric on $\R^{1+3}=\R\times\R^3$,  and $\kappa\ge 4$ is an integer and
${{a}}(t,y)\not =0$ is a smooth Schwartz class rapidly decreasing function in $\R\times\R^3$.
We denote  the coordinates
of the Minkowski space by $x=(t,y)\in \R\times\R^3$ and the Minkowski metric by $\eta = -\d t^2 + \sum_{j=1}^3 (\d y^j)^2$.
To warm up, we consider the case when the topology of the space-time is that of $\R^4$  the difference $g-\eta$ and all of its derivatives vanish faster than any polynomial uniformly as $|x|\to \infty$. {\mnewtext  Analogously to the inverse problems for the linear wave equations, see \cite{Melrose}, we will consider the
radiation fields
\beq\label{Minkowski radiation fields}
u_\pm(s,\theta)=\lim_{t\to\pm\infty} |t|^{(3-1)/2} u(t+s,t\theta)
\eeq
where $\theta\in \mathbb S^2=\{\theta\in \R^3:\ \|\theta\|_{ \R^3}=1\}$ is the direction where the asymptotics of $u$ is observed  and  $s\in\R$ is a delay parameter. Below, by re-parametrizing the ingoing radiation field $u_-(s,\theta)$, $(s,\theta)\in \R\times \mathbb S^2$ as a function $h_-$, it holds that for any $s_{\i},s_{\out}\in \R$ there is $\e(s_{\i},s_{\out})>0$ such that 
when the function $u_-(s,\theta)$ is supported in $(-s_{\i},s_{\i})\times
 \mathbb S^2$ and its Sobolev norm in $H^k(\R\times
 \mathbb S^2)$, $k\ge 5$ is smaller than $\e(s_{\i},s_{\out})$,
 then the value of the outgoing radiation field $u_+(s_{\out},\theta_{\out})$
 is well-defined. For such ingoing radiation fields we
 define the scattering functionals $S_{s_{\i},s_{\out},\theta_{\out}}$ which map the function $u_-(s,\theta)$ to the number $u_+(s_{\out},\theta_{\out})$.
 
}

\subsubsection{Penrose {\mnewtext compactification} on the perturbed Minkowki space}
To define a geometric scattering operator for the wave equation \eqref{eq:nonlinear wave equation simple},
we next recall the properties of the \emph{Penrose {\mnewtext compactification}} studied in detail e.g.\ in 
 \cite{Penrose-republication-original,Wa1984}. To do this,
consider the Minkowski space $\R^{1+3}$, $n=1+3$, with time $t$ and spherical space coordinates $(r,\theta,\varphi)
\in [0,\infty)\times  [0,\pi] \times  [0,2\pi]$, in which
the Minkowski metric is given by
$$
\eta := \d s^2 = -\d t^2 + \d r^2 + r^2\left( \d\theta^2 + \sin^2(\theta)\d\varphi^2 \right).
$$
We use the auxiliary coordinates $v = t+r$ and $u=t-r$ and define  
a metric  $\eta_c= \Omega^2\eta$ that is a conformal to the Minkowski metric
 with the conformal factor 
 \beq\label{def: Omega}
 \Omega^2 = 4(1+v^2)^{-1} (1+u^2)^{-1}.
 \eeq
To represent $(\R^{1+3},\eta_c)$ in the Penrose coordinates, we define
 a map $\Phi : \R\times \R^3 \to \R\times\Sp^3$.
To do that, on $\R\times\Sp^3$ we use the time coordinate $T\in \R$ and on the 3-dimensional sphere $ \Sp^3$ the Riemannian normal coordinates 
 $(R,\theta,\varphi)\in [0,\pi]\times  [0,\pi] \times  [0,2\pi]$
 at  the North pole (denoted $\NP$). Here, $R$ is the distance from the North pole.
 On $\R\times\Sp^3$ we use the product Lorentzian metric, given in the above coordinates by
 \begin{align}\label{eq:Conformal minkowski metric}
g_{\R\times\Sp^3}= \d \widetilde{s}^2 &= -\d T^2 + \d R^2 + \sin^2(R)\left( \d\theta^2 +\sin^2(\theta)\d\varphi^2 \right).
\end{align}
We consider  the map $\Phi:\R^{1+3}\to\R\times\Sp^3$ that maps a point $x\in  \R\times \R^3$ with
 the coordinates $(t,r,\theta,\varphi)$ to a point on $\R\times\Sp^3$  that has 
 the coordinates $(T,R,\theta,\varphi)$  with 
 \begin{equation}\label{eq:Penrose coordinates}
\begin{split}
T &= \tan^{-1}(v)+\tan^{-1}(u),\\
R &= \tan^{-1}(v)-\tan^{-1}(u).
\end{split}
\end{equation}
Then ${{{{\hat N}}} }=\Phi(\R^{1+3})\subset \R\times\Sp^3$  consists of the points whose coordinates $(T,R,\theta,\varphi)$ satisfy
\begin{align*}
 -\pi < T+R < \pi,\quad
-\pi < T-R < \pi,\quad
R\geq 0.
\end{align*}
The map $\Phi:(\R^{1+3},\eta_c)\to ({{{{\hat N}}} },g_{\R\times\Sp^3}) $
is an isometric diffeomorphism. This implies that the Minkowski space
$(\R\times\R^3,\eta)$ is conformal to the subset ${{{{\hat N}}} } \subset \R\times\Sp^3$ endowed with the standard product metric
of $ \R\times\Sp^3$.
We will call the image ${{{{\hat N}}} }=\Phi(\R\times\R^3)$  the \emph{Penrose {\mnewtext compactification}} of the Minkowski space (see Figure \ref{fig:penrose2}, Left).

We use subsets of the boundary $\p {{{{\hat N}}} }\subset  \R\times\Sp^3$ that are named as follows: The future light-like infinity $\nullinf^+$ and  the past light-like infinity $\nullinf^-$,
\begin{equation}\label{def:null_infinities} 
 \nullinf^+=\p  {{\hat N} }\cap \{0<T<\pi\},\quad
  \nullinf^-=\p  {{\hat N} }\cap \{-\pi<T<0\},    
\end{equation}
  and the 
 future time-like infinity ${i}_+$,  the 
 past time-like infinity ${i}_+$, and the space-like infinity (i.e., the spatial infinity)
  ${i}_0$, are the points
  \begin{align}\label{def:other_infinities}
 {i}_+&=\p  {{\hat N} }\cap \{T=\pi\}=\{(NP,\pi)\},\\ {i}_-&=\p  {{\hat N} }\cap \{T=-\pi\}=\{(NP,-\pi)\},\\ {i}_0&=\p  {{\hat N} }\cap \{T=0\}=\{(SP,0)\}. 
 \end{align}
  Below, we use $i_0$ to denote both the single point set $\{(\SP,0)\}$ and the point $(\SP,0)$ and use similar notations for the points $i_\pm$.
 
\begin{figure}
\centerline{
    \includegraphics[height=6cm]{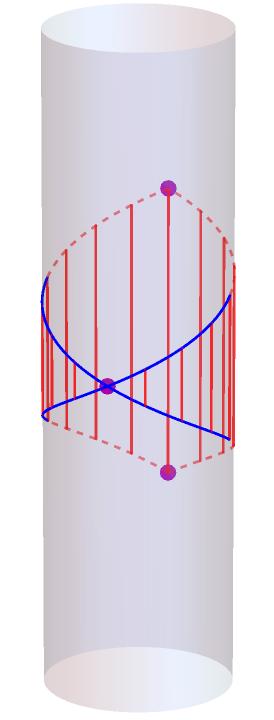}\hspace{17mm} \includegraphics[height=6cm]{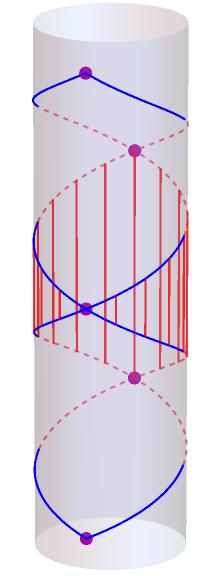}
    }
\caption{
 {\bf Left:} The Penrose map  is a conformal map
 $\Phi:\R\times \R^3\to \R\times\Sp^3$ and its image ${{{\hat N}}} = \Phi(\R\times\R^3)\subset  \R\times\Sp^3$ is 
 the Penrose {\mnewtext compactification} of the Minkowski space, see \eqref{eq:Penrose coordinates}.  In the figure $\R\times\Sp^3$ is visualized as
 a  cylindrical surface $\R\times \Sp^1$, and ${{{\hat N}}} $ is visualized as the area shaded by the red lines,
 that is, ${{{\hat N}}} $ is visualized as a subset that is cut from the cylinder by two ``circles'', one of which passes through  the points $i_0$ and $i_-$ and 
 the other passes through $i_0$ and $i_+$. 
 The lower part of the  boundary of ${{{\hat N}}} $ is  the past conformal infinity and the upper part of the boundary of ${{{\hat N}}} $ is the future conformal infinity. 
 {\bf Right:} The Penrose {\mnewtext compactification} ${\hat N}$ is 
 the extended to a manifold $\Next$ by gluing to ${\hat N}$ 
non-physical extensions on the other sides of the future and the past (light-like) infinities. 
In the figure, the boundary of extended space-time $\Next$ is marked by blue (color online) curves.
The shaded region is the Penrose {\mnewtext compactification} ${\hat N}$ of the Minkowski space.
The space ${\hat N}$ is extended to the Lorentzian manifold  
$\Next={\hat N}\cup {{{N}}}^+\cup {{{N}}}^-$ where ${{{N}}}^+$ and ${{{N}}}^-$
are the non-physical parts of extended manifold that are the mirror images of the
space ${\hat N}$ on the other side of the future and the past light-like infinities.
The boundary of $\Next$ is marked in the figure by blue curves.
}
\label{fig:penrose2} 
\end{figure}

Let $g_c = \Omega^2g$ be a metric that is conformal to the perturbed Minkowski metric $g$ on 
$\R^{1+3}$  with the conformal factor $\Omega^2$ and
denote by $\tilde g=\Phi_*g_c$ the push-forward metric on ${{{\hat N}}} $.
By using the transformation properties of the conformal Laplacian we see that a function $u:\R^{1+3}\to \R$ satisfies the non-linear wave equation \eqref{eq:nonlinear wave equation simple}
if and only if $\tilde u=(\Omega^{-1}u)\circ \Phi^{-1}$ solves the wave equation
\begin{equation}\label{simplified non-linear problem}
 \big( \square_{\gpert} + B_{\gpert}\big)\tilde u + A\cdot (\tilde u)^\kappa=0,\quad\text{in } {{\hat N}},
\end{equation}
where 
\begin{align}\label{eq: potentials A and B}
A := (\Phi^{-1})^*(a\Omega^{\kappa-3} ),
\quad
B_{\tilde g} :=-\frac{1}{6}(\Phi^{-1})^* ( R_{\Omega^2 g}-\Omega^{-2}R_g ).
\end{align}
and $R_g$ is the scalar curvature of the metric $g$.

As seen below, when $\kappa\ge 4$ is an
integer, 
for any $0>T_0>-\pi$ there are $s_\kappa>0$ and $\e>0$ such that boundary value problem
\beq\label{simplified non-linear problem 123}
\begin{cases}
& \big( \square_{\gpert} + B_{\gpert}\big)\tilde u + A\cdot (\tilde u)^\kappa=0,\quad\text{in } {{{{\hat N}}} }\\
&\tilde u|_{\mathcal I^-}=h,\\
&\tilde u=0\quad \hbox{for }T<T_0 
\end{cases}
\eeq
has a unique solution,
where $h\in H^{s_k}({\nullinf^-})$ satisfies $\supp(h)\subset \{T_0<T<0\}$
and $\|h\|_{H^{s_\kappa}({\nullinf^-}) }<\e$, where $ H^s({\nullinf^-})$  denotes the Sobolev space
on ${\nullinf^-}$  with smoothness index $s$. As ${{{\hat N}}} $  is a conformal compactification
of the Minkowski space, we call \eqref{simplified non-linear problem 123} a scattering problem.

Thus, the zero-function has in the space $C^\infty_0({\nullinf^-})$
a neighborhood $\mathcal U$ in which the  scattering problem \eqref{simplified non-linear problem 123}
has a unique solution for all $h\in  \mathcal U$. We define the geometric scattering operator
$S_{{{{\hat N}}} ,\tilde g,A}:\mathcal U\to C^\infty({\nullinf^+})$
 for 
the equation \eqref{simplified non-linear problem 123} by setting
\beq
S_{{{{\hat N}}} ,\tilde g,A}(\tilde u|_{\mathcal I^-})= \tilde u|_{\nullinf^+},\quad\hbox{for } h=\tilde u|_{\nullinf^-}\in \mathcal U.
\eeq
We will prove the following theorem for the perturbed Minkowski space.

\begin{theorem}\label{main thm for Minkowski}
Let $\eta$ be the standard Minkowski metric in the space $\R^{1+3}$, ${{a}}(x)$  be  a nowhere vanishing, Schwartz rapidly decreasing function
and $g$ be a globally hyperbolic Lorentzian metric in  $\R^{1+3}$ 
such that the tensor
$g_{jk}(x)-\eta_{jk}$ 
is a Schwartz rapidly decreasing function. Then the 
geometric scattering operator
$S_{{{{\hat N}}} ,\tilde g,A}$, defined in a neighborhood of the zero function in $C^\infty_0({\nullinf^-})$
determines the {\mltext conformal class of the metric $g$ uniquely.}
\end{theorem}

\subsection{Scattering functionals and manifolds with an asymptotically Minkowskian infinity}

In this section, we will use the standard causality notations used in Lorentzian geometry
defined below in section \ref{sec: Basic notations}. We will consider manifolds which may
not be homeomorphic to the Minkowski space.

\subsubsection{Manifolds with an asymptotically Minkowskian infinity}
$ $\\ 
Let  $g_{\R\times\Sp^3}=-dt^2+g_{\Sp^3}$ be the standard Lorentzian metric of the product space $ \R\times\Sp^3$.

 As defined above, let $\Phi: \R^{1+3}\to \R\times\Sp^3$ be Penrose's conformal map, denote by  ${{{{\hat N}}}} =\Phi(\R^{1+3})$ the Penrose {\mnewtext compactification} of the Minkowski space. We denote $g_{{{{\hat N}}} }=g_{\R\times\Sp^3}|_{{{{\hat N}}} }$ ((see Figure \ref{fig:penrose2}, Left). 
 Recall that $\Omega\in C^\infty(\R^{1+3})$ is the function for which
 $\Phi: (\R^{1+3},\Omega^2g_{\R^{1+3}}) \to ({{{{\hat N}}} },g_{{{{\hat N}}} })$ is an isometry.
 Also, let $\omega=\Omega\circ \Phi^{-1}:{{{\hat N}}} \to \R_+$.
 
 \begin{definition}\label{Def: visible infinity} Let  $V\subset\R^{1+3}$ and $g_V$ be a Lorentzian metric on $V$.
 We say that $(V,g_V)$ is a neighborhood of  
the 
 light-like infinity in $\R^{1+3}$ with an asymptotically Minkowskian metric
 if 
 \begin{itemize}
  \item[(i)] there is an open set $\hat V\subset \R\times\Sp^3$ such that
 $\nullinf^+\cup \nullinf^-\cup \{i_0\}\subset \hat V$ and
  $$
 V=\Phi^{-1}({\hat V}\cap {{{{\hat N}}} }),
 $$
 \item[(ii)]  there is a $C^\infty$-smooth Loretzian metric $g_{\hat V}$ on $\hat V$
 such that $g_{\hat V}=g_{\R\times\Sp^3}$ on ${\hat V}\setminus {{{{\hat N}}} }$ and
 $$
 \Omega^2 g_V=\Phi^*g_{\hat V}\quad\hbox{on }V.
 $$
 \end{itemize}

\end{definition}

\begin{definition}\label{def:asymptotically nice infinity A}

{\it A manifold $(\M,\gM)$ has an asymptotically Minkowskian infinity $E$ (up to infinite order)
that is visible in the whole space-time $\M$ if 
  \begin{itemize}
  \item[(i)] $(\M,\gM)$ is a globally hyperbolic manifold and it has a subset 
 $E\subset \M$ such that $(E,\gM|_E)$ is isometric to 
  a neighborhood $(V,g_V)$ of  
the 
 light-like infinity in $\R^{1+3}$ with an asymptotically Minkowskian metric.

  \item[(ii)]  $J^+_\M(E)=J^-_\M(E)=\M$.
  \end{itemize}}
\end{definition}

{\mltext Let us note that it is possible to construct examples of space-times with an
asymptotically Minkowskian infinity whose topological completion admits a boundary with the structure of a bifurcate null surface, with the geometric properties of the event horizons of the Schwarzschild exterior, non-trivial topology, or several ends. See for example Section~\ref{sec: examples of spacetimes}, Examples 3-4.
Also, in section \ref{sec: generalizations} we consider conformally equivalent models similar to those
used in cosmology.}

Observe that  above it is possible  that $i_+,i_-\not \in \hat V$. This happens 
in the examples where we consider product space-times 
$M=\R\times (\R^3\# K)$ where $K$ is a compact, closed 3-dimensional manifold
which is not simply connected and $\R^3\# K$ is the connected sum of $\R^3$ and $K$. See Example 2 and Fig. \ref{fig:product manifold1} below
on a `wormhole' manifold $N_0=\R^3\# K$ with a handlebody that is obtained by glueing smoothly a 3-torus $K=\mathbb T^3$ and $\R^3$.
Roughly speaking, when we do the conformal change in Lorentizan metric, the handlebody becomes arbitrarily small when one approaches
the time-like infinity $i_+$, and thus $i_+$ can be considered as a singular boundary point of the conformally compactified space-time 

When $(E,\gM|_E)$ is an asymptotically Minkowskian infinity,
we see (using the notations of the above definitions) that there is an isometry 
(see Fig. \ref{fig:visualization of E}) $$\psi:(E,\gM|_E)\to (\hat V\cap {{{\hat N}}},
\omega^{-2}g_{\hat V}),$$

{\mltext We say that a function $a\in C^\infty (\M)$ is  a Schwartz class function in an 
 asymptotically Minkowskian infinity of $\M$ if for all $\alpha\in \mathbb N^4$ and $m\in \Z_+$ there is $C_{\alpha,m}>0$ such that
 \beq\label{Schwartz function}
|\p_x^\alpha ( a\circ \psi^{-1}\circ \Phi(x))|\leq C_{\alpha,m}(1+|x|)^{-m},\quad \hbox{for all }x\in V,
 \eeq
  where $|x|$ is the Euclidean length of $x\in \R^4=\R^{1+3}$.}
 
\begin{figure}[ht!]
\centerline{
\includegraphics[height=5.5cm]{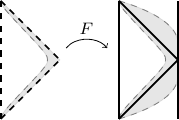}
}
\caption{Visualization of the definition of the asymptotically Minkowskian 
infinity $E\subset M$. The figures show the Penrose diagrams that are 2-dimensional analogs of
the cylinders shown in Figure \ref{fig:penrose2}.
The map $F$ takes $V$ conformally to $\hat V\cap {{{\hat N}}} $ where ${{{\hat N}}} =\Phi(\R^{1+3})\subset \R\times\Sp^3$ is the Penrose {\mnewtext compactification}
of the Minkowski space  and
 $\hat V\subset \R\times\Sp^3$ is a neighborhood of $\nullinf^+\cup \nullinf^-\cup i_0$.
The Lorentzian metric on $\hat V$ coincides with the standard metric of $\R\times\Sp^3$
outside ${{{\hat N}}} $.
 }
\label{fig:visualization of E}
\end{figure}

Let 
 $\omega_\M\in C^\infty(\M)$ be a strictly positive function satisfying $\omega_\M|_E=\omega\circ \psi$.
Without loss of generality, we can assume that 
$\psi$ is the identity map, that is,
$E$ and ${\hat V}\cap {{{\hat N}}} $ are identified as sets and the metric tensors on them are conformal,
{\mnewtext that is, they are the same up to a conformal factor.}
Let us denote   
\beq\label{N is M}
(\N,\gN)=(\M,(\omega_M)^{2}\gM),
\eeq
that is, $\N=\M$ but we use {\mnewtext different symbols for the two conformally related space-times}.

Next we extend the manifold $(\N,\gN)$
by gluing subsets of $ \R\times\Sp^3$ to it.
Let  (see Figure \ref{fig:penrose2}, Right)
\ba
& &{{{N}}}^+= J^+_{\R\times\Sp^3}(\nullinf^+\cup  i_0)\setminus J^+_{\R\times\Sp^3}(i_+)\subset {\R\times\Sp^3},\\
& &{{{N}}}^-= J^-_{\R\times\Sp^3}(\mathcal I^-\cup i_0)\setminus J^-_{\R\times\Sp^3}(i_-)\subset {\R\times\Sp^3}
\ea
be endowed with the Lorentzian metric $g_{\R\times\Sp^3}$ of ${\R\times\Sp^3}$. 
Recall, that  we can assume that $E$ in Definition \ref{def:asymptotically nice infinity A}
is identified with a subset of $\R\times\Sp^3$.

We define a Lorentzian manifold
\beq\label{extended manifold}
\Next=\N\cup {{{N}}}^+\cup {{{N}}}^-
\eeq
such that the differentiable structure of $\Next$ in $E\cup {{{N}}}^+\cup {{{N}}}^-$ coincides with the  one inherited from ${\R\times\Sp^3}$ (see Figure \ref{fig:penrose2}, Right).  We emphasize that $\Next$ is an open Lorentzian manifold and the time-like infinities $i_+$ and $i_-$ are not contained in $\Next$. Moreover, $i_+$ and $i_-$ may be singular boundary points of $\Next.$
The Lorentzian metric $\gext$ of $\Next$  is such a   $C^\infty $-smooth metric that on $\N$ it coincides with $\gN,$ and   
on $ {{{N}}}^+\cup {{{N}}}^-$  it coincides with the standard metric $g_{\R\times\Sp^3}$ of $\R\times\Sp^3$.

\commented{
\begin{figure}[ht!]
\centerline{\includegraphics[width=0.9\textwidth]{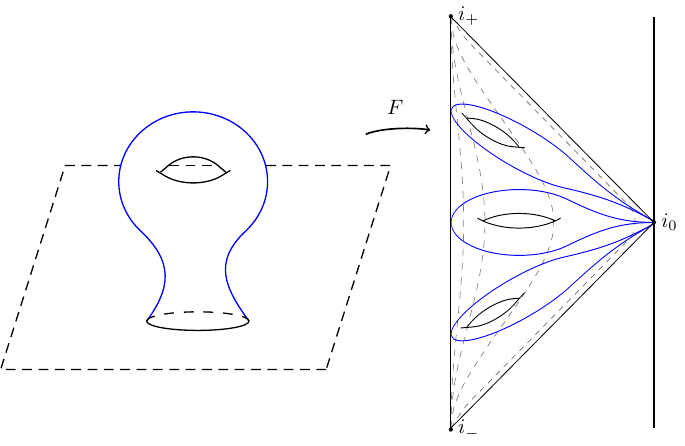}}
\caption{ 
{\bf Left.} A schematic visualization of a Riemannian manifold $(N,g_N)$ with an asymptotically 
Euclidean infinity. The Lorentzian manifold $(M,g_M)$, where $M=\R\times M$ and $g_M=-dt^2+g_N$
has a Minkowskian infinity $E$ and a non-trivial topology.
{\bf Right.}  A schematic visualization of the Lorentzian manifold $(M,\omega_M^2g_M)$ that is conformal to $(M,g_M)$. The manifold $(M,\omega_M^2g_M)$ is extended to a manifold $\Next$ by gluing `non-physical' subsets ${{N}} ^-$ and  ${{N}} ^+$  to the conformal light-like infinities. The scattering data on $(M,g_M)$ determines the source-to-solution map on $\Next$
corresponding to sources on  ${{N}} ^-$ and  observation on ${{N}} ^+$.}
\label{fig:handle compactification}
\end{figure}

}

\subsubsection{Scattering problem on a manifold with an asymptotically Minkow\-skian infinity}

For  $q\in \nullinf^+$ and  $p\in \nullinf^-$ we denote, slightly abusing the notations 
 \[
 I^-_{\M}(q)= I^-_{\Next}(q)\cap N,\quad  I^+_{\M}(p)= I^+_{\Next}(p)\cap N
 \]
 and 
 \[
 J^+_{\mathcal I^-}(p)=\mathcal I^-\cap J^+_{\Next}(p),\quad I^-_{\mathcal I^+}(q)=\mathcal I^+\cap I^-_{\Next}(q).
 \]

Let $k\in \Z_+$, $k\geq 5$,
and 
$h_-\in H^k(I^-)$ be supported in the future of the point $p\in \nullinf^-$.
Let  $q,q_1\in \nullinf^+$ be such that $q_1>q$, that is, $q_1$ is in the future of $q$. 
Let $a(x)>0$ be a {\mltext Schwartz class function in an asymptotical
Minkowskian infinity of  $\M$.}

\begin{definition} {\mnewtext  Let $k\ge 5$, ${q_0}\in \nullinf^+$,
$M({q_0}):=I^-_{M}(q_0)\subset M$ and $h_-\in C(\nullinf^-)$ be supported in the future of the point $p\in \nullinf^-$. 
{\mntext Let $a(x)$ and $d(x)$ Schwartz functions on $M$.} We say that a function $u\in H^k_{\loc}(M({q_0}))\cap C(M({q_0}))$ is a solution of the scattering
problem on $( M({q_0}),\gM)$  at rest prior to $p$  with the past radiation field $h_-$,
if
\beq\label{general non-linear problem on M}
  & &\square_{\gM}  u(x){\mntext +d(x)u(x)} + {\mltext a(x)} u(x)^{\kappa}=0,\quad\text{on } I^-_{\M}(q_0),\hspace{-20mm}
  \\ \label{general non-linear problem on M2}
  & &\lim_{x\to q} \omega_M(x)^{-1} u(\Phi(x))= h_-(q)\quad \hbox{for all }q\in  {\mathcal I^-},\\ 
  & &u=0\hbox{ on }M\setminus J^+_\M(p). \label{general non-linear problem on M3}
  \eeq
Moreover, we say that 
$\tilde u\in H^k( I^-_{\Next}(q_0)\cap N)$
and is a solution of the Goursat-Cauchy boundary value problem with the past radiation field (or with the Goursat data) $h_-$ if
\beq\label{general non-linear problem1}
  & & \square_{\gN} \tilde u(x)+ (B_{\gN}(x)+{\mntext D(x)})\tilde u(x) + A(x) \tilde u(x)^{\kappa}=0,\quad\text{in } N\cap I^-_{\Next}(q_0),\hspace{-20mm}
  \\ \label{general non-linear problem2}
  & &\tilde u|_{\nullinf^-}=h_-,\\
  & &\tilde u=0\hbox{ on }\Next\setminus J^+_{\Next}(p)  \label{general non-linear problem3}
  \eeq
  where 
  \begin{align}\label{ABD functions}
A:= a\cdot \omega_M^{{\kappa}-3},\quad {\mntext D:=d\cdot \omega_M^{{\kappa}-1}}
\quad
B_{\gN} :=-\frac{1}{6} (R_{g_N}  - \omega_M^{2}R_{g_M}).
\end{align}}
\end{definition}

{\mnewtext  

We say that $u$ has the future radiation field $h_+$ 
if 
\beq\label{future rad field formula}
\lim_{x\to q} \omega_M(x)^{-1} u(\Phi(x))= h_+(q)\quad \hbox{for all }q\in  {\mathcal I}^+(q_0),
\eeq
see formula \eqref{Minkowski radiation fields}.
 Also, we say that $\tilde u$ has the future radiation field $h_+$ if 
\beq  \label{general non-linear problem4}
\tilde u|_{ I^-_{\nullinf^+}(q_0)}=h_+. 
\eeq
The} existence and uniqueness of the solution $\tilde u$ of the 
Cauchy-Goursat problem \eqref{general non-linear problem1}-\eqref{general non-linear problem3} is considered  below in Theorem~\ref{thm:nonlinear_goursat copy}.
{\mnewtext 
By Sobolev embedding theorem $H^k(M(q_0))\subset C(M(q_0))$ for $k>2$, and we see that when $\tilde u$ is a solution of the Goursat-Cauchy boundary value problem
on $N\cap I^-_{\Next}(q_0)$ with the future radiation field $h_+$ 
then  
 $$
 u(x)=\omega_M(\Phi^{-1}(x))\tilde  u(\Phi^{-1}(x))
 $$ is a solution of
the scattering
problem on $( M(q_0),\gM)$ having the same future radiation field $h_+$.}

}
  
  \subsubsection{The past and future radiation fields of the waves that are compactly supported in space at any time.} In this section, we consider waves in the Minkowski space.
  The extension of the Penrose compactifiction of the Minkowski space $\R^{1+3}$
  is the product space $\R\times \mathbb S^3$ with the metric
  $-dT^2+g_{\mathbb S^3}$, where $g_{\mathbb S^3}$ is the Riemannian
  metric of the unit 3-sphere ${\mathbb S^3}\subset \R^4$. 
  For $R>0$, let 
  $$
  P(R)=\{\Phi(t,y)\mid t=0,\ y\in \R^3,\ |y|\leq R\}
  $$
  be the image of the set $\{0\}\times\overline B_{\R^3}(0,R)$ under the Penrose map $\Phi$.
  Moreover, let
  \ba
  S(R)=\{\gamma_{x,\xi}(s)\in \R\times \mathbb S^3\mid x\in P(R),\  \xi\in L_x(\R\times \mathbb S^3),\ s\in \R\}
\ea
(see Figure \ref{fig:sets B}, Left) and let
 \ba
  S_-(R)=S(R)\cap  \nullinf^-
  \ea
  be the set of the intersection points of light-like geodesics on the past light-like  geodesics
  emanating from the points in $S(R)$.
  Observe that $S_-(R)\subset\nullinf^-$ is compact.
  Figure~\ref{fig:sets B} depicts the sets $P(R),S(R)$, and $S_-(R)$.
\begin{figure}
    \centering
    \includegraphics[height=6cm]{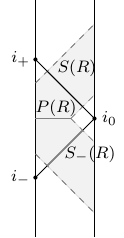}
    \hspace{2cm}
    \includegraphics[height=6cm]{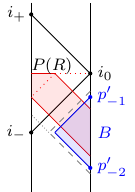}
    \caption{{\bf Left:} The set $P(R)$ is shown as the horizontal bold gray line. The grayed region depicts the set $S(R)$, while its restriction $S_-(R)$ to the past null infinity $\nullinf^-$ is shown as the diagonal gray line.
    {\bf Right}:
    Sets and the support of the cut-off function $\rho$ used in the proof of Theorem~\ref{thm:nonlinear_goursat copy}.}
    \label{fig:sets B}
\end{figure}

  Let $
  (\phi_0,\phi_1)\in \mathcal E'({\mathbb S^3}\setminus \{i_0\})^2$
  be distributions supported on $P(R)$.
  As the scalar curvature of the sphere ${\mathbb S^3}$ equals 6, we see that
in the Penrose compactification of the Minkowski space $\tilde b=1$. By
\cite[Appendix A]{Lax}, on the unit 3-sphere ${\mathbb S^3}$ the
wave equation 
\beq\label{wave on sphere2 A1} 
& &(\p_T^2-\Delta_{\mathbb S^3}+1)\tilde u=0,\quad \hbox{on }\R\times \mathbb S^3,
\\  \label{wave on sphere2 A2} 
& &(\tilde u|_{T=0},\p_T\tilde u|_{T=0})=(\phi_0,\phi_1),
\eeq
 satisfies the strong
 Huygen's principle, that is,
\begin{equation*}
\begin{split}
\supp(\tilde u)\subset \{\gamma_{x,\xi}(s)\in \R\times \mathbb S^3 \mid \ &x=(0,y)\in\R\times \mathbb S^3,\hbox{ where} \\
&\hspace{-25mm}y\in \supp(\phi_0)\cup \supp(\phi_1),\ \xi\in L_x(\R\times \mathbb S^3),\ 
s\in \R\},
\end{split}
\end{equation*}
where $L_x(\R\times \mathbb S^3)$ is the set of light-like vectors in the tangent space $T_xM_0$ of  manifold $M_0=\R\times \mathbb S^3$. 
Thus for  $
  (\phi_0,\phi_1)\in \mathcal E'({\mathbb S^3})^2$
  supported on $P(R)$ it holds that
\beq
\supp(\tilde u)\subset S(R),\quad 
\supp(\tilde u|_{\nullinf^-})\subset S_-(R).
\eeq
 
Observe also that the Penrose map $\Phi$ maps  the  surface $\{0\}\times \R^3$
to the  surface $\{0\}\times {\mathbb S^3}$.
Thus, the support of the initial data $(u|_{t=0},\p_tu|_{t=0})$
is compact in $\R^3$ if and only if the 
the support of the initial data $(\tilde u|_{T=0},\p_T\tilde u|_{T=0})$
is a compact subset of ${\mathbb S^3}\setminus \{i_0\}$.
Using this observation we define 
{\mnewtext the past and future radiation fields arising from compactly supported initial data}
\begin{align}\label{eq: sets B}
\mathcal B_\pm&=\{\tilde u|_{ \nullinf^\pm} \mid \tilde u\hbox{ solves 
\eqref{wave on sphere2 A1}-\eqref{wave on sphere2 A2} with
}(\phi_0,\phi_1)\in \mathcal E'({\mathbb S^3}\setminus \{i_0\})^2\},
\end{align}
where $\mathcal E'({\mathbb S^3}\setminus \{i_0\})$ is the set of distributions on
$\mathbb S^3$ whose support does not contain the point $i_0$.
{\mnewtext  We note that when  $\tilde u$ is the solution of the wave equation
having the initial data $(\phi_0,\phi_1)\in \mathcal E'({\mathbb S^3}\setminus \{i_0\})^2$,
then the wavefront set of $\tilde u$ does not intersect the normal bundle of 
$\nullinf^+$ or $\nullinf^-$, and thus the traces $\tilde u|_{ \nullinf^\pm}$
are well-defined as distributions, see \cite{Duistermaat}.}
Observe that if $h_-\in\mathcal B_-$ then $\supp(h_-)$ is a compact subset of
${\nullinf^-}$. Below, we call $\mathcal B_-$ and $\mathcal B_+$
the past and future radiation fields of concentrated waves (that in Minkowski
space correspond to waves $u(t,y)$ that at any time $t$ are
supported in a bounded subset of $\R^3$).  
We also denote for $R>0$
\begin{align}
\mathcal B_\pm(R)&=\{\tilde u|_{ \nullinf^\pm} \mid \tilde u\hbox{ solves \eqref{wave on sphere2 A1}-\eqref{wave on sphere2 A2} with
}(\phi_0,\phi_1)\in \mathcal E'(P(R))^2\},
\end{align}

On space-time $ \R\times \mathbb S^3$ we define a time function
${\bf t}: \R\times \mathbb S^3\to \R$ by setting ${\bf t}(p):=t$
for $p=(t,y)\in \R\times \mathbb S^3$. In particular, this defines the
time function ${\bf t}$ for 
$p=(t,y)\in N\cup \nullinf^-\cup \nullinf^-$.
Moreover, for $t_-,t_+\in [-\pi,\pi]$,  {\mnewtext $t_-<t_+$,} 
we denote  $\nullinf^-(t_-,t_+)=\{p\in \nullinf^- \mid {\bf t}(p)\in (t_-,t_+)\}$ and 
 $\nullinf^+(t_-,t_+)=\{p\in \nullinf^+ \mid {\bf t}(p)\in (t_-,t_+)\}$, and
for $t_1<0$
\beq\label{def R(p) NEW}
R(t_1)=\inf\{R>0\mid   \nullinf^-(-\pi-t_1,t_1) \subset  S_-(R)\}.\hspace{-15mm}
\eeq
For $-\pi<t_1<0$ it holds that $R(t_1)<\infty$ and the points $i_0,i_-$ have neighborhoods
$U_0,U_-\subset \R\times \mathbb S^3$, respectively, such that $S(R(t_1))\cap (U_0\cup U_-)=\emptyset$.

Let us next describe the scattering problem (see, e.g. \cite{Zworski1,Joshi,Melrose,SaBarretoWang1,Vasy-scattering0,Vasy-scattering2}) in a more general setting.
To this end, we consider scattering functionals which are also defined in many cases, where the scattering operator is not defined.

\begin{definition}\label{def: Scattering functionals} Let $k\ge 3$,
$-\pi<t_1<0$, 
 $q,q_1\in \nullinf^+$, where $q_1$ is in the future of $q$ 
 and $\e>0$ and  
\begin{equation}
\begin{split}
 \label{domain of S}
\mathcal D(S_{t_1,{q}})&=\mathcal D_{(\e)}(S_{t_1,{q}})= \{h\in H^k(\nullinf^-)\cap \mathcal B_-(R(t_1))\mid 
 \|h\|_{H^k(\nullinf^-)}<\e\}
 \end{split}
\end{equation}
be  an open neighborhood
of the zero function in $H^k(\nullinf^-)\cap \mathcal B_-(R(t_1))$.
Let  $\e=\e(t_1,q_1)>0$ be so small that for any past radiation field $h_-\in \mathcal D_{(\e)}(S_{t_1,{q}})$
there is a unique solution $\tilde u$  for
the Goursat-Cauchy boundary value problem
\eqref{general non-linear problem1}-\eqref{general non-linear problem3}.
Let $q\in \nullinf^+(0,t_2)$  and  $h_+$ be the future radiation field of $\tilde u$, see \eqref{general non-linear problem4}. Then, we say  that the {\mnewtext  (non-linear) functional} %
 \ba
& &S_{t_1,{q}}:\mathcal D_{(\e)}(S_{t_1,{q}})\to \R,\\
& &S_{t_1,{q}}(h_-)=h_+(q)
\ea
is a scattering functional 
associated to $(\M,\gM,a)$, and the time $t_1$ and the point $q\in \mathcal I^+$ (See Fig.\ \ref{fig: sources}). 
We also denote $S_{\M,\gM,a,{\mntext d};t_1,{q}}=S_{t_1,{q}}$.
\end{definition}

Observe also that if a scattering operator $S_{\M,\gM,a}$ exists, it determines the scattering functionals  $S_{\M,\gM,a;t_1,{q}}$ for all $(t_1,{q})$.

\medskip

\begin{remark}
 {\mnewtext    The past and future radiation fields of concentrated waves (i.e., waves in the Minkowski space 
    produced by compactly supported initial data) particularly suitable
    in study of inverse problems as a phenomenon called
the {\it causality violation at infinity (CVI)}, see  
\cite{Penrose-diagrams4-violation-of_causality_at_i0}, does not appear for them. In this  phenomenon one considers the extended spacetime $\R\times \mathbb S^3$ and light-like geodesic 
travelling from $\nullinf^-$ to  $\nullinf^+$  through the space-like infinity $i_0$.
Let us consider Green's function
$G(x,x')$ of the wave equation, $(\p_T^2-\Delta_{\mathbb S^3})G(\cdot,x')=\delta_{x'}$,
where  
 the source point $x'$ is $i_-$. The wavefront set of the function
 $G(x,i_-)$ is the union of bicharacteristics corresponding to light-like geodesics
  that travel through the point $i_0$. The function  $G(x,i_-)$ is a wave on 
  $\R\times \mathbb S^3$ whose wavefront set contains the normal bundle of $\nullinf^+$,
and thus  one can state that this wave carries information from $i_-\cup \nullinf^-$
to $\nullinf^+$, see also \cite{Zworski1}. However, the wave  $G(x,i_-)$  vanishes in the physical part $N=I^+_{\R\times \mathbb S^3}(i_-)\cap I^-_{\R\times \mathbb S^3}(i_+)$ of
the space-time $\R\times \mathbb S^3$. Moreover, there are distributions $f$ supported on the past light-like infinity $\nullinf^-$ 
for which the solution $u$ of the wave equation $\square_{\R\times \mathbb S^3}u+u=f$ is such that the singularities propagate along the (non-smooth) surface $\nullinf^-\cup i_0\cup \nullinf^0$ to the future light-like infinity but the wave $u$ is $C^\infty$-smooth on the physical part of the space-time.
This paradox is called the causality violation at infinity.
As the past radiation fields
 $h_-\in \mathcal B_-$ give rise to waves vanishing near $i_0$, these waves avoid causality violation at infinity.
 Below, we show that the scattering functionals $S_{\M^{(1)},\gM,a,d;t_1,{q}}$,
 where $-\pi<t_1<0$ and $q\in \nullinf^+$, are equivalent to the source-to-solution maps
 defined for sources supported in compact subsets $K_n\subset (N^-)^{\mathrm{int}}$, see \eqref{Kn sets}. In particular
 these sources are not supported on $\nullinf^-$, and therefore do not produce waves
 whose wavefronts propagate along the light-like geodesics that pass through $i_0$.
}
\end{remark}

\subsection{Main result}

Our main result is the following uniqueness result for the inverse scattering problem for 
a semi-linear wave equation. 

\begin{theorem}\label{main thm for general manifold}
Let $(\M^{(j)},g^{(j)})$, $j=1,2$ be two globally hyperbolic manifolds with asymptotically Minkowskian infinities (up to infinite order) that are visible in the whole space-time $\M^{(j)}$, {\mnewtext see Definition \ref{Def: visible infinity}}.
Let $a^{(j)},{\mntext d^{(j)}}\in C^\infty(\M^{(j)})$  be Schwartz functions in an asymptotically Minkowskian infinity of $M^{(j)}$, $a^{(j)}(x)\not=0$ for all $x\in \M^{(j)}$, and ${\kappa}\ge 4$.
 Assume that   for all $-\pi<t_1<0$ and $q\in \nullinf^+$  the scattering functionals  
 (see Definition \ref{def: Scattering functionals}) for equations 
 \eqref{general non-linear problem on M}-\eqref{general non-linear problem on M3} satisfy
$$
S_{\M^{(1)},\gM^{(1)},a^{(1)},{\mntext d^{(1)}};t_1,{q}}(h)
=S_{\M^{(2)},\gM^{(2)},a^{(2)},{\mntext d^{(2)}};t_1,{q}}(h)
$$
when  $ h\in \mathcal D(S_{\M^{(1)},\gM^{(1)},a^{(1)},{\mntext d^{(1)}};t_1,{q}})\cap \mathcal D(S_{\M^{(2)},\gM^{(2)},a^{(2)},{\mntext d^{(2)}};t_1,{q}})$.
Then there is a diffeomorphism $\Psi:\M^{(1)} \to \M^{(2)}$ and 
a function $\gamma\in C^\infty(\M^{(1)})$ such that the metric tensors $g^{(j)}$
and the coefficients $a^{(j)}$ of the non-linear terms 
satisfy  
\begin{equation}\label{eqtrans B}
\begin{split}
& g^{(1)} = e^{2\gamma}\Psi^*g^{(2)},\\
& a^{(1)}=  e^{(\kappa-3)\gamma(x)}\Psi^*a^{(2)},
\end{split}
\end{equation}
that is, the non-linear scattering functionals uniquely determine the topology, the differentiable structure, and the conformal type of the Lorentzian manifold,
and the Lorentzian metric, and the coefficient function of the non-linear term up to the transformations in \eqref{eqtrans B}. 
\end{theorem}

\bigskip

   \begin{remark} While working on the paper we became aware of another group of researchers, P. Hintz, A. S\'{a} Barreto, G. Uhlmann, and Y. Zhang, who were also working on non-linear inverse scattering problems, and in private communication we agreed on publishing their and our preprints on the same day. It turned out that their results in \cite{hintz2024inversenonlinearscatteringmetric} and our work are complementary, in the sense that the classes of manifolds studied in these works are disjoint. Moreover,  \cite{hintz2024inversenonlinearscatteringmetric} analyses more general non-linear equations than those considered in this paper.
    An essential difference in the papers is that the conformal compatifications of the Lorentzian manifolds studied in \cite{hintz2024inversenonlinearscatteringmetric} are globally hyperbolic spaces that have a spatial infinity that is diffeomorphic to $\Sp^{n-1}$, where the dimension of the Lorentzian manifold is $n+1$. 
   In contrast, in the present paper we consider Lorentzian manifolds whose conformal compactification has a spatial infinity that is a single point $i_0$, similarly to the Penrose compactification of the Minkowski space \cite{Penrose-republication-original}.  This makes the classes of manifolds considered in the papers different.
For example, for the Minkowski space $(\R^{1+3},\eta)$, denoted below by $M_0$,
the conformal boundary (that is, the conformal infinity) $\partial_c M_0$ is defined by using equivalence classes
of Cauchy sequences on a certain bundle $\mathcal P^1 (M_0)$ over the space-time $M_0$ so that the conformal boundary depends only on the conformal
structure of the space-time $M_0$ and is thus independent of the choice of conformal compactification; see \cite[Section 2]{Schmidt}.
By results of B. Schmidt, see \cite{Schmidt}, the conformal boundary $\partial_c M_0$ of the
Minkowski space can be identified with the 
boundary of $N$ in $\R\times \mathbb S^3$, that is,
with $\mathcal I^-\cup \mathcal I^+\cup i_+\cup i_-\cup i_0$,
where  the light-like infinities $\mathcal I^-,\mathcal I^+$ are the subsets defined above of the boundary of $N\subset \R\times \mathbb S^3$, see \eqref{def:null_infinities}, 
and  the future time-like infinity $i_+$, the past time-like infinity $i_-$, and the spatial infinity $i_0$ are the points of the boundary of $N$, see \eqref{def:other_infinities}.
This implies that that conformal boundary of the Minkowski space contains only three points,  $ i_+,i_-$, and $ i_0$,
which are not limits of light-like geodesics of $M_0$ (that is, points for which there exists a light-like geodesic $\gamma\subset M_0$
such that  the equivalence class of the  point is a limit point 
of the lift of the curve $\gamma$ on $\mathcal P^1(M_0)$).
Hence, for the Minkowski space $\R^{1+3}$ there do not exist two conformally flat extensions where the spatial infinity is on one extension a single point and the sphere $\mathbb S^2$ on the other extension.  
 We refer the reader to the monograph by R.  Wald,  \cite[Section 11.1]{Wa1984}, for in-depth discussion on the structure of spatial infinity on Lorentzian spaces and its role in physical models.\end{remark}

\begin{figure}[ht!]
\centering
\begin{tikzpicture}
\path[draw=black](0,0) rectangle (4,6);
\draw[draw=black] (0,0) -- (4,3);
\draw[draw=black] (0,6) -- (4,3);
\draw (2,0) node [above] {$\R\times\Sp^3$};
\draw (4,0) node [right] {$\mu$};

\draw[draw=red] (1.5,1.125) node [right] {}  circle (1.5pt);
\draw[draw=red,thick,dashed] (1.5,1.125) -- (0,1.125);
\draw[draw=red,thick] (1.5,1.125) -- (3.5,10.5/4);
\draw[draw=red] (3.5,10.5/4) node [right] {} circle (1.5pt);
\draw (2.2,1.6) node [right] {$\supp(h_-)$};

\filldraw[draw=red,fill=red] (2.4,4.2) node [right] {$q$} circle (1.5pt);
\filldraw[draw=red,fill=red] (1,5.25) node [right] {$q_1$} circle (1.5pt);
\draw[draw=red,thick,dashed] (1,5.25) -- (0,4.4);

\filldraw[draw=black,fill=black] (0,0) node [left] {$i_-$} circle (1.5pt);
\filldraw[draw=black,fill=black] (4,3) node [right] {$i_0$} circle (1.5pt);
\filldraw[draw=black,fill=black] (0,6) node [left] {$i_+$} circle (1.5pt);
\node  [waves, right, rotate=120] at (2.5,2.5*3/4) {};
\end{tikzpicture}
\hfill
\includegraphics[height=6.5cm]{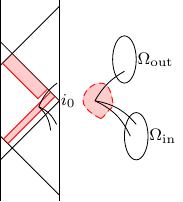}
\caption{\label{fig: sources}
{\bf Left}: Visualization of setting where scattering functionals are defined. The in-going radiation field $h_-$ is supported on a relatively compact subset of $\nullinf^-$. When $\|h_-\|<\e(t_1,{q_1})$, the solution $u$ of the scattering problem is defined in the past of the point $q_1$. The scattering functional $S_{t_1,{q}}(h_-)$ takes the value of the out-going radiation field $h_+$ evaluated at the point $q<q_1$. 
{\bf Middle:} In our spacetime sources will be produced in the nonphysical past (the lower triangular region below past null infinity). The nonlinear interaction of waves produces new waves in the physical region (shaded red), and the interactions are observed in the nonphysical future. The sources and receivers are separated by the point $i_0$. {\bf Right}: 
Schematic picture on the reconstruction. Sources located in $\Omega_\i$ produce waves that interact inside the red shaded region $D$ and cause signals that can be observed in the causally separated domain $\Omega_\out$. The closures of the domains 
 $\Omega_\i$,  $\Omega_\out$ and $D$ are disjoint. This causes difficulties that are encountered also in the figure on the figure on the left, and we overcome this issue by introducing a reconstruction algorithm which works in the situation 
 when the light-like geodesics connecting  $\Omega_\i$ to  $D$ do not have conjugate or cut points.
}
\end{figure}

\begin{figure}
    \centering
    \includegraphics[width=0.63\textwidth]{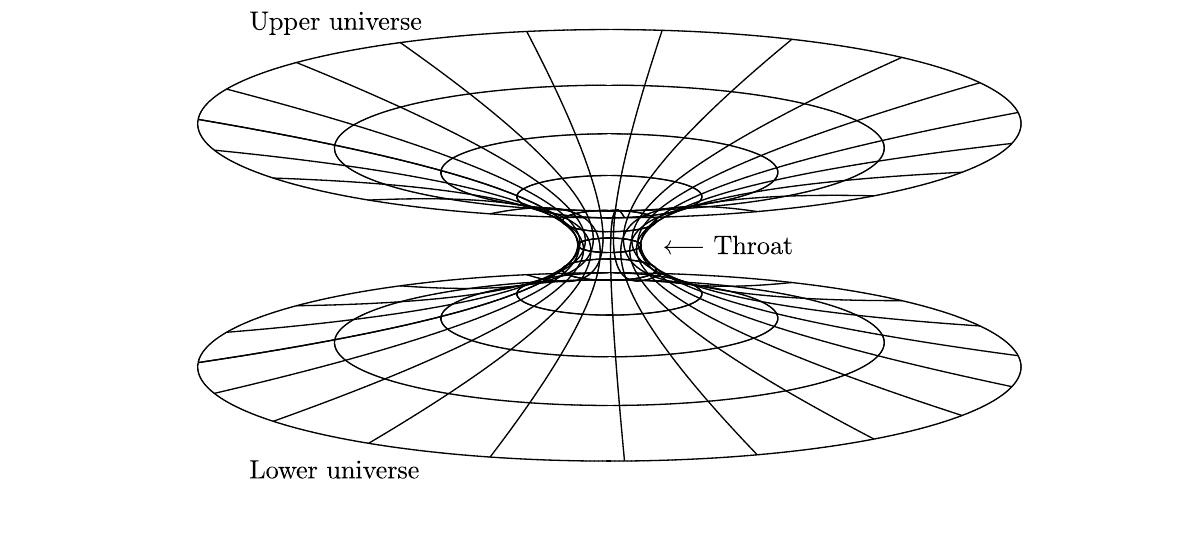}
    \hfill
    \includegraphics[width=0.34\textwidth]{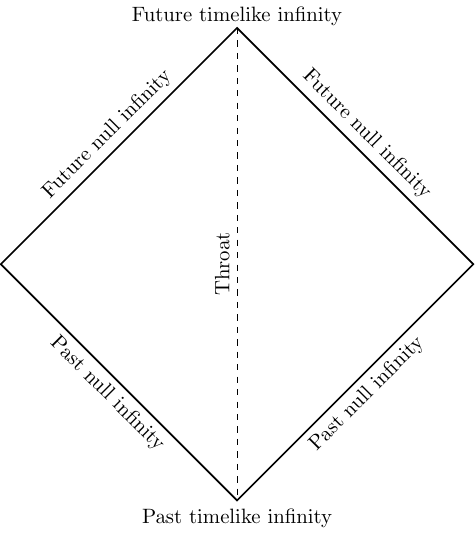}
    \caption{Left: {Morris-Thorne} wormhole manifold, see \cite{MorrisThorne,MorrisThorne1}, is a static universe (that is non-physical due to negative mass)
    of the form $M=\R\times N_0$, where $N_0$ is illustrated in the figure. Right: Penrose diagramm of a (non-physical)  traversable wormhole, see \cite{Garattini,Bambi,visser}. Note that this resembles the Penrose diagram of a Schwarzschild blackhole, though in that case the consideration of the point $i_0$ is more complicated due to singularities in the compactification, see \cite{blackhole-compactification}.} 
    \label{fig:wormhole2}   
\end{figure}

\subsubsection{Properties of the extended space-time}\label{sec:properties_for_scattering}

To define certain useful points on $ \R\times \Sp^3$,
let {$\hat \mu(s)=(s,\SP)$ be the path $\hat\mu:[-2\pi,2\pi]\to \R\times \Sp^3$ associated to the South Pole $\SP$ of the sphere $ \Sp^3$, and let  $0<s_{+2} <2\pi$ and $0>s_{-2} >-2\pi$.
 We denote $p_{+2} = \hat\mu(s_{+2}) $ and $p_{-2} =\hat \mu(s_{-2})$.
Moreover, let $s_{-2}<s_-<0<s_+<s_{+2}$, see Figure \ref{fig:penrose} (Left). We denote }
\ba
p_-=\hat\mu(s_-)\quad \hbox{and}\quad p_+=\hat \mu(s_+).
\ea

Let us next consider the extended manifold 
\beq\label{extended manifold copy}
\Next=\N\cup {{{N}}}^+\cup {{{N}}}^-.
\eeq
The following lemma is essential to the direct problem.

\begin{lemma}\label{lem: Next is globally hyperbolic}
The manifold $\Next$ is globally hyperbolic.
\end{lemma}
The proof of Lemma~\ref{lem: Next is globally hyperbolic} is postponed to later.

\subsection{Examples}\label{sec: examples of spacetimes}

We emphasize that in Definition \ref{def:asymptotically nice infinity A} we do not assume that the manifold $\M$ has a Cauchy surface $\Sigma$ for which $\Sigma\setminus E$ is compact. Thus the manifold $\M$ may have several infinities of which at least one is an asymptotically Minkowskian infinity.
\medskip

\noindent {\bf Example 1.}
 The Lorentzian product manifold $\R\times (\R^3\# \R^3)$
has an asymptotically Minkowskian infinity (in fact, it has two asymptotically Minkowskian infinities, but it suffices to define the scattering functionals by considering measurements  only one infinity), see Figure  \ref{fig:wormhole2}.
{\newnewtex Recall that all globally hyperbolic
 manifolds are homeomorphic to a product space $\R\times N_1$. Any differentiable
 manifold $N_1$ has a complete Riemannian metric, and using this metric, one can construct 
 a Lorentzian metric  on $\R\times (\R^3\# N_1)$ that is globally hyperbolic
 and has an asymptotically Minkowskian infinity that is visible from the infinity.
 Thus Definition \ref{def:asymptotically nice infinity A} allows all such manifolds where
$N_1$ has a general topology.}
 
\begin{figure}[ht!]
\centerline{
\includegraphics[height=2.5cm]{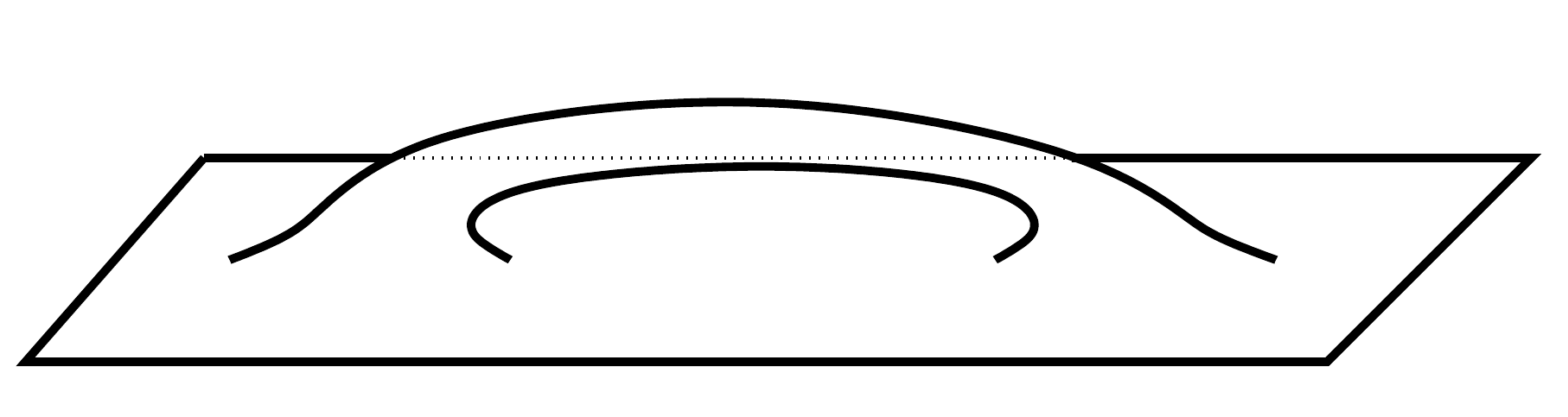}
    }
\caption{
 We can consider product space-times $M=\R\times N_0$ where $N_0$ is a 3-dimensional Riemannian manifold with non-trivial topology or several ends.  One of the ends of $N_0$ has to be asymptotically Euclidean so that the manifold $M$ has an asymptotically Minkowskian infinity. The figure visualizes a (2-dimensional) Riemannian manifold $N_0$ that is obtained by gluing a handlebody in the Eucl\-id\-ean space 
}
\label{fig:product manifold1}
\end{figure}

\medskip

\noindent {\bf Example 2.}  
Let us consider product space-times $M=\R\times N_0$ where $(N_0,g_{N_0})$ is a complete 3-dimensional Riemannian manifold with several ends of which at least one is be asymptotically Euclidean (a Schwartz class perturbation of the Euclidean metric).
Also, let $g$ on $M$ be a Lorentizan time-oriented metric that coincides with $ -dt^2+g_{N_0}$
outside a compact set then the manifold $M$ has an asymptotically Minkowskian infinity
(see Figure \ref{fig:product manifold1}). 

\medskip

\noindent 
{\bf Example 3.}
Let us consider a Lorentzian manifold {\mnewtext   $(M_0,g_{r_s})$ having a locally Schwarz\-schild event horizon, see \cite[p. 376]{ONeill}, 
and a Minkowskian infinity. 
An example of such a space is
$M_0=\R\times (\R^3\setminus \overline {B_{\R^3}(0,r_s)})$,
where $r_s>0$ is a parameter (the Schwarz\-schild radius) and we use the coordinates
$(t,r,\varphi,\theta)$, where $t\in \R$ is the Schwarzschild time coordinate and $r>r_s$ is the Schwarzschild 
radial coordinate, determined by the standard spherical coordinates $(r,\varphi,\theta)$ of $\R^3$, see \cite[p. 364-365]{ONeill}.
In these coordinates, the metric $g_{r_s}$ of  $M_0$ is given  by} 
\beq\label{eq: mod Schw.}
g_{r_s}\hspace{-1mm}=\hspace{-1mm}-\hspace{-1mm}\left(1-\phi(r){\frac {r_{\mathrm {s} }}{r}}\right)\,dt^{2}+\left(1-\phi(r){\frac {r_{\mathrm {s} }}{r}}\right)^{-1}\,dr^{2}+r^{2}\left(d\theta ^{2}+\sin ^{2}(\theta) \,d\varphi ^{2}\right)\hspace{-15mm}
\eeq
where $r_s>0$ is the Schwarzschild radius 
and $\phi\in C^\infty((r_s,\infty))$
is a {\mnewtext   function such that $0\le \phi\le 1$ and that $\phi(r)=1$ for $r<2r_s$ and $\phi(r)=0$ for $r>4r_s$.
Note that  in the region $r>4r_s$ the metric tensor $g_{r_s}$ coincides with the metric of the Minkowski space
and in the region $r\in (r_s,2r_s)$ the metric tensor $g_{r_s}$ coincides with the metric of the Schwarzschild space.
In particular, $(M_0,g_{r_s})$ contains, as an isometric subset, the region $r\in (r_s,\frac 32 r_s)$ that in
the Schwarzschild black hole is the region between the event horizon and the photon sphere $r=\frac 32 r_s$.
As stated in the introduction, this and the other examples of Lorentzian manifolds do not solve the Einstein equations for physical matter model --- they are purely examples of Lorentzian space-times, see Figure \ref{fig:black holes} (left).}

By using the fact that the exterior of the Schwarzschild black hole is a globally hyperbolic manifold,
one  sees first for the space-spherically symmetric sets $S=[-T,T]\times \p B_{\R^3}(0,r)$ that 
$J^+(S)\cap J^-(S)$ are compact and then that 
the  space-time $(M_{{0}},g_{r_s})$, where $M_{{0}}=\R\times (\R^3\setminus B_{\R^3}(0,r_s))$, is a globally hyperbolic Lorentzian manifold. Alternatively, we observe that, in the sense of \cite{Geroch}
$g_{r_s}$ is causally dominated by $f^*g_{Sc}$, denoted
 $g_{r_s}<f^*g_{Sc}$, where $g_{Sc}$ the standard Schwarzschild metric in the exterior of the event horizon in $M_0$
(given by formula \eqref{eq: mod Schw.} when $\phi$ is identically 1) and $f:M_0\to M_0$ is a map that scales the time variable by $f(t,r,\varphi,\theta)=(4t,r,\varphi,\theta)$, and 
therefore by \cite[Thm.\ 12]{Geroch}, $(M_{{0}},g_{r_s})$,
also, as the subset  $\R\times (\R^3\setminus B_{\R^3}(0,4r_s))$  is isometric to a domain of the Minkowski space, we see that  $(M_{{0}},g_{r_s})$ has an asymptotically
Minkowskian infinity, $E=\{(x^0,x')\in \R^{1+3}:\ c_0|x^0|\le |x'|,\ |x'|>r_0\}$
with suitable chosen $0<c_0<1$ and $r_0>0$. It is also easy to see that this infinity  is visible in the whole space.
We will use this in the example below.
\medskip

\begin{figure}[ht!]
\centerline{
    \includegraphics[height=6cm]{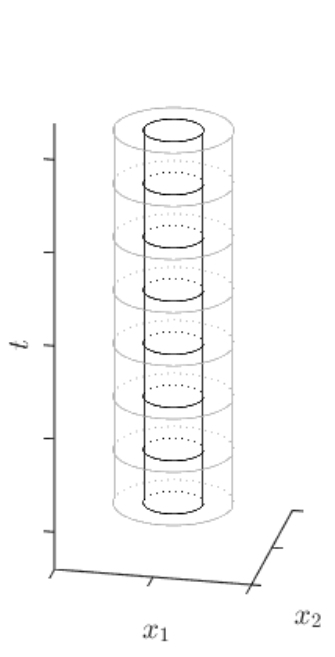}
       \hspace{1cm}
    \includegraphics[height=6cm]{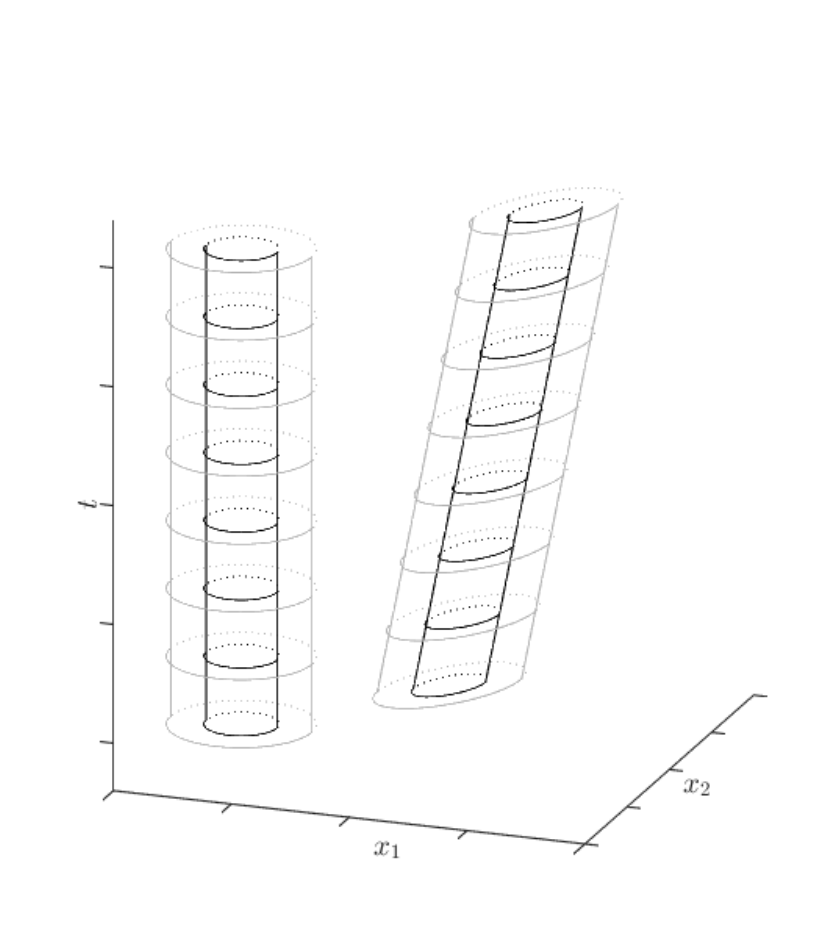}
    }

\caption{\label{fig:black holes}
{\mnewtext  
{\bf Left.} {\mnewtext  The figure shows the  $(1+2)$ dimensional spacetime $\R\times (\R^2\setminus 
\overline {B_{\R^2}(0,r_s)})$
that is analogous to the $(1+3)$ dimensional spacetime $M_0=\R\times (\R^3\setminus \overline {B_{\R^3}(0,r_s)})$,  see Example 3.
In this space-time, we consider the metric $g_{r_s}$ given in \eqref{eq: mod Schw.}. The surface $r=r_s$ is visualized as a black cylinder
and the grey cylinder is the surface $r=2r_s$. The function $\phi(r)$ in \eqref{eq: mod Schw.} is
equal to one for $r\in (r_2,2r_s)$ and thus space-time $(M_0,g_{r_s})$ contains in this region a locally Schwarzschild
event horizon. The function $\phi(r)$ vanishes for $r>4r_s$, that is, outside the grey cylinder. In this region the metric $g_{r_s}$ coincides with the metric tensor of the Minkowski space.
The coordinate axis in the figure
are $(t,x^1,x^2)$ are %
the standard
Minkowski coordinates in $r>4r_s$.
{\bf Right.} In Example 4, we consider a toy model for a space that has
 two moving black holes, and energy density which may have both positive and negative values. The figure 
 visualizes the space-time $(M_1,g_{\ell_0,\ell})$, see 
 \eqref{metric g ell ell0} and \eqref{manifold M0}, where the set $M_1\subset \R\times \R^3$
has the Lorentzian metric $g_{\ell_0,\ell}$.
In the figure, the space-time $M_1$ is visualized as the exterior of the black cylinders.  The space-time $(M_1,g_{\ell_0,\ell})$ has an (asymptotically) Minkowskian infinity and two locally Schwarzschild event horizons. 
 Note that the metric of the space-time coincide
 with the Minkowski metric outside the grey cylinders.
 }}
}
\end{figure}

\noindent
{\bf Example 4.}
{\mnewtext  Let $g_{r_s}$ be the  Lorentzian metric tensor
of the form 
\eqref{eq: mod Schw.} in $\R\times (\R^3\setminus B_{\R^3}(0,r_s))$,
where $\phi\in C^\infty([r_s,\infty))$ is such that $\phi(r)=1$ for $r_s<r<R$ and $\phi(r)=0$ for $r>2R$ where
$R=2r_s$. Moreover, let $\Lambda_{\ell}\in O(1,3)$ be a Lorentz 
transformation $\Lambda_{\ell}:\R\times \R^3\to \R\times \R^3$
that maps the line $\ell_0=\{(t,0,0,0)\in \R\times \R^3\}$ to
the affine time-like line $\ell\subset \R\times \R^3$.
For $r>0,$ we denote by $V_r$ the closed ``cylinder''
$V_r=\R\times  \overline{B_{\R^3}(0,r)}\subset \R^4$.
When the line $\ell$ is such that the sets
$\Lambda_{\ell}(V_{2R})\subset \R^{1+3}$ and $V_{2R}\subset \R^{1+3}$ do not intersect,
we define the Lorentzian metric
\beq\label{metric g ell ell0}
g_{\ell_0,\ell}:=
\begin{cases}
  (\Lambda_{\ell})_*g_{r_s},&\quad \hbox{in }
  (\R\times \R^3)\setminus (V_{2R}\cup \Lambda_{\ell}(V_{r_s})),\\
 \ \  \ g_{r_s},&\quad \hbox{in } V_{2R}\setminus V_{r_s}.
\end{cases}
\eeq
The  Lorentzian manifold $(M_1,
g_{\ell_0,\ell})$, where 
\beq\label{manifold M0}
M_1=(\R\times \R^3)\setminus (V_{r_s}
\cup \Lambda_{\ell}(V_{r_s}))\subset \R\times \R^3
\eeq
is a globally hyperbolic Lorentzian manifold with an asymptotically
Minkow\-skian infinity that is visible in the whole space.
Observe that this space-time has one Minkowskian end and the metric tensor $g_{\ell_0,\ell}$ is isometric to the metric tensor of
a Schwarzschild space-time in the domain $ V_{R}\setminus V_{r_s}$ as well as in the domain
 $\Lambda_{\ell}( V_{R}\setminus V_{r_s})$. In these two regions, the space-time has locally
Schwarzschild event horizons, see Figure \ref{fig:black holes} (Right).
If we operate with an additional Lorentz 
transformation $\Lambda_{\ell'}\in O(1,3)$ to the space-time $(M_1,
g_{\ell_0,\ell})$, we obtain   the space-time $$
( (\R\times \R^3)\setminus  (\Lambda_{\ell'}(\Lambda_{\ell}(V_{r_s}))\cup  \Lambda_{\ell'}(V_{r_s}))\, ,\, (\Lambda_{\ell'})_*g_{\ell_0,\ell})
$$
that is a toy model to a space-time with two black holes
that move along the lines $\ell'$ and $\Lambda_{\ell'}(\ell)$. 
Iterating
the above construction we obtain an asymptotically Minkowskian space-time that has several Schwarzschild (bifurcate) event horizons. 
As before, 
this space-time do not solve Einstein's field equations.}
On physically realistic space-times obtained
by gluing black-hole and other vacuum space-times, see \cite{Hintz-glueing} and references therein,
and  \cite{VasyA} on scattering from physical black holes.

On the above examples, {\mnewtext  as well on any other space-time that satisfies our assumptions}, we can consider the non-linear wave equation
\begin{equation}\label{eq:nonlinear wave equation compact perturbation}
\square_g u(x) + {\mntext d(x)u(x)}+
{\mltext a(x)}u(x)^{\kappa} = 0,\quad x\in M
\end{equation}
where $\kappa \ge  4 $ and 
$(M,g)$ is
a perturbed Lorentzian 
product space with an Euclidean infinity.
 Then the family of the scattering functionals  $S_{\M,\gM,a;t_1,{q}}$, given for all $-\pi<t_1<0$
 and $q\in \nullinf^+$ 
determines the manifold $\M$, and the {\mltext conformal class of the} metric $\gM$ on $\M$,
and the pair $(\gM,A)$ up to a conformal transformation.

\subsection{Reduction of the scattering measurements to the near field measurements in the extended space-time}\label{sec:ideas_for_scattering}

Next, we will consider the equation
\begin{equation}\label{eq:nonlinear wave equation in Penrose 2}
\begin{cases}
&(\square_{\gext}+B)w + Aw^{\kappa}=f,\quad
\text{in } x\in I^-_{\Next}(p),\\
&\supp( w)\subset J^+_{\Next}(\supp(f))
\end{cases}
\end{equation}
where $p\in \Next.$ To make  our conformal transformations below
possible,
we will allow $B$ to be a general smooth function. In particular,
we will consider the case when $B=B_g+{\mntext D}$, {\mntext and $B_g$ and $D$ are defined in \eqref{ABD functions}.}

\begin{lemma}\label{Lemma L-maps} Let  ${\kappa}\ge 1$ and $k\ge 4$. %
Moreover, let $A,B\in C^\infty(\Next)$,  $K\subset \Next$ be a compact set and $p\in \Next.$
Then there is $\eps>0$ such that for all 
$f\in H^k_0(K)$, $\norm{f}_{H^k(\Next)} < \eps$
there exists a unique (small) solution $w\in H^{k+1}(I^-_{\Next}(p))$
  to  the equation \eqref{eq:nonlinear wave equation in Penrose 2}.
Moreover, the solution $w$ depends in $ H^{k+1}(I^-_{\Next}(p))$ continuously on the source $f
\in H^k_0(K)$ and
\[
\Vert w \Vert_{H^{k+1}(I^-_{\Next}(p))} \leq C \Vert f \Vert_{H^k(\Next)},
\]
where $C$ is independent of $f$.
\end{lemma}

The proof of
Lemma~\ref{Lemma L-maps}  can be obtained using the proof of the 
results in  \cite[Prop.\ 9.12 and 9.17]{Ringstrom}, by using the energy estimates for the wave equation 
\eqref{eq:nonlinear wave equation in Penrose 2}.
Alternatively, see
  \cite [Thm.\ III]{HKM},  \cite{Kato1975}, or 
\cite[App.\ III, Thm.\ 3.7 and 4.2]{ChBook}.

Let $K_n\subset ( {{N}} ^-)^{\mathrm{int}}$, $n=1,2,\dots$ be compact sets that care closures of open sets such that 
$K_n\subset K_{n+1}$ and $\bigcup_{n=1}^\infty K_n=
 ( {{N}} ^-)^{\mathrm{int}}$. For example, we can choose 
 \beq\label{Kn sets}
 K_n=J^+((-\pi+\frac 1n,\SP))\cap J^-(-\frac 1n,\SP)).
 \eeq

Let $ \mathcal V_n\subset H^k_0( 
K_n)$ be a sufficiently small neighborhood of the zero function.
Then the source-to-solution map 
\beq\label{L maps defined}
L_{g_{\ext},B,A,p_+,K_n}: \mathcal V_n\subset H^k_0(K_n) \to
L^2(I^-_\Next(p_+)\cap {{N}} ^+)
\eeq
is well defined by setting $L_{g_{\ext},B,A,p_+,K_n}(f) = u|_{I^-_\Next(p_+)\cap {{N}} ^+}$, where $u$ solves \eqref{eq:nonlinear wave equation in Penrose 2}.

\commented{
\begin{remark}
Below, we consider also the extended source-to-solution operator $L_{g_{\ext,2},p_+}$
that is defined as the union of  the domains $ \mathcal V_n$, $n=1,2,\dots$,
\beq
 \mathcal V=\bigcup_{n=1}^\infty  \mathcal V_n \subset H^k_0({{N}} ^-) .
\eeq
For $f\in  \mathcal V$ we define
\beq\label{def: extended source-to-solution operator}
L_{g_{\ext},B,A,p_+}(f)=L_{g_{\ext},B,A,p_+,K_n}(f),\quad \hbox{when }f\in  \mathcal V_n.
\eeq

\end{remark}
}

The following theorem shows that the scattering functionals determine the source-to-solution maps.

\begin{theorem}\label{t:scattering_operator_determines V}
Let $(M_1,g_{M_1})$ and  $(M_2,g_{M_2})$ be globally
hyperbolic Lorentzian manifolds 
 with an asymptotically Minkowskian infinity that are visible in the whole space-time. Let
  $(N_1,g_{N_1})$ and $(N_2,g_{N_2})$ be  conformal manifolds given in Definition \ref{def:asymptotically nice infinity A} and 
$(N_{\ext}^1,g_{1})$ and $(N_{\ext}^2,g_{2})$ be the corresponding extended manifolds.
Let $a_j,d_j\in C^\infty(M_j)$, $j=1,2$ be Schwartz
functions in an asymptotically Minkowskian infinity of $M_j$, 
$a_j(x)\not=0$ for all $x\in M_j$, and ${\kappa}\ge 4$.

Let $S_{\M_j,{\gM}_j,a_j,t_1,{q}}$ be the scattering functionals related to $(M_j,g_{M_j})$ and coefficients $a_j$.
If the  scattering  functions on 
 $(M_j,g_{M_j})$ satisfy
$$
S_{\M_1,{\gM}_1,a_1,{\mntext d_1};t_1,{q}}(h)=S_{\M_2,{\gM}_2,a_2,{\mntext d_2};t_1,{q}}(h),
$$
for all $-\pi<t_1<0$, $q\in \nullinf^+$ 
and $ h\in \mathcal D(S_{\M_1,{\gM}_1,a_1,t_1,{q}})\cap \mathcal D(S_{\M_2,{\gM}_2,a_2,t_1,{q}})$,
then the corresponding source-to-solution maps
on  $(N^j_{ext},g_{j})$ {\mnewtext  (see formula \eqref{L maps defined})}   satisfy for all $p_+\in \hat \mu(0,\pi)$ and $n$
$$
L_{g_{1},B_1,A_1,p_+,K_n}(f) = L_{g_2,B_2,A_2,p_+,K_n}(f)
$$
 when $f$ is
in some neighborhood $\mathcal U_n$ of the zero function in $H^k_0(K_n)$ and  {\mntext $B_j=B_{g_j}+D_j$
and  $B_{g_j}$ and $D_j$ are defined in \eqref{ABD functions}.}
\end{theorem}

\subsection{Structure of the proofs of the main theorems}

{\mnewtext
To show that the scattering functionals
are well defined in an asymptotically Minkowskian space-time, we reduce the
the scattering problem to a nonlinear Cauchy-Goursat problem in a suitable subset 
$W$ (see \eqref{eq:domain_for_wave_eq} and Fig.\ \ref{fig:penrose}) of the extended space-time $N_{\ext}$.
The boundary of $W$ is only Lipschitz-smooth near the space-like infinity $i_0$ and in 
Appendix A we obtain regularity estimates in the the higher order Sobolev spaces  for 
 the  Cauchy-Goursat problem in the space-time $W$. Using these estimates  
and the strong Huygens' principle on $\R\times \mathbb S^3$, we show that the
scattering functionals are well defined, and prove
Theorem \ref{t:scattering_operator_determines V}. This result implies that one can reduce the inverse scattering problem to an inverse problem for local
measurements in the extended space-time. In this problem with local measurements, a source-to-solution
operator is given in the case when the source and the observation sets are different and separated.
Earlier, the  source-to-solution map with different source and observation sets
has been studied in \cite{FLO} in the situation when 
\beq\label{causally connected2}
M^{\i,\out}:= \bigg(\bigcap_{p\in \Omega_{\i}}I^+(p)\bigg) \cap  \bigg(\bigcap_{q\in\Omega_{\out}} I^-(q) \bigg)=\emptyset.
\eeq

In the inverse problem on the extended Penrose {\mnewtext compactification} $\Next$ 
the sources are supported in the set $\tilde \Omega_{\i}=({{N}} ^-)^{\mathrm{int}}$ 
and 
solutions are observed in the set $\tilde \Omega_{\out}=({{N}} ^+)^{\mathrm{int}}$. On the space-time $\Next$
there holds for the space-like infinity
$i_0$ that (see Figure \ref{fig: sources}, Middle)
\beq\label{causally disconnected}
 i_0\in 
 N_{\ext}^{\i,\out}:= \bigg(\bigcap_{p\in \tilde \Omega_{\i}}I^+(p)\bigg) \cap \bigg(\bigcap_{q\in\tilde \Omega_{\out}} I^-(q)\bigg)
\eeq
that is, 
the sets $\tilde \Omega_{\i}=({{N}} ^-)^{\mathrm{int}}$ and  $\tilde \Omega_{\out}=({{N}} ^+)^{\mathrm{int}}$ are causally separated by $i_0$. This implies that the Lorentzian time separation function $\tau:\Next\times \Next\to \R$
satisfies $\tau(p,q)>0$ for all $p\in \tilde \Omega_{\i}$ and $q\in \tilde \Omega_{\out}$ and $\tilde \Omega_{\i}$ and thus $\tilde \Omega_{\out}$ can not be connected by light-like geodesics having no conjugate or cut points. This prevents  to use of
layer-wise constructions (e.g., the layer-stripping) to solve the inverse problem.
We deal with this difficulty by first constructing open subsets of space-time that are not connected  to the source and observation
sets, but which can be connected to the source set with light-like geodesics that have no cut or conjugate points, see  Figure \ref{fig: sources} (Middle and Right).
Second, we prove that the source-to-solution map in the original source and observation domains 
$N^-$ and $N^+$ and the open sets reconstructed above determine two new local
 source-to-solution maps in the neighborhoods of the sets $ \nullinf^+\cup i_0$ and  $\nullinf^-\cup i_0$. That is, we show that by using the source-to-solution map in the non-physical part of the extended space-time $N_{\ext}$, one can 
 reproduce the results of the measurements where both the sources and the observations
would be located the physical part $N$ of the space-time.
Theorem \ref{main thm for general manifold} is then proven
by combining the original and the reproduced measurements in the extended space-time and using the multiple linearization results for the non-linear wave equation. 
}

{\newnewtext \subsection{Generalizations of the inverse scattering problem for FLRW space-times} 
\label{sec: generalizations}

Next we will consider  space-times 
similar to the {\it Fried\-mann-Lema\^itre-Robertson-Walker (FLRW)}
space-times studied in cosmology. 
It is possible to apply our constructions to perturbations of such space-times, given in \eqref{background_metric} below, when the 
condition \eqref{no particle horizons} holds. 
We note that such cosmological backgrounds are conformally related to the whole Minkowski space-time and thus no particle horizons appear, see 
 \cite[p.\ 105 and Fig. 5.6]{Wa1984} -- this has sometimes been suggested as a solution to the particle horizon problem in big-bang cosmology,
 see  \cite[Sec.\ 4.2-4.3]{Inflation-course-Kinney}. 

An appealing feature of this construction  is that then the sources that generate waves could lie near  the cosmological singularity (i.e., the big-bang) in the $M_{FLRW}$ space-time.

While the conformal structure of the space-times we consider here  is the same as the Minkowski space, it should be said that 
our notion of radiation field is an extrapolation.

To define the scattering problem in FLRW space-times, we consider a  class of manifolds with a more general conformal factor close to the infinities. Let $V\subset\R^{1+3}$ and $g_V$ be a   Lorentzian metric  on $V$.
 We say that $(V,g_V)$ is a neighborhood of  
the 
 light-like infinity in $\R^{1+3}$ with a conformally asymptotically Minkowskian metric with conformal factor ${\tilde \Omega}\in C^\infty(\R^{1+3})$
 if the conditions in {Definition} \ref{Def: visible infinity} are valid when $\Omega$ is replaced by ${\tilde \Omega}|_V$. 
  Moreover,  we say that a manifold $(\M,\gM)$ has a conformally asymptotically Minkowskian infinity $E$ 
  (with a conformal factor ${\tilde \Omega}$) that is visible in the whole space-time $\M$ if
 \begin{itemize}
  \item[(i)] $(\M,\gM)$ is a globally hyperbolic manifold and it has a subset 
 $E\subset \M$ such that $(E,\gM|_E)$
 is isometric to 
  a neighborhood $(V,g_V)$ of  
the 
 light-like infinity in $\R^{1+3}$ with  a conformally asymptotically Minkowskian metric with conformal factor ${\tilde \Omega}$,

  \item[(ii)]  $J^+_\M(E)=J^-_\M(E)=\M$.

 \end{itemize}
 Our results for the inverse scattering problem generalize  for the manifolds  with a conformally asymptotically Minkowskian infinity $E$ 
 with a priori given conformal factor  ${\tilde \Omega}$.  
  In such manifold, let us consider the equation for the conformal wave operator  
  \begin{equation}
      \label{modified equation with Q}
  \square_gu  - \frac 16 R_{g}u+d\hspace{.2mm}u+au^\kappa=0,\quad\hbox{on }M,
  \end{equation}
  where $d=d(x)$ and $a=a(x)$ are scalar functions.
 Using a conformal transformation, we see that the equation \eqref{modified equation with Q}  is  equivalent to the wave equation 
  \begin{equation}\label{conformally Minkowskian wave equation}
   \square_{g_0}u_0 - \frac 16 R_{g_0}u+d_0u_0+a_0 u_0^\kappa=0,\quad\hbox{on }M,
   \end{equation}
 where  $u_0=(\Omega/{\tilde \Omega})^{-1}u$ and 
\begin{equation}
\begin{split}&\qquad g_0=(\Omega/{\tilde \Omega})^2g,\quad  d_0=(\Omega/{\tilde \Omega})^{-2}d,
\quad a_0=a\cdot (\Omega/{\tilde \Omega})^{{\kappa}-3} .
\end{split}\label{Q0 and a0} 
    \end{equation}
    Here, $(M,g_0)$ is a Lorentzian manifold that has an asymptotically Minkowskian infinity $E$ 
  that is visible in the whole space-time $\M$.
Below, we assume that $d_0(\Phi(x))$ and $a_0(\Phi(x))$ are Schwartz functions (that is, those vanish up to infinite order near 
${\mathcal I}^-\cup {\mathcal I}^+\cup i_0$).

For the equation \eqref{conformally Minkowskian wave equation}, we define the scattering functionals
as in the
Definition} \ref {def: Scattering functionals},
where the equation 
\eqref{general non-linear problem on M} is replaced by the equation \eqref{conformally Minkowskian wave equation} having the potential term $d_0$.
That is,
$\tilde S_{\M^{(1)},\gM^{(1)},a^{(1)},d^{(1)};t_1,{q}}:h_-\to h_+(q)$ is as in the
Definition} \ref {def: Scattering functionals},
when the equation 
\eqref{general non-linear problem on M} is replaced by the equation \eqref{conformally Minkowskian wave equation}.
That is, $\tilde S_{\M^{(1)},\gM^{(1)},a^{(1)},d^{(1)};t_1,{q}}(h_-)=h_+(q)$, where $h_-$ and $h_+$ 
are the generalized past and future radiation fields, respectively, that are given by 
\beq\label{future rad field formula gen}
& &\lim_{x\to p} {\tilde \Omega}(\Phi^{-1}(x))^{-1} u(\Phi^{-1}(x))= h_-(p)\quad \hbox{for }p\in  {\mathcal I}^-,\\
& &\lim_{x\to q} {\tilde \Omega}(\Phi^{-1}(x))^{-1} u(\Phi^{-1}(x))= h_+(q)\quad \hbox{for }q\in  {\mathcal I}^+,
\eeq
where $x\in \Phi(V)\cup ({\mathcal I}^-\cup {\mathcal I}^+)$, c.f.\ \eqref{future rad field formula} and $u$
satisfies \eqref{modified equation with Q}.
 In this setting Theorem~\ref{main thm for general manifold} generalizes to the following result where 
 the conformal factor is ${\tilde \Omega}$ and we are given only a restricted  scattering data.
 
\begin{theorem}\label{main thm for general manifold generalized}
Let $(\M^{(j)},g^{(j)})$, $j=1,2$ be two globally hyperbolic manifolds with conformally asymptotically Minkowskian infinities with
conformal factor ${\tilde \Omega}$, that are visible in the whole space-time $\M^{(j)}$.
Let $a^{(j)},d^{(j)}\in C^\infty(\M^{(j)})$, $j=1,2$ be such that the corresponding functions
$a_0^{(j)}(\Phi(x))$ and $d_0^{(j)}(\Phi(x))$ defined by formula \eqref{Q0 and a0}  are Schwartz functions
and $a_0^{(j)}$
are strictly positive.
Moreover, let $-\pi\le t^*_1<0$.
 Assume that   for all $t^*_1<t_1<0$ and $q\in \nullinf^+$  the scattering functionals for
 the equations \eqref{modified equation with Q} satisfy
$$
\tilde S_{\M^{(1)},\gM^{(1)},a^{(1)},d^{(1)};t_1,{q}}(h)
=\tilde S_{\M^{(2)},\gM^{(2)},a^{(2)},d^{(2)};t_1,{q}}(h)
$$
when  $ h\in \mathcal D(\tilde S_{\M^{(1)},\gM^{(1)},a^{(1)},d^{(1)};t_1,{q}})\cap \mathcal D(\tilde S_{\M^{(2)},\gM^{(2)},a^{(2)},d^{(2)};t_1,{q}})$.
Then there is a diffeomorphism $\Psi:I_{\M^{(1)}}^+(\nullinf^-\cap \mathcal B_-(R(t_1^*)))\to I_{\M^{(2)}}^+(\nullinf^-\cap \mathcal B_-(R(t_1^*)))$ and 
a function $\gamma\in C^\infty(\M^{(1)})$ such that the metric tensors $g^{(j)}$
and the coefficients $a^{(j)}$ of the non-linear terms 
satisfy  
\begin{equation}\label{eqtrans}
\begin{split}
& g^{(1)} = e^{2\gamma}\Psi^*g^{(2)},\\
& a^{(1)}=  e^{(\kappa-3)\gamma(x)}\Psi^*a^{(2)},
\end{split}
\end{equation}
in the domain $I_{\M^{(1)}}^+(\nullinf^-\cap \mathcal B_-(R(t_1^*)))$,
that is, the non-linear scattering functionals uniquely determine the topology, the differentiable structure, and the conformal type of the Lorentzian manifold,
and the Lorentzian metric, and the coefficient function of the non-linear term up to the transformations in \eqref{eqtrans}. 
\end{theorem}
To apply the above theorem we will recall the definition of the 
Friedmann-Lema\^itre-Robertson-Walker (FLRW) space-times.
These space-times are  Lorenzian manifolds 
$( M_\mathrm{FLRW},g_\sigma)$ of the following warped product form
$M_\mathrm{FLRW}=(0, \infty) \times \R^3$ with the metric
\begin{align}
\label{background_metric} 
g_\sigma(t,x) = -dt^2 + \sigma^2(t)((dx^1)^2+(dx^2)^2+(dx^3)^2), 
\end{align}
where  $(t,x^1,x^2,x^3) \in 
 (0, \infty) \times \R^3,$ $\sigma(t) > 0$ is smooth on $(0,\infty)$ and tends to zero as $t \to 0$,
 possibly causing a singularity near $t=0$.
A singularity at the boundary $t = 0$ 
corresponds physically to the Big Bang and the behaviour of the metric at large $t$ corresponds to the expansion of the Universe, see \cite{Wa1984,SW}. 
In particular, the choice $\sigma(t) = t^{2/3}$ gives the Einstein-de~Sitter cosmological model \cite[p. 31]{SW}. 
Let us consider next
a model where  
\begin{align}\label{no particle horizons}
\hbox{$\int_1^\infty \frac 1{\sigma(t)}dt=\infty$ \quad and \quad $\sigma(t)<c_1t$ for $0<t<1$,}
\end{align} 
where $c_1>0$.
 In this case
 there are no particle horizons, that is, all observes 
on the co-moving curves $x=\hbox{constant}$ can see at a given time all other co-moving observes, see
 \cite[p.\ 105 and Fig. 5.6]{Wa1984}. The condition \eqref{no particle horizons} is closely
 related to cosmological models with inflation, see \cite[Sec.\ 4.3]{Inflation-course-Kinney}.
 
 We note that the Einstein tensor of several FLRW metrics satisfy the null energy condition (that corresponds to non-negative energy density in gereral relativity), see \cite{CW}, and thus by using perturbations of the FLRW space-times we can consider inverse problems for manifolds satisfying  the null energy condition. 
\medskip

\noindent
 {\bf Example 5.}
Let
$( M_\mathrm{FLRW},g_\sigma)$ be a space time the form (\ref{background_metric})
that satisfies the condition \eqref{no particle horizons}.
Under the assumptions, $( M_\mathrm{FLRW},g_\sigma)$ 
is conformal to the Minkowski space in suitable coordinates. 
Indeed, let us define the conformal time, that is the strictly increasing function 
\begin{align}
\label{conformal_time}
\tau(t) := \int_1^t \frac 1{\sigma(t')} dt',\quad t\in (0,\infty).
\end{align}
Observe that $\tau(t)\to -\infty$ as $t\to 0.$
Let 
\begin{equation}\label{FLRW transformation}
    F:\R_+\times \R^3\to \R\times \R^3, \quad F(t,x)=(\tau(t),x)
\end{equation}
Then, $F$ is an isometric diffeomorphism  from the space $( M_\mathrm{FLRW},g_\sigma)$ 
 to $(\R \times \R^3,g_{\tilde \sigma})$, 
\begin{align}\label{conformal Minkowski space A}
g_{\tilde \sigma}(\tau, x) = \tilde\sigma^2(\tau) (-d\tau^2 + (dx^1)^2+(dx^2)^2+(dx^3)^2),
\end{align}
where $(\tau,x^1,x^2,x^3) \in \R \times \R^3$
and  $\tilde\sigma(\tau(t))=\sigma(t)$. 

Next, let us consider the space-time $\R_+\times \R^3$ with a metric $g$, 
that is a perturbation of a FLRW space-time $(\R_+\times \R^3,g_\sigma)$ in
\eqref{background_metric}. We assume that $\sigma(t)$ is such that there are no particle horizons, that is, the conditions given below in \eqref{no particle horizons}
  are valid. Let $F:\R_+\times \R^3\to \R\times \R^3
$ be the map given  in \eqref{conformal_time} and  \eqref{FLRW transformation}
and assume that $F_*(g-g_\sigma)$ is a Schwartz class function in $\R\times \R^3$.
Let $\tilde g=F_*g$. Then,  $(\R\times \R^3,\tilde g)$ is
a globally hyperbolic manifold with a conformally asymptotically Minkowskian infinity with
conformal factor ${\tilde \Omega}=\tilde\sigma(\tau)\Omega$, that is visible in the whole space-time,
and Theorem~\ref{main thm for general manifold generalized} can be used to reconstruct
the metric $g$ in the future of the domain $\nullinf^-\cap \mathcal B_-(R(t_1^*)))$, where
the ingoing radiation fields are supported.
Here, the radiation fields $h_+$ and $h_-$ are given by
\beq\label{future rad field formula gen 2}
& &\lim_{x\to p} {\tilde \Omega}(F^{-1}(\Phi^{-1}(x)))^{-1} u(F^{-1}(\Phi^{-1}(x)))= h_\pm (p)\quad \hbox{for }p\in  {\mathcal I}^\pm ,
\eeq
where $u$ is the solution of the equation
\eqref{modified equation with Q} on 
$(\R_+\times \R^3,g).$

{\bf Acknowledgements.}
{H.~I. was supported by 
     Grant-in-Aid for Scientific Research (C) 24K06768 Japan Society for the Promotion of Science.}
  M.~L. was partially supported by a AdG project 101097198 of the European Research Council, Centre of Excellence of Research Council of Finland and the FAME flagship of the Research Council of Finland (grant 359186).
   S.~A. was supported by the Natural Sciences and Engineering Research Council of Canada (NSERC): RGPIN-2020-01113 and RGPIN-2019-06946.
 T.~T was supported by the FAME flagship and the Emil Aaltonen foundation.
 Views and opinions expressed are those of the authors only and do not necessarily reflect those of the funding agencies or the  EU.

\section{Geometric properties and notations}

\subsubsection{Preliminary notations on Lorentzian manifolds}\label{sec: Basic notations}
We recall some notations and definitions for Lorentzian manifolds. A more extensive treatment of Lorentzian geometry can be found in \cite{ONeill}. 
Let $(\M,\gM)$ be an $(n+1)$-dimensional Lorentzian manifold with metric of signature $(-,+,+,\dots,+)$ where $n=3$.
We assume that $\M$ is time-oriented.
Then a smooth path $\mu:(a,b)\to \M$ is  time-like if $\gM(\dot\mu(s),\dot\mu(s))<0$ for all $s \in (a,b)$.
The path $\mu$ is causal, if $\gM(\dot\mu(s),\dot\mu(s))\leq 0$ and $\dot\mu(s)\neq 0$ for all $s\in (a,b)$.
For $p,q\in \M$ we denote $p\ll q$ if $p\neq q$ and there is a future-pointing time-like path from $p$ to $q$.
Similarly, $p<q$ if $p\neq q$ and there is a future-pointing causal path from $p$ to $q$, and $p\leq q$ when $p=q$ or $p<q$.
The chronological future of $p\in \M$ is the set $\I^+(p)=\{q\in \M\mid p\ll q\}$ and the causal  future of $p$ is $J^+(p)=\{ q\in \M\mid p\leq q\}$.
The chronological past $\I^-(q)$ and causal past $J^-(q)$ of $q\in \M$ are defined similarly.
If $A\subset \M$, then we denote $J^\pm(A)=\cup_{p\in A}J^\pm(p)$.
The diamond sets are denoted by $J(p,q)= J^+(p)\cap J^-(q)$ and $\I(p,q)=\I^+(p)\cap\I^-(q)$.

A time-orientable Lorentzian manifold $(\M,\gM)$ is globally hyperbolic if there are no closed causal paths in $\M$ and for $q_1,q_2\in \M$ with $q_1<q_2$ the diamond $J(q_1,q_2)\subset \M$ is compact \cite{Bernal}.
In a globally hyperbolic manifold the sets $J^\pm(p)$ are closed and $\mathrm{cl}(\I^\pm(p))=J^\pm(p)$.

Let $L_p\M=\{\xi\in T_p\M\setminus\{0\}\mid \gM(\xi,\xi)=0\}$ be the set of light-like vectors in the tangent space $T_p\M$. This set is called the light-cone of $p$.
Let also $L_p^+\M$ and $L_p^-\M$ denote the future and past light-like vectors in $T_p\M$
and  $L_p^{*,+}\M$ and $L_p^{*,-}\M$ denote the future and past light-like covectors in $T^*_p\M$,
and denote $L_p^{*}\M=L_p^{*,+}\M\cup L_p^{*,-}\M$.
Let $\mathcal L^+(q)=\exp_q(\overline {L^+_q\M})$.

We use also some other notations:
For $\xi\in T_x^*\M$ we will denote $(\xi^\sharp)^j = \gM^{jk}\xi_k$ and for $\zeta\in T_x\M$ we let $(\zeta^\flat)_j = g_{jk}\zeta^k$.
Here $g^{jk}$ is the inverse matrix of $g_{jk}$.

The time-separation function $\tau(x,y)$ of $x,y\in \M$, $x<y$ is defined as the supremum of the lengths $L(\alpha) = \int_0^1 \sqrt{-g(\dot\alpha(s),\dot\alpha(s))}\d s$ of the piece-wise smooth causal paths $\alpha :[0,1] \to \M$ from $x$ to $y$.
If $x< y$ is not valid,  we set $\tau(x,y)=0$.
If $(x,\xi)$ is a non-zero vector, the  number $\T(x,\xi)\in (0,\infty]$ will be used to denote the maximum time for which the geodesic $\gamma_{x,\xi}:[0,\T(x,\xi)) \to \M$ is defined.

For a light-like vector $(x,\xi)\in L^+\M$ we define the cut-locus function
\begin{equation}\label{eq:cut-locus function}
\rho(x,\xi) = \sup\{ s\in [0,\T(x,\xi))\mid \tau(x,\gamma_{x,\xi}(s))=0 \}.
\end{equation}

\begin{definition}
The light observation set from the point $q$ in $V\subset \M$ is
$$
\P_V(q) = {\mtext \mathcal L}^+(q)\cap V,
$$
and the earliest light observation set from the point $q$ is
\begin{equation}\label{eq:earliest light observation sets}
\begin{split}
\E_V(q) = \{ x\in \P_V(q) &\mid \text{ there are no } y\in \P_V(q) \text{ and }\\
    &\text{future-pointing time-like path }\alpha:[0,1]\to V\\
    &\text{such that } \alpha(0) = y \text{ and } \alpha(1) = x \}.
\end{split}
\end{equation}
Let $W \subset \M$ be open.
The family of the earliest light observation sets with source points in $W$ is 
\begin{equation}\label{eq:earliest light observations family}
\E_V(W) = \{ \E_V(q)\subset V\mid q\in W\}
\end{equation}
When we want to emphasize the manifold $\M$ on which we consider the
earliest light observation sets, we use notation
$\E_V(W) =\E_{V,\M}(W).$
\end{definition}

The earliest light observation set are discussed in detail in \cite{KLU2018}.
The earliest light observation set $\E_V(q)$ can be written also as
\beq
\E_V(q)=\bigg(\bigcup_{\xi\in L^+_q\M} \gamma_{q,\xi}([0,\rho(x,\xi)])\bigg)\cap V,
\eeq
that is $\E_V(q)$ is the intersection of the observation set $V$ and the union of all future-directed light-like geodesic segments 
$\gamma_{q,\xi}([0,\rho(x,\xi)])$ from $q$ to their first cut or conjugate point $\gamma_{q,\xi}(\rho(x,\xi))$.
In the case when a point source is located in the point $q$ of the space-time,  the earliest light observation set $\E_V(q)$ 
is the set when the light (or other signal) arriving to $V$ are detected first time, i.e. the signal is  detected at $p$ but
not in the chronological past of $p$.

\subsubsection{Properties of the extended manifold $\Next$}

Next we prove Lemma~\ref{lem: Next is globally hyperbolic}.

\begin{proof}  (of Lemma~\ref{lem: Next is globally hyperbolic}). 
{\mnewtext  Let us start with the simple cases when $(M,g_M)$ is either the Minkowski space-time with the standard metric or the space-time
 $\R^{1+3}$ where we have added a Schwartz class perturbation to the metric of the Minkowski space.
In these cases the manifold $(N,g_N)=(M,\omega_M^2g_M)$ is the subset $I^+_{\R\times \mathbb S^3}(i_-)\cap I^-_{\R\times \mathbb S^3}(i_+)\subset \R\times \mathbb S^3$ with a metric tensor that 
can be $C^\infty$-smoothly extended to the space $\R\times \mathbb S^3$ so that the extended metric coincides
with the standard metric of $\R\times \mathbb S^3$ outside the set $N$. In this case use the
whole space $N$ as the neighborhood $E$ of $\nullinf^+\cup \nullinf^-\cup i_0$ and obtain the extended space-time $\Next$
by gluing $N$ with $N^-$ and $N^+$ smoothly using the standard coordinates of $\R\times \mathbb S^3$.
This makes $\Next$ a globally hyperbolic space-time.

Next, we consider the more general space-times $N$ with (possibly) non-trivial topology.}

Let $\Sigma\subset \N$ be a Cauchy surface of $\N$.
We recall that a subset  $\Sigma\subset \N$ of a time-oriented Lorentzian manifold is a Cauchy surface if the intersection of every inextendable timelike curve in $\N$ and the set  $\Sigma$ is one point,  see \cite[Ch.\ 14, Definition 28]{ONeill}. We recall that a 
 timelike curve $\mu:I_0=(a,b)\to \Next$ is future (or past) inextendable if there is an increasing (or decreasing) sequence $s_j,$ $j\in \Z_+$ such that $\mu(s_j)$ does not converge as $j\to \infty$, and that a path is  inextendable if it is both past and future  inextendable.
Observe that the definition of a Cauchy surface does not
require that $\Sigma$ is a smooth submanifold.
By \cite[Ch.\ 14, Corollary 39]{ONeill}, the time-oriented Lorentzian manifold $\N$  is globally hyperbolic if and only if it has a Cauchy surface.
Let us next show that  $ \Sigma_{\ext}=\Sigma\cup i_0$ is a Cauchy surface of $\Next$.

Let $\mu:I_0\to \Next$ be an inextendable causal curve in $\Next$, $I_0\subset \R$.
If $\mu\cap \N \not=\emptyset$, we see that $\mu\cap \N $ is an inextendable causal path in  $\N $ and hence 
$\mu$ intersects $ \Sigma$ once.
Hence, if $\mu\cap \N \not=\emptyset$, then $\mu$ has to intersect $ \Sigma\cup i_0$.

Next, assume that $\mu\cap \N =\emptyset$. To show that 
and $\mu$ contains the point $i_0$, we assume to the contrary that 
$i_0\not \in \mu.$ As the path $\mu$ is a connected
set but $N_{\ext}\setminus (N\cup i_0)={N}^+\cup {{{N}}}^-$
is disconnected, there holds $\mu\subset {N}^+$ or  $\mu\subset {{{N}}}^-$.
Suppose that the former is valid, that is,  $\mu\subset {{{N}}}^+$.
Because $\mu$ is past inextendable, there is a decreasing sequence $r_j\in I_0$, $j=1,2,\dots$
such that $\mu(r_j)$ does not converge as $j\to \infty$.
However, as $\mu$ is causal  $\mu(r_j)\subset J^-_{\Next}(x_1)$ for all $j$,
where $x_1=\mu(r_j)$. As ${N}^+\subset  J^+_{\Next}(i_0)$
and the set $ J^-_{\Next}(x_1)\cap  J^+_{\Next}(i_0)$ is compact,
we see that there is a subsequence $x_k=\mu(r_{j_k})$ such that  $\mu(r_{j_k})$ converges to $p\in 
J^-_{\Next}(x_1)\cap  J^+_{\Next}(i_0)\subset {{{N}}}^+$. 
As ${{{N}}}^+\subset \R\times \Sp^3$ and $x_k\to p$ as $k\to \infty$,
we see that $x_k\in J^+_{\Next}(p)$ and for $j'\ge j_k$ we have 
$$
\mu(r_{j'})\subset 
J^-_{\Next}(x_k)\cap  J^+_{\Next}(x_{k+1})\subset J^-_{\Next}(x_k)\cap  J^+_{\Next}(p).
$$
As the sets $J^-_{\Next}(x_k)\cap  J^+_{\Next}(p)$ converge to the singleton $\{p\}$
as $k\to \infty$ (that is, for any neighborhood $U\subset \R\times \Sp^3$ of $p$ there is $k$ such that
$J^-_{\Next}(x_k)\cap  J^+_{\Next}(p)\subset U$), we see that
the sequence $\mu(r_j)$ converges to the limit point $p\in {{{N}}}^+$ as $j\to \infty$.
As this is in contradiction with the assumption that  $\mu$ is
 an inextendable causal path, we have shown that $\mu$ contains the point $i_0$.
 Thus, we have seen that also when $\mu\cap \N =\emptyset$ then $\mu$ has to intersect $ \Sigma\cup i_0$.
 
Hence, we have shown in all cases that any inextendable causal curve $\mu$ in $\Next$ intersects 
 $ \Sigma\cup i_0$.
 
  Let us next show that $\mu$ can intersect $\Sigma\cup i_0$ only once.
  Using the definition of the infinity $E$ of the manifold $\Next$, the causal
  vectors at the point $i_0$ with respect to
  the metric $g$ coinside with
  the causal vector with respect to the metric $-dT^2+g_{\Sp^3}$.
  Thus all causal paths from $i_0$ enter in the future to $ {{{N}}}^+$ and in the past to $ {{{N}}}^-$.
  This means that a causal path can not enter from $N$ to the point $i_0$.

 Let us first consider the case when $\mu\cap \N\not =\emptyset$. Then the path $\mu$ may exit from $N$ only to
 $\mathcal I^+$ or $\mathcal I^-$ and in particular, it can not intersect $i_0$. Let us consider the case when
 $\mu$ intersects  $\mathcal I^+$ but not  $\mathcal I^-$.
 Then $\mu\subset\N \cup {{{N}}}^+$ as $\mu$ can not be extended in $\Next$ to the past, there is a decreasing
 sequence $r_j\in I_0$ such that $\mu(r_j)\in N$ does not converge in $\Next$ as $j\to \infty$.
Then $\mu(r_j)\in N$ can neither converge in $N$ as $j\to \infty$. These imply that $\mu\cap \N$ is inextendable in $N$.
Moreover, there exists the smallest value $r'\in I_0$
such that $\mu(r')\in {{{N}}}^+$. 
Then there is an increasing
 sequence $r_j'\to r'$ such that $\mu(r_j')\in N$.
 As $\lim_{j\to \infty}\mu(r_j')=\mu(r')\in {{{N}}}^+$,
 we see that $\mu(r_j')$ does not converge in $N$ as $j\to \infty$.
 Thus, $\mu\cap \N$ is inextendable in the future.

Similarly, by considering all cases when  $\mu$ does or does not intersect  $\mathcal I^+$ or  $\mathcal I^-$, we see that in all cases $\mu\cap \N$ is inextendable in $N$.
Hence, $\mu\cap \N$ intersects $\Sigma$ at most once (and do not intersect $i_0$).

 Let us next consider  the case when $\mu\cap \N =\emptyset$.
We observe that a causal curve in   ${{{N}}}^+$ or in   ${{{N}}}^-$ can intersect
$i_0$ only at once. Moreover, when $\mu $ is
 parametrized to the future direction, we see that $\mu$  can enter from 
 ${{{N}}}^-$ to  ${{{N}}}^+$ only through $i_0$,
 and then it stays in $({{{N}}}^-)^{\mathrm{int}}$.
 Thus we see that $\mu$ can intersect $i_0$ only once.
 
Above, we  have seen that  any  inextendable causal curve
 in $\Next$ can intersect  $ \Sigma\cup i_0$ at most once.
  Thus,   $ \Sigma\cup i_0$ is a Cauchy surface of $\Next$ and hence
  $\Next$ is a globally hyperbolic Lorentzian manifold.
\end{proof}

\subsubsection{Notations related to the extended manifold}

Recall that
\beq\label{hat M 2}
\Next=\N \cup  {{{N}}}^+\cup  {{{N}}}^-.
\eeq
 We define
\begin{equation}\label{eq:setD}
\begin{split}
&\bD = \big(\I_{(\Next,g_\ext)}^+(p_-)\cap \I_{(\Next,g_\ext)}^-(p_+)\big)\setminus ({{{N}}}^+\cup {{{N}}}^-),
 \\
&\bD_0 = \I_{(\Next,g_\ext)}^+(p_-)\cap \I_{(\Next,g_\ext)}^-(p_+).   
\end{split}  
  \end{equation}

\begin{figure}[ht!]
\centering
\begin{tikzpicture}
\path[draw=black](0,0) rectangle (2,8);
\draw[draw=red] (0,2) -- (2,4);
\draw[draw=red] (0,6) -- (2,4);
\draw (1,0) node [above] {$\R\times\Sp^3$};
\draw (2,2) node [right] {$\mu$};
\draw (0.5,4) node [right] {$\bD$};
\draw (1,2.5) node [right] {$\Wminus$};
\draw (1,5.5) node [right] {$\Wplus$};

\filldraw[draw=black,fill=black] (0,2) node [left] {$(-\pi,\{\NP\})$} circle (1.5pt);
\filldraw[draw=black,fill=black] (2,4) node [right] {$(0,\{\SP\})$} circle (1.5pt);
\filldraw[draw=black,fill=black] (0,6) node [left] {$(\pi,\{\NP\})$} circle (1.5pt);

\filldraw[draw=black,fill=black] (2,1) node [right] {$p_-$} circle (1.5pt);
\draw[draw=black] (2,1) -- (0,3);

\filldraw[draw=black,fill=black] (2,0.4) node [right] {$p_{-2}$} circle (1.5pt);
\draw[draw=black] (2,0.4) -- (0,2.4);

\filldraw[draw=black,fill=black] (2,7) node [right] {$p_+$} circle (1.5pt);
\draw[draw=black] (2,7) -- (0,5);

\filldraw[draw=black,fill=black] (2,7.6) node [right] {$p_{+2}$} circle (1.5pt);
\draw[draw=black] (2,7.6) -- (0,5.6);

\end{tikzpicture}
\hfill
\includegraphics[height=8cm]{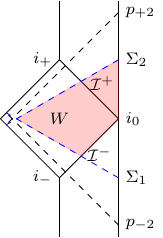}
\caption{{\bf Left:} Penrose diagram.  The red curves correspond to 3-dimensional light-like surfaces $\nullinf^-$ and $\nullinf^+$. 
{\mtext Note that in the case when the metric $g$ in $\R^3\times\R$ is time-independent,
the metric $\tilde g$ may be non-smooth near the points $i_+=(\pi,\{\NP\})$ and $i_-=(-\pi,\{\NP\})$. Also, $i_0=(0,\{\SP\})$ }
{\bf Right:} It suffices to consider the non-linear wave equation in the red shaded set $W$. The boundary $\p W$ consists of the past part that is a subset of $\mathcal I^-$ and subset of a Cauchy surface of $N_1=I^+_\Next(p_{-2})$ and the future part $\Sigma_f$ that is a subset of a smooth space-like surface.
\label{fig:penrose}}
\end{figure}

We  define
\beq\label{eq:sets W}
V^\pm =  \I^\pm_{\Next}(i_0)= ({{{N}}}^\pm)^\mathrm{int},
\eeq
and recall that (see Figure \ref{fig:penrose} (Left))
\ba
\Wminus &=& \I^-_{\Next}(i_0)\cap \I^+(p_-),\\
\Wplus &=& \I^+_{\Next}(i_0)\cap \I^-(p_+).
\ea

\commented{
As the manifold  $(\Next, g_\ext)$
 is globally hyperbolic, by \cite{Bernal2} there is a  diffeomorphism $\tilde \Phi:\Next\to \R\times \tilde \Sigma$, where $\tilde \Sigma$
 is a 3-dimensional manifold, and there are smooth maps ${\bf t}_0:\Next\to \R$ and $\sigma:\Next\to \tilde \Sigma$ 
 such that $ \tilde \Phi(x)=({\bf t}_0(x),\sigma(x))$ and that $\Sigma^0_{t_1}=\{x\in \Next \mid {\bf t}_0(x)=t_1\}$ are Cauchy surfaces. Without loss of generality, we can assume that $
 {\bf t}_0(i_0)=0$.
 
Let
$(\M,\gM)$ be a globally hyperbolic Lorentzian manifold  with an asymptotically Minkowskian infinity  that is visible in the whole space and let
$(\N,\gN)$ be a conformal manifold given in Definition \ref{def:asymptotically nice infinity A}.
Let $(\Next,\gext)$ be the corresponding
extended manifold.
 When
the map $\tilde \psi:\M\to \N$ is the conformal 
diffeomorphism that identifies $\M\to \N$,
we define a time function $\tilde {\bf t}_0:\M\to \R$ 
on $\M$ by setting $\tilde {\bf t}_0(x)={\bf t}_0(\tilde \psi(x))$.
}

\subsubsection{Two  time functions defining foliations}
Let us make the following auxiliary construction. 

Let (see Figure \ref{fig:penrose} (Left))
$$\N_1=I^+_{(\Next,\gext)}(p_{-2}),\quad \N_2=I^-_{(\Next,\gext)}(p_{+2}).$$ Also, let $\N_{12}=\N_1\cap \N_2$.
Observe that $(\N_1,\gext)$, $(\N_2,\gext)$ and $(\N_{12},\gext)$ are globally hyperbolic manifolds
and $\overline{\mathbb D_0}=J^-_{(\Next,\gext)}(p_{+})\cap J^+_{(\Next,\gext)}(p_{-})$ is a compact subset of $\N_1$ as well as a compact subset of $\N_2$.

As $(\N_j,\gext)$, $j=1,2$ are globally hyperbolic manifolds,  by \cite{Bernal2},
both these Lorentzian manifolds have a smooth time function 
\beq\label{time function tj}
{\bf t}_j:\N_j\to \R
\eeq
such that $$\Sigma_j(T)=\{x\in \N_j\mid {\bf t}_j(x)=T\},\quad T\in \R$$ are smooth
Cauchy surfaces of $\N_j$.

Let us next consider $N_1$.
 We denote 
$$
E^k([T',T''])= 
\bigcap_{j=0}^k C^j([T',T''];H^{j-k}(\Sigma_1(T)))$$ 
and 
$$
E^k_0([T',T''])= 
\bigcap_{j=0}^k C^j_0([T',T''];H^{j-k}(\Sigma_1(T)))
$$

The above setting is useful when we consider the
non-linear scattering problem on the set
\begin{equation}\label{eq:domain_for_wave_eq}
   W:=\{x\in \N_1 \mid {\bf t}_1(x)>T_1\}\cap  \{x\in \N_2\mid {\bf t}_2(x)<T_2\}\setminus  {{{N}}}^-
   \end{equation}
(see Figure \ref{fig:penrose}, Right)  and
\begin{equation}\label{eq:domain_for_wave_eq B}
   W_0:=\{x\in \N_1 \mid {\bf t}_1(x)>T_1\}\cap  \{x\in \N_2\mid {\bf t}_2(x)<T_2\}, 
\end{equation}
 where $T_1,T_2\in \R$ are chosen so that $\overline{\mathbb D_0}\subset W_0^{\mathrm{int}}.$

We will next consider a non-linear wave equation
in the domain $W$
 that is relatively compact in $\Next$, see Figure~\ref{fig:penrose}. Observe that the set $W$ can be covered with a finite number of
 coordinate neighborhoods of $\Next$. We consider $W$ primarily as a subset of a globally hyperbolic $\N_1$ space.

Note that the surfaces $\Gamma:=\{{x\in N_1\mid\bf t_1}(x)=T_1\} \cap \{x\in N_2\mid{\bf t_2}(x)=t\}$ are not necessarily smooth but we are going to consider the initial and boundary values which will imply that the solutions of the wave equations vanish identically near $\{x\in N_1\mid{\bf t_1}(x)=T_1\}$.

Below, we use $$T_{\mnewtext  1}^-<\min_{x\in \overline W}{\bf t}_{\mnewtext  1}(x),
\quad
T_{\mnewtext  1}^+>\max_{x\in \overline W}{\bf t}_{\mnewtext  1}(x).
$$

\section{Scattering functionals and the source-to-solution operator} 

\subsection{Existence of scattering functionals}

For $T_1<0$
let 
$$
\nullinf^-(T_1) := \nullinf^- \cap\{x\in N_1\mid {\bf t}_1(x)>T_1\}
$$
be the part of the past light-like infinity $\nullinf^-$, see Figure \ref{fig:penrose}, Right.
In the next theorem, we consider the well-posedness of the direct problem. For several physically relevant non-linear wave equations (e.g., with cubic or similar non-linear terms), more general
existence results can be obtained using the Baez-Segal-Zhou method \cite{Baez,Baez2}, but in the theorem below we will constructs 
special solutions whose extension to   the non-physical part of the space-time $N_-$  vanish in $N_-\cap U_0$, where 
 $U_0$ is  a neighborhood of the spatial infinity $i_0$ as it is convenient to analyze below the (conormal) singularites of such solutions.

\begin{theorem}\label{thm:nonlinear_goursat copy}
Let $k\ge 5$ and  $\kappa\geq 4$. 
Let $-\pi<t_1<0$. 
Let $T_1<0$ be  such that $S_-(R(t_1))\subset \nullinf^-(T_1)$. Moreover, let $T_2>0$
be such that ${\bf t}_2(p_+)<T_2$ and $W$ and $W_0$ be  the sets given in \eqref{eq:domain_for_wave_eq} and \eqref{eq:domain_for_wave_eq B}. Then,

\begin{enumerate}
    \item[(i)]\label{thm:nonlinear_goursat copy: exists}
{\mnewtext There {exist}  
$\e_0(t_1,T_1)>0$ and $C_0=C_0(t_1,n)$ such that
 when $0<\e<\e_0(t_1,T_1)$ and $G\in H^{k+1}(\nullinf^-)\cap \mathcal B_-(R(t_1))$  and
 $\|G\|_{H^{k+1}(\nullinf^-)}<\e$,
 the nonlinear Cauchy-Goursat problem
\begin{equation}\label{eq:nonlinear_goursat copy}
 \begin{cases}
\square_{g_\ext} {\tilde u} + {\mntext B}{\tilde u} +A{\tilde u}^\kappa= 0,&\text{in } W,\\
{\tilde u}=G, &\text{on } \nullinf^-(T_1),\\
{\tilde u}=0,&\text{on } \{x\in W\mid {\bf t}_1(x)<T_1\},
\end{cases} 
\end{equation}
has a unique solution ${\tilde u}\in H^{k}(W)$
satisfying $\|{\tilde u}\|_{H^{k}(W)}\le C_0\e$.

\item[(ii)] \label{thm:nonlinear_goursat copy: sources produce scattering data}
For any $n\in \mathbb Z_+$ there exist $\e_1(n)>0$   and $C_1=C_1(n)$ such that if $0<\e<\e_1(n)$ and   $\tilde{f}\in H^{k+1}_0(K_n)$ is
such that
\beq\label{source tilde f norm pre} 
\|\tilde f\|_{H^{k+1}(K_n)}\le \eps
\eeq
then the solution  $\tilde w\in H^k(W_0)$ (that exists due to Lemma~\ref{Lemma L-maps})   
 \begin{equation}\label{eq:linear_non-goursat2b pre}
 \begin{cases}
\square_{g_\ext} \tilde w + {\mntext B}\tilde w+A \tilde w^\kappa= \tilde f,\quad \text{in }W_0,
\\
\tilde w(x)=0\quad \hbox{for }{\bf t}_1(x)<T_1.
\end{cases}   
\end{equation}
and $G=\tilde w|_{\nullinf^-}$ satisfy
$G\in \mathcal B_-\cap H^{k}(\nullinf^-)$ and 
$\|G\|_{H^{k}(\nullinf^-)}\le C_1\e$.

\item[(iii)] \label{thm:nonlinear_goursat copy: scattering corresponds to source}
There {exist}  
$\e_2 (t_1,T_1)>0$, $C_2=C_0(t_1,T_1)$ and $n_0=n_0(t_1,T_1)$ such that
 when $0<\e<\e_2(t_1,T_1)$,
$G\in H^{k+1}(\nullinf^-)\cap \mathcal B_-(R(t_1))$,
 $\|G\|_{H^{k+1}(\nullinf^-)}<\e$,
 and $\tilde u$ solves \eqref{eq:nonlinear_goursat copy},
 then 
there is $\tilde f\in H^k(K_{n_0})$ satisfying
 \eqref{source tilde f norm pre} and  the solution 
 $\tilde w\in H^k(W)$ of \eqref{eq:linear_non-goursat2b pre}
such that  $\tilde w|_{W}={\tilde u}$,
and $G=\tilde w|_{\nullinf^-}$ and 
$\|\tilde f\|_{H^{k}(K_{n_0})}\le C_2\eps$.
}
\end{enumerate}
\end{theorem}

{\mnewtext The claim (i) and the Sobolev embedding theorem imply that under the 
assumptions on $h_-$ given in the claim, the scattering problem
\eqref{general non-linear problem on M}-\eqref{general non-linear problem on M3}
has a solution.}

\begin{proof}
(i) As $G\in  \mathcal B_-(R(t_1))$, 
it follows from the definitions of $\mathcal B_-(R)$ and $R(t_1)$ that
there are $(\phi_0,\phi_1)\in \mathcal E'(P(R(t_1)))^2$
and a solution $\tilde v$ of \eqref{wave on sphere2 A1}-\eqref{wave on sphere2 A2}, that is,
\beq\label{wave on sphere2 cp2} 
& &(\p_T^2-\Delta_{\mathbb S^3}+1)\tilde v=0,\quad \hbox{on }\R\times \mathbb S^3,
\\ & &(\tilde v|_{T=0},\p_T\tilde v|_{T=0})=(\phi_0,\phi_1), \nonumber
\eeq
such that $\tilde v|_ {\nullinf^-}=G$. Moreover, by using the finite speed of wave propagation
and  the
strong Huygens' principle,
we see that
 the function $\tilde v$ 
vanishes in the open set $V_0=( \R\times \mathbb S^3)\setminus S(R(t_1))$ that contains $i_0$ and  $i_-$. 

{\mnewtext  Recall that $\hat \mu(s)=(s,\SP)$.
We see that there are $t'_{-1},t'_{-2}\in (0,-2\pi)$, $t'_{-1}>t'_{-2}$ 
such that for $S(R(t_1))\cap \hat \mu(0,-2\pi) \subset \hat \mu (t'_{-2},t'_{-1})$
Then, $S(R(t_1))\subset \{\gamma_{p,\xi}(\R_+)\ \mid p\in \hat \mu (t'_{-2},t'_{-1}),
\quad \xi\in L^+_p(( \R\times \mathbb S^3))$.

Let now $0>t'_{0}>t'_{-1}$.}
Then, for the values 
 $0>t'_{0}>t'_{-1}>t'_{-2}>-2\pi$ 
 and  the points  $p'_{0}=(t'_{0},\SP)$, $p'_{-1}=(t'_{-1},\SP)$ and  $p'_{-2}=(t'_{-2},\SP)$
it holds that the set $B=J^-_{N^-}(p'_{-1})\cap J^+_{{N} ^-}(p'_{-2})$, 
is  a compact set
 $B\subset ({N}^-)^\mathrm{int}$ that satisfies \footnote{See the blue set in the Figure \ref{fig:sets B}, Right.}
\beq\label{B set condition 1}
& &B\subset  I_{{N}^-}^-(p'_0),\\
\label{B set condition 2}
& &( S(R)\cap {{N}}^-) \setminus  I^-(p'_0)\subset I_{{{N}}^-}^+(B^\mathrm{int}).
\eeq

Let  $n\in \mathbb Z_+$ be such that the compact
set $K_n\subset {{N}}^-$  satisfies
\beq
\label{B set condition 3}
& &B\subset (K_n)^\mathrm{int},\\
\nonumber
& &J^+(B)\cap J^-(p_0')\subset (K_n)^\mathrm{int}.
\eeq
As $G\in   \mathcal B_-(R(t_1))$, 
the solution $\tilde v$ of \eqref{wave on sphere2 cp2} 
satisfies
\beq\label{eq: support v tilde}
\supp(\tilde v)\subset S(R),\quad \supp(G)\subset S_-(R)=S(R)\cap \nullinf^-.
\eeq
As  $\|G\|_{H^{k+1}(\nullinf^-)}<\e$, the formula \eqref{eq: support v tilde} and  the energy estimate for
 the Goursat problem in the set $S(R)\subset \R\times \mathbb S^3$, we see using Proposition~\ref{prop:energy_inequality} in the Appendix for the linear wave equation
 that 
there is  $C_1=C_1(t_1)$ such that 
$\phi_0={\tilde v}|_{T=0}\in H^{k+1}(\mathbb S^3)$ and 
$\phi_1=\p_T{\tilde v}|_{T=0}\in H^{k}(\mathbb S^3)$
satisfy
\beq
\|\phi_0\|_{H^{k+1}(\mathbb S^3)}
+\|\phi_1\|_{H^{k}(\mathbb S^3)}\le C_1(t_1)\eps.
\eeq
These and the standard energy
estimates for the Cauchy problem for the linear wave equation, see   \cite[Prop.\ 9.12]{Ringstrom} and \cite[Thm. 3.7, p. 596] {ChBook}, imply that 
the solution $\tilde v$ of \eqref{wave on sphere2 cp2} 
satisfies
\ba
{\tilde v}\in  \bigcap_{\ell =0}^{k+1} C^\ell([-2\pi ,2\pi ];H^{k-\ell}(\mathbb S^3))\subset 
H^{k+1}([-2\pi ,2\pi ]\times \mathbb S^3)
\ea
and there is $C'_2(t_1)>0$ such that
\beq\label{norm of tilde V}
\|\tilde v\|_{H^{k+1}([-2\pi ,2\pi ]\times \mathbb S^3)}\le C'_2(t_1)\eps.
\eeq

 Let us now choose a function $\rho\in C^\infty(\R\times \mathbb S^3)$ such that 
  \beq\label{supp of rho and SR}
& &\rho(x)=1,\quad \hbox{ for } x\in   S(R)
 \setminus   I_{{{N}}^-}^-(p'_0),\\
\label{supp of rho in K}
& &\supp(\rho) \subset I^+_{\R\times \Sp^3}(K_n),
\\
\label{supp of rho minus p'0}
& & \supp(\rho) \cap I^-_{{{N}}^-}(p'_0)\subset K_n.
\eeq

Let us define
\beq\label{new wave}  
\tilde v_1(x)=\rho(x)\tilde v(x),\quad\hbox{for }x\in \R\times \mathbb S^3.
\eeq 
Then, $\tilde v_1$ satisfies the equation
\beq\label{wave on sphere2 sourceB} 
& &(\p_T^2-\Delta_{\mathbb S^3}+1)\tilde v_1=\tilde f,\quad \hbox{on }\R\times \mathbb S^3,
\\ 
\nonumber & &\supp(\tilde v_1)\subset J^+(\supp(\tilde f)).
\eeq
As  $\rho=1$ in $S(R)
 \setminus   I_{{{N}}^-}^-(p'_0),$ we have $\supp(\tilde f)\cap   I_{{{N}}^-}^-(p'_0)=\emptyset.$
 Moreover, by \eqref{supp of rho minus p'0},  $\supp(\tilde f)\cap  I_{{{N}}^-}^-(p'_0)\subset K_n$.
These yield that
 \beq\label{supp of rho}
\supp(\tilde f) \subset K_n.
\eeq
Moreover,  as $\rho=1$ on $\nullinf^-\cap S(R)$, the formula \eqref{supp of rho and SR} implies that 
\beq\label{wave on sphere2 source Goursat value} 
\tilde v|_{\nullinf^-}=G.
\eeq
Then,  by using  formulas \eqref{norm of tilde V}, \eqref{new wave},  and \eqref{wave on sphere2 sourceB} 
we see that 
\beq\label{source tilde f norm} 
\|\tilde f\|_{H^{k}([-2\pi ,2\pi ]\times \mathbb S^3)}\le C'_3(t_1,n)\eps.
\eeq

When we consider $K_n$ as a compact subset
of $N_1$, we see that 
there is $T_1\in \R$ such that $K_n\subset \{x\in N_1 \mid {\bf t}_1(x)>T_1\}$.

Observe that the function $\tilde f|_{ ({N}^-)^\mathrm{int}}$  in $H^{k}( ({{N}}^-)^\mathrm{int})$ that is
compactly supported in $K_n\subset ({N}^-)^\mathrm{int}$
can be continued from the set ${N}^-$ by zero to a function in $N_1$
so that the obtained function, that we continue to denote by $\tilde f$,
satisfies $\tilde f \in E^{k-1}([T_1^-,T_1^+])$ and  $\|\tilde f\|_{E^{k-1}([T_1^-,T_1^+])}\le C_3(t_1,n)\eps$.

By using local existence results for the Cauchy problem for the non-linear wave equation on
the globally hyperbolic manifold $N_1$ with a  source supported in the compact set $K_n$,
see \cite[Thm. 3.12] {ChBook},
we see that when $\eps$ is assumed to be small enough, there exists $\tilde w\in  \bigcap_{\ell =0}^{k}C^\ell([T_1^-,T_1^+];H^{k-\ell}(\Sigma_1(T))$ that satisfies
 \begin{equation}\label{eq:linear_non-goursat2b}
 \begin{cases}
\square_{g_\ext} \tilde w + {\mntext B}\tilde w+A \tilde w^\kappa= \tilde f,\quad \text{in }\{x\in N_1 \mid
{\bf t}_1(x)<T_1^+\},\\
\tilde w(x)=0\quad \hbox{for }{\bf t}_1(x)<T_1^-
\end{cases}   
\end{equation}
and $$\|\tilde w\|_{E^{k}(T_1^-,T_1^+)}\le C'_4(t_1,n)\e.$$
Observe that in ${{N}}^-$ we have $\tilde w=\tilde v_1$, as the functions $A$ and $B$ vanishes in  ${{N}}^-$
and the function $\tilde f$ is supported in the set where ${\bf t}_1(x)\ge T_1^-$.
Thus, $ \tilde w= \tilde v_1= \tilde v=G$ on $ \nullinf^-$. Hence, ${\tilde u}=\tilde w|_{W}$ 
is a solution of 
\begin{equation}\label{eq:non-linear_goursat B2}
 \begin{cases}
\square_{g_\ext} {\tilde u} + {\mntext B}{\tilde u} +A{\tilde u}^\kappa= 0,&\text{in } W,\\
{\tilde u}=G, &\text{on } \nullinf^-(T_1),\\
{\tilde u}=0,&\text{on } \{x\in W\mid {\bf t}_1(x)<T_1\},
\end{cases}   
\end{equation}
and satisfies $\|\tilde w\|_{H^k(W)}\le C_4\e$.

Summarizing the above, we have shown that  when $\eps$ is small enough,
for any $G\in  \mathcal B_-(R(t_1))$ satisfying
$\|G\|_{H^{k+1}(\nullinf^-)}<\e$, there is
 $\tilde f$ satisfying  $\|\tilde f\|_{H^k(W)}\le C_5\|G\|_{H^{k+1}(\nullinf^-)}$, and
a unique solution $\tilde w$ for the equation
 \eqref{eq:linear_non-goursat2b}
and that  ${\tilde u}=\tilde w|_{W}$ is a
 solution  for the equation \eqref{eq:non-linear_goursat B2}.
 In particular, we have shown the existence of solutions for the equation \eqref{eq:non-linear_goursat B2}.

{\mnewtext 
Next, we consider the uniqueness of the solutions for the equation \eqref{eq:non-linear_goursat B2}.
The results of Nicolas \cite{Nicolas} for the linear wave equation
imply, see the Proposition 
 \ref{prop:energy_inequality uniqueness}
 in the Appendix, 
  for the Goursat problem for the non-linear 
 equation  that
the solution ${\tilde u}$ of \eqref{eq:non-linear_goursat B2}, when it exists, is unique.
This yields the claim (i).

(ii) For any $n$ there is $R=R_n$ such that $K_n\subset S(R)$. 
As the coefficient $A$ of the non-linear term vanishes in $N^-$, we see that the solution $\tilde w$ of 
\eqref{eq:linear_non-goursat2b} coincides with the solution of the corresponding linear wave equation and thus by
the strong Huygen's principle \cite{Lax},
$\supp(\tilde w)\cap N^-$ is contained in $ S(R)$. Hence $\tilde w|_{\nullinf^-}\in \mathcal B^-(R)$.
By Lemma~\ref{Lemma L-maps}, $\tilde w|_{\nullinf^-}\in H^k(\nullinf^-)$. Theise imply the claim (ii).

Finally the claim (iii) follows by the construction of $\tilde f$
given in the proof of the claim (i).}
\end{proof}

Let $\kappa\geq 4$, $-\pi<t_1<0$ and  $q\in \nullinf^+$. 
For $t_2$, we choose $T_2\in \R$ such that
\beq
p_+,q\in   \{x\in W\mid {\bf t}_2(x)<T_2\}.
\eeq

By Theorem  \ref{thm:nonlinear_goursat copy}, there are $\e>0$ and $C_0$ such that if
 $h_-\in H_0^{k+1}(\nullinf^-)\cap \mathcal B_-(R(t_1))$ 
then  the nonlinear Cauchy-Goursat problem 
\eqref{eq:nonlinear_goursat copy} with $G=h_-$
has a unique solution $\tilde u$.
Let us  denote $ u(x)=\omega_M(x) \tilde u(x)$ the solution of the 
non-linear scattering problem
\begin{equation}\label{general non-linear problem on M cp}
\begin{cases}
  \square_{\gM}  u(x)+{\mntext d(x)} u(x) + {\mltext a(x)}\, u(x)^{\kappa}=0, &\text{in } 
   \{x\in \M\mid {\bf t}_2(x)<T_2\}
  \\ 
  u(x) = h_-(x) & \hbox{on } \mathcal I^-\cap \{x\in \M\mid {\bf t}_2(x)<T_2\},\\
  u=0 &\hbox{on }M\setminus J^+_\M(\supp(h_-)),
  \end{cases}    
\end{equation}
that  satisfies
$$
\|\omega_M^{-1}{\tilde u}\|_{ C([T_1^-,T_1^+]; H^k(\Sigma_1(T)))\cap C^{k}([T_1^-,T_1^+]; L^2(\Sigma_1(T)))}<C_0\e.
$$

In the set
\beq
 \label{domain of S copy} &&
\mathcal D(S_{t_1,{q}})=\mathcal D_{(\e)}(S_{t_1,{q}})
= \{h\in H^k(\nullinf^-)\cap \mathcal B_-(R(t_1))\mid 
 \|h\|_{H^k(\nullinf^-)}<\e\}\hspace{-1mm}
\eeq
we have a well-defined  map
\beq
& &S_{t_1,{q}}:\mathcal D(S_{t_1,{q}})\to \R,\\
& &S_{t_1,{q}}(h_-)=\tilde u(q),
\eeq
where $\tilde u$ solves \eqref{general non-linear problem on M cp}. {\mnewtext  In particular, this implies that the scattering
functionals $S_{t_1,{q}}$ are well-defined.}

The next goal is to study the globally hyperbolic space-time $(\Next, \gext)$, where the sources are supported in $ {{{N}}}^-$ and
the waves are observed on  ${{{N}}}^- $, and  aim to construct the conformal type of $(\mathbb D, \gext)$.
{We recall that $\gext$ coincides with $\gN$ within $\N$.}

We recall that $ \mathcal V_n\subset H^k_0( 
K_n)$ are sufficiently small neighborhoods of the zero function
and we have defined the source-to-solution maps 
\beq\label{L maps defined copied}
L_{g_{\ext},B,A,p_+,K_n}: \mathcal V_n\subset H^k_0(K_n) \to
L^2({{N}}^+),
\eeq
that map $L_{g_{\ext},B,A,p_+,K_n}f = u|_{{{N}}^+}$, where $u$ solves \eqref{eq:nonlinear wave equation in Penrose 2}.

\subsection{Scattering functionals determine the source-to-solution operator}

Next we prove Theorem~\ref{t:scattering_operator_determines V}, that is,
the  scattering functional determine the near field measurements, i.e., 
source-to-solution operator. 

\begin{proof} (of Theorem~\ref{t:scattering_operator_determines V})
Given a source $f\in \mathcal V_n\subset H^k_0(K_n)$, we solve on both manifolds $N_{\ext}^{(j)}$ the following linear initial value problem
 \begin{equation}\label{eq: initial with f}
\begin{cases}
(\square_{\gext}+{\mntext B}) u^{(j)}  = f,\quad \hbox{in }N_{\ext}^{(j)},\\
\supp(u^{(j)})\subset J^+_{N_{\ext}^{(j)}}(\supp(f)).
\end{cases}
\end{equation}

In ${{N}}^-$ it holds that $u^{(j)}=\tilde u^{(j)}|_{{{N}}^-} $ where $\tilde u^{(j)}$ is  the wave equation 
\beq\label{wave on sphere} 
& &(\p_T^2-\Delta_{\mathbb S^3}+1)\tilde u^{(j)}=f,\quad \hbox{on }\R\times \mathbb S^3,
\\
& &\nonumber
\supp(\tilde u^{(j)})\subset J^+(\supp(f)).
\eeq 
The
strong Huygen's principle, see \cite{Lax}, implies then that 
\begin{align*}
\supp(\tilde u^{(j)}) \subset \{\gamma_{x,\xi}(s) &\mid x\in \supp(f),\ \xi\in L_x^+(\R\times \mathbb S^3),\ 
s\ge 0\}.    
\end{align*}
Recall that $\supp(f)\subset K_n$ where $K_n\subset  ({N}^-)^\mathrm{int}\subset I^-(i_0)$ is compact
and that no geodesics connecting a point $x\in K_n$ to $i_0$ is  light-like.
This and the strong Huygen's principle imply
that $i_0$ has a neighborhood $V_0\subset \R\times \mathbb S^3$ such that
 $\supp(\tilde u^{(j)})\cap V_0=\emptyset$. Moreover, by causality, $\tilde u^{(j)}$ vanishes
 in a neigborhood of $i_-$. These imply that 
 \beq
 \tilde u^{(j)}|_{{\nullinf}^-}\in \mathcal B_-(R(t_1))
 \eeq
 with some $-\pi<t_1<0$.
 
As the subsets  ${{N}}^-$ of the manifolds $N_{\ext}^{(j)}$, $j=1,2$ coincide on both manifolds  (i.e. those are isometric), the boundary values of the solutions on $ {\nullinf}^-$
coincide, 
\beq
u^{(1)}|_{{\nullinf}^-}=u^{(2)}|_{{\nullinf}^-}.
\eeq
By assumption, scattering functionals 
$S_{\M^{(j)},\gM^{(j)},a^{(j)};t_1,{q}}=S^{(j)}_{t_1,{q}}$, $j=1,2$ coincide for all $-\pi<t_1<0$ and 
$q\in \nullinf^+$. 

Let us next use  $-\pi<t_1<0$ such that for $R=R(t_1)$ the condition
\beq
\label{B set condition inverse 1}
\mathcal B(K_n)\subset S(R)
\eeq
holds, where
\begin{align}
\mathcal B(K_n):= \{\gamma_{x,\xi}(s)\mid\ x\in K_n,
 \xi\in L_x^+(\R\times \mathbb S^3),\ 
s\ge 0\}.
\end{align}
Moreover, we use $q\in \nullinf^+$ such that 
$$
J^-_{{{N}}^+}  (p_+)\cap\nullinf^+ \subset  \nullinf^+(q).
$$ 
Then, for any $q'\in  \nullinf^+(q)$
\begin{equation}
\begin{split}
u^{(1)}|_{\nullinf^+(t_2)}(q')&=S^{(1)}_{t_1,{q}} u^{(1)}(q')=S^{(2)}_{t_1,{q}} u^{(2)}(q')=
u^{(2)}|_{\nullinf^+(t_2)}(q').   
\end{split}
\end{equation}
Using $h_+=u^{(1)}|_{\nullinf^+(t_2)}$,
we solve for both manifolds $(\M^{(j)},\gM^{(j)})$, $j=1,2$ a linear Goursat problem
 \begin{equation}\label{eq: initial with future Goursat}
\begin{cases}
(\square_{g_{\ext}^{(j)}}+{\mntext B^{(j)}}) u^{(j)}  = 0,\quad \hbox{in }x\in {{{N}}}^+\cap J^-_{N_{\ext}^{(j)}}(p_+),\\
u^{(j)}\big|_{{\nullinf}^+\cap  J^-_{N_{\ext}^{(j)}}(p_+)}={h_+}.
\end{cases}
\end{equation}
By \cite{H90,Nicolas}, the Goursat problem \eqref{eq: initial with future Goursat} has a unique solution
in $H^1({N}^+\cap J^-_{N_{\ext}^{(j)}}(p_+))$. 
Hence, we see that the solutions of  the equations \eqref{eq: initial with future Goursat} with $j=1,2$ (defined using the two manifolds  $(\M^{(j)},g^{(j)})$) coincide on the set ${N}^+\cap J^-_{N_{\ext}^{(1)}}(p_+)=
{N}^+\cap J^-_{N_{\ext}^{(2)}}(p_+)$. This implies that the source-to-solution operators satisfy $L_{g_{\ext,1},B_{1},A_1,p_+,K_n}(f)=L_{g_{\ext,2},B_{2},A_2,p_+,K_n}(f)$
for all $f\in \mathcal V_n$. 
This proves the claim.
\end{proof}

\section{Microlocal analysis of the source-to-solution operator}

Below, we consider the inverse problem when the sources are supported in $K_n\subset 
({N}^-)^\mathrm{int}$
and the waves are observed in  $({N}^+)^\mathrm{int}$. As  $({N}^-)^\mathrm{int}$
is in the chronological past of $i_0$ and   $({N}^+)^\mathrm{int}$
is in the chronological future of $i_0$, the set where the sources are supported and the
set where the wave are observed are causally separated. This situation causes
difficulties: As an example, let us consider the Lorentzian manifold $\R\times \Sp^3$ and the sets
$\omega_j=(-2j\pi,0)\times \Sp^3$, $j=1,2$.
For the sake of presenting a simple example, let us assume that 
we can use sources on
$(\R\times \Sp^3)\setminus \omega_j$ and make observations on the set 
$U=(0,\infty)\times \Sp^3$ to determine the light observation sets
of points $q$ in $\omega_j$, that is, 
$$\mathcal E_U(\omega_j)=\{\mathcal E_U(q) \mid q\in \omega_j\}.$$

Then, as all great circles of $\Sp^3$ are closed geodesics of length $2\pi$,
we see that $\mathcal E_U(\omega_1)=\mathcal E_U(\omega_2)$,
that is, the light observation sets
of points in $\omega_1$ and  $\omega_2$ coincide and these
sets are indistinguishable using light observation sets. 
Observe that $\R\times \Sp^3$ is a special manifold
in the sense that antipodal points on $\Sp^3$ give rise to conjugate
points for light-like geodesics. Due to these observations, below in Section
\ref{sec: proof of main thm}, we will
consider inverse problem for  the source-to-solution maps with causally separated
sources and observations paying
particular attention to the cut-points and conjugate points.
Before that, in this section we give a modification of the notations on conormal sources
and interacting waves introduced in \cite{KLU2018,LUW}. Instead of reconstructing
the space-time using a layer striping process done in \cite{KLU2018,LUW}, we show
that the light observation set $\mathcal E_{{{{N}}}^+}(q)$ of the point $q$
can be directly reconstructed if the distant areas of the space-time can be reconstructed 
when the point $q$ can be connected to the set ${{{N}}}^-$
with a light-like geodesic that have no cut points.
Moreover,
to be able to change the conformal factor of the metric, we will pay attention to
the fact that many of the construction steps are independent of the coefficient $B$ of the 
zeroth order term in the wave equation.

\subsubsection{Definitions in microlocal analysis and nonlinear interactions}

Below,  we use that $\Nextended$ and its subset $W$ are  globally hyperbolic Lorentzian manifolds of dimension $(1+3)$.
Observe that below the coefficient $B$ is allowed to be a general smooth function on $\Next.$
To simplify the notations, we denote the metric $\gext$ of  $\Nextended$ just by $g$.
We will consider the equation
\beq\label{non-linear wave equation on cal M}
\begin{cases}
\big(\square_g + B\big)u +Au^{\kappa}=f,&\text{in }W,\\
u=0,\text{ in } \quad {\mtext W}\setminus J^+(p^-),
\end{cases}
\eeq
where $f\in \cup_n \mathcal V_n\subset H_0^{k}(\Wminus\setminus J^+(p^-))$.
We assume that we are given  the source-to-solution maps $
L_{g_{\ext},B,A,p_+,K_n}: f\mapsto u|_{{{N}}^+\cap I^-_{\Next}(p_+)}$ where $p_+\in \hat \mu(0,\pi)$.

Below, we use  for a pair $(x,\xi)\in T\Nextended$  the notation 
\[
(x(t),\xi(t))=(\gamma_{x,\xi}(t),\dot \gamma_{x,\xi}(t)).
\]
Let $g^+$ be a smooth Riemannian metric on $\Nextended$.

Let  $\eta_0\in L^+_{i_0}\Nextended$,
$\norm{ \eta_0}_{g^+}=1$.
Since $\rho(x,\xi)$ on $(\Nextended,g)$ is lower semi-continuous and 
$\overline {\mathbb D}_0=J^+(p^-)\cap J^-(p^+)$ is compact, we see that for all $\e_0>0$ there are $\delta_0,\delta_1>0$ such that for $\widehat{x} = \mu_\i(s_0)$, $ -\delta_0\le s_0\le 0$, $x\in {\mathbb D}_0,$ with $d_{g^+}(x,\widehat{x})<\delta_1$ and $\xi\in L^+_x\Nextended,$ with $\norm{ \xi }_{g^+}=1$, we have $\rho(x,\xi)>\rho(i_0,\eta_0)-\e_0$.

\subsubsection{Observation time functions}

Let us define the observation time functions as in \cite{KLU2018}:

{\mltext

\begin{definition}
Let $\mu: [s^-,s^+]\to \Nextended$  be a time-like path. We define $f_\mu^+(x),f_\mu^-(x)\in \R$ by the formulae
\begin{align*}
f_\mu^+(x)&:= \inf(\{s\in [s^-,s^+]\mid \tau(x,\mu(s))>0\}\cup\{s^+\}),\\
f_\mu^-(x)&:= \sup(\{s\in[s^-,s^+]\mid \tau(\mu(s),x)>0\}\cup\{s^-\}).
\end{align*}
The value $f_\mu^+(x)$ is called the earliest observation time from the point $x$ on the path $\mu$. 

\end{definition}

}

\subsubsection{Notation for the sources and Lagrangian submanifolds}

We will consider sources supported in $\Wminus$ and observations made in $\Wplus$.
For $x_0\in \Wminus$, $\zeta_0\in L_{x_0}^+\Nextended$, and $s_0>0$ we define 
\begin{equation}\label{eq:set_V}
\mathcal{V}_{x_0,\zeta_0,s_0}
=
\{
\eta\in T_{x_0} \Nextended \mid \Vert \eta-\zeta_0\Vert_{g^+} < s_0,\, \Vert \eta\Vert_{g^+} =\Vert\zeta_0\Vert_{g^+}
\}.
\end{equation}
Let $\mathcal{W}_{x_0,\zeta_0,s_0} = L_{x_0}^+\Next\cap\mathcal{V}_{x_0,\zeta_0,s_0}$ and define
\begin{equation}\label{eq:set_K}
K(x_0,\zeta_0,s_0)
=
\{
\gamma_{x_0,\eta}(t)\in \Nextended \mid \eta\in \mathcal{W}_{x_0,\zeta_0,s_0},\, t\in (0,\infty)
\}.
\end{equation}
We also define
\begin{equation}\label{eq:sets_Sigma_Lambda}
\begin{split}
\Sigma(x_0,\zeta_0,s_0) &= \{
(x_0,r\eta^\flat)\in T^*\Nextended \mid \eta \in \mathcal{V}_{x_0,\zeta_0,s_0},\, r\in \R\setminus\{0\}
\},\\
\Lambda(x_0,\zeta_0,s_0) &=
\{
(\gamma_{x_0,\eta}(t),r\dot\gamma_{x_0,\eta}(t)^\flat)\in T^*\Nextended \mid
\eta\in \mathcal{W}_{x_0,\zeta_0, s_0},
\\&\quad\quad\quad\quad\quad\quad\quad\quad\quad\quad t\in (0,\infty),\, r\in \R\setminus\{0\}
\}.
\end{split}
\end{equation}
Observe that $\Lambda(x_0,\zeta_0,s_0)$
is a Lagrangian submanifold of $\Next$.
Roughly speaking, near the points where $K$ is a smooth manifold, $\Lambda$ is its conormal bundle.
Intuitively, $K$ is the light-cone associated to a small spherical cap $\mathcal{V}$, $\Sigma$ is the vertex of this cone in the cotangent space, 
c.f.\ \cite{Vasy1,Vasy2} or related results on scattering from corners, and $\Lambda$ is the set to which this spherical cap propagates under the Hamiltonian flow. At the limit $s_0\to 0$ the set $K$ tends to the  geodesic $\gamma_{x_0,\eta}$. We will soon construct sources with singularities on $\Sigma$. The corresponding solutions propagate along geodesics and their singularities propagate in $\Lambda$. This allows us to use the conormal calculus to study singularities produced in collision of such waves.

Let $p(x,\xi)=g^{ij}\xi_i\xi_j$ be the principal symbol of $\square_g$.
The Hamilton vector field of $p$ is denoted by $H_p$ and it is given in local coordinates by
$$
H_p = \sum_{j=0}^3 \left(
\frac{\p p}{\p\xi_j}\frac{\p}{\p x^j}-\frac{\p p}{\p x^j}\frac{\p}{\p \xi_j}
\right).
$$
The integral curves of $H_p$ are called null bicharacteristics, denoted by $\Theta$.
We denote by $\Theta_{x,\xi}$ the bicharacteristic containing  {\color{black}$(x,\xi)\in L^*\Nextended$}.
It is worth noting that $(y,\eta)\in \Theta_{x,\xi}$ if and only if 
$(y,\eta^\sharp) = (\gamma_{x,\xi^\sharp}(t),\dot \gamma_{x,\xi^\sharp}(t))$ for some $t\in \R$, where $\gamma_{x,\xi^\sharp}$ is the light-like geodesic with $(x,\xi^\sharp)\in L{\Nextended}$.

Now, denoting by $\mathrm{char}(\square_{g}) = \{ (x,\xi)\in T^*{\Nextended}\mid p(x,\xi)=0 \}$ the characteristic variety of $p$,
we have that $\Lambda(x_0,\zeta_0,s_0)$ is the Lagrangian submanifold that is obtained by flowing-out of $\mathrm{char}(\square_{g})\cap\Sigma(x_0,\zeta_0,s_0)$ by the Hamilton flow of $p$ in the future direction. %

We will use conormal and Lagrangian distributions to analyse the propagation of singularities.
To do this, let us recall some relevant notations and definitions.
Suppose $X$ is a smooth $n$-dimensional manifold and $\Lambda\subset T^*X\setminus\{0\}$ is a Lagrangian submanifold.
Let $(x,\theta)\in X\times \R^n$ and $\phi(x,\theta)$ be a non-degenerate phase function, which locally parametrizes $\Lambda$ near $(x_0,\xi)\in\Lambda$, that is, in some conic neighbourhood $\Gamma \subset T^*X\setminus\{0\}$ the set $\Lambda\cap \Gamma$ coincides with the set $\{ (x, d_x\phi(x,\theta))\in \Gamma \mid d_\theta\phi(x,\theta)=0\}$.
The set of classical Lagrangian distributions $I^m(X;\Lambda)$ is then defined to be those distributions $u\in \mathcal{D}'(X)$ that can be represented in local coordinates $X: V\to \R^n$,
defined in an open set $V\subset X$, as an oscillatory integral {\color{black} (modulo a $C^\infty$ function)} of the form
$$
u(x) = \int_{\R^N} e^{i\phi(x,\theta)} a(x,\theta) d\theta, \quad x\in V.
$$
Here $a(x,\theta) \in S^{m+\frac{n}{4}-\frac{N}{2}}(V;\R^N)$ is a classical symbol of order $m+\frac{n}{4}-\frac{N}{2}$.
Corresponding to a classical Lagrangian distribution $u\in I^m(X;\Lambda)$ there is a principal symbol $\sigma_u^{(p)}(x_0,\zeta_0)$, $(x_0,\zeta_0)\in \Lambda$, which satisfies
$$
\sigma_u^{(p)}(x_0,\zeta_0) \in S^{m+\frac{n}{4}}(\Lambda,\Omega^\frac{1}{2}\times L)/
S^{m+\frac{n}{4}-1}(\Lambda,\Omega^\frac{1}{2}\times L),
$$
where $L$ is the Maslov-Keller line bundle and $\Omega^\frac{1}{2}$ is the half-density on $X$.
When $\Lambda$ is a conormal bundle of a smooth submanifold $S\subset X$, that is,
$\Lambda=N^*S$, 
the distributions $u\in I^m(X;\Lambda)$ are called conormal distributions.

Furthermore, we will need distributions associated to two cleanly intersecting Lagrangians \cite{GuU1,MU1}.
We say that two Lagrangians $\Lambda_0,\Lambda_1\in T^*X\setminus\{0\}$ intersect cleanly 
 if
$$
T_p \Lambda_0 \cap T_p \Lambda_1 = T_p(\Lambda_0\cap \Lambda_1)
$$
for all $p\in \Lambda_0\cap \Lambda_1$.
The set of  {\color{black}  distributions associated to two Lagrangian manifolds} $\Lambda_0$ and $\Lambda_1$ is denoted by $I^{k,l}(X;\Lambda_0,\Lambda_1)$.
It is known that if $u\in I^{k,l}(X;\Lambda_0,\Lambda_1)$, then $\WF(u) \subset \Lambda_0\cup\Lambda_1$ and moreover microlocally away from $\Lambda_0\cap\Lambda_1$ we have $u\in I^{k+l}(X;\Lambda_0\setminus\Lambda_1)$ and $u\in I^k(X;\Lambda_1\setminus\Lambda_0)$.
Since all cleanly intersecting Lagrangian submanifolds with given dimension of intersection are locally equivalent (see \cite{GuU1} or Theorem 3.5.6 of \cite{Duistermaat}), it will be enough to consider only model Lagrangians in the Euclidean case.
Let us denote $(x_1,\ldots,x_n) = (x',x'',x''')\in \R^n$, where $x' = (x_1,\ldots,x_{d_1})$, $x'' = (x_{d_1+1},\ldots, x_{d_1+d_2})$ and $x''' = (x_{d_1+d_2+1},\ldots,x_n)$.
Following \cite{GU1993}, we represent the distributions using the 
Lagrangian distributions %
$$
\Lambda_0 = N^*\{x'=x''=0\} \text{ and } \Lambda_1 = N^*\{x'=0 \}.
$$ 
Then $u\in I^{k,l}(\R^n;\Lambda_0,\Lambda_1)$ can be explicitly written in terms of oscillatory integrals  {\color{black} (modulo a $C^\infty$ function)} as
$$
u(x) = \int_{\R^{d_1+d_2}} e^{i (x'\cdot\theta' + x''\cdot \theta'')} a(x,\theta',\theta'')d \theta' d\theta''.
$$
Here the symbol $a(x,\theta',\theta'')\in S^{\mu_1,\mu_2}(\R^n;(\R^{d_1}\setminus\{0\})\times \R^{d_2})$ is of product type, that is, $a\in C^\infty(\R^n\times \R^{d_1}\times \R^{d_2})$ and
for all compact subsets  $ K\subset \R^n$ it holds that
$$
|\p_{x'''}^\gamma \p_{\theta'}^\alpha \p_{\theta''}^\beta a(x,\theta',\theta'')|
\leq C_{\gamma\alpha\beta K} (1+|\theta'|+|\theta''|)^{\mu_1-|\alpha|} (1+|\theta''|)^{\mu_2-|\beta|},
$$
for all $x\in K$, where
$\mu_1=k-\frac{n}{4}+\frac{d_1}{2}$ and $\mu_2 = l+\frac{d_2}{2}$. %

We will often abbreviate
$$
I^p(X;\Lambda) = I^p(\Lambda)\quad\text{and}\quad
I^{m_1,m_2}(X;\Lambda_1,\Lambda_2) = I^{m_1,m_2}(\Lambda_1,\Lambda_2).
$$
Also, if the cotangent bundles of submanifolds $S_j$ of codimension $d_j$, $j=1,2$, are cleanly intersecting, we will use notations
$$
I^\mu(S_1) = I^{\mu+\frac{d_1}{2}-\frac{n}{4}}(N^*S_1)\quad\text{and}\quad
I^{\mu_1,\mu_2}(S_1,S_2) = I^{m_1,m_2}(N^*S_1,N^*S_2),
$$
where $m_1 = \mu_1+\mu_2+d_1/2-n/4$ and $m_2 = -\mu_2+d_2/2$.
It is occasionally useful to use the embedding $I^\mu(\Lambda)\subset H^s({\Nextended})$, that is continuous for all $s<-\mu-\frac{n}{4}$.

\subsubsection{Causal inverse of $\square_{g}+B$}
When $\Lambda_1\subset T^*{\Nextended}$ is a Lagrangian manifold which intersects the characteristic variety of $\square_g$, we can consider solutions $u_1$ of $\square_{g} u_1 + Bu_1 = f_1$, with source $f_1\in I^m(\Lambda_1)$.
As $\Lambda_1\cap \Char(\square_{g})$ is not empty, we find that the wavefront set of $u_1$ is contained in the union of $\Lambda_1$ with the bicharacteristics that contain points of the intersection.

More precisely, %
on globally hyperbolic Lorentzian manifolds the  hyperbolic operator $\square_{g}+\sct$ has a unique causal inverse operator $Q=(\square_{g}+\sct)^{-1}$, see \cite[Theorem 3.2.11]{BGP}.
Denoting the Schwartz kernel of the operator $Q$ again by $Q=Q(x,y)$, we have $Q\in I^{-\frac{3}{2},-\frac{1}{2}}(\Delta^\prime_{T^*{\Nextended}}, \Lambda_{g})$, see \cite{GU1993}.
Here $\Delta^\prime_{T^*{\Nextended}}$ denotes the conormal bundle of the diagonal, $\Delta^\prime_{T^*{\Nextended}}=N^*(\{(x,x)\mid x\in {\Nextended}\})$ and $\Lambda_{g}\subset T^*{\Nextended}\times T^*{\Nextended}$ is the Lagrangian manifold associated to the canonical relation of $\square_g$, given by,
\begin{equation}\label{eq:canonical relation for gpert}
\Lambda_{g} = \{(x,\xi,y,-\eta)\mid (x,\xi)\in\Char(\square_{g}),\,(y,\eta)\in \Theta_{x,\xi} \},
\end{equation}
where $\Theta_{x,\xi}$ is the bicharacteristic of $\square_{g}$ containing $(x,\xi)$.

{\color{black} When $\Lambda_0$ is a Lagrangian manifold  and the intersection $\Lambda_0\cap \Char(\square_{g})$ is transversal, the union $\Lambda_1$ of bicharacteristics intersecting this set} is a Lagrangian manifold.

\begin{lemma}\label{lem: Lagrangian 1 wave} 

Let $n$ be an integer,  ${ s_0}>0$, 
$K=K(x_0,\zeta_0,s_0)$, $\Lambda_1=\Lambda(x_0,\zeta_0,s_0)$ and $\Sigma=
\Sigma({x_0,\zeta_0,s_0})$.
Let  $({ x},\xi)\in  { \Sigma} \cap L^{*}{\Nextended}$, $v=\xi^\sharp\in L_x{\Nextended}$, $r\in \R$ and
$y=\gamma_{x,v}(r)$ and $\eta=(\dot \gamma_{x,v}(r))^\flat$ 
be such that  ${ x}<{ y}$. Assume that ${f_1}\in I^{n+1}({ \Sigma})$ is a compactly supported distribution with a classical symbol.

Then
 $w_1 = (\square_{ {{g}}}+B)^{-1}{f_1}$ satisfies
$w_1\in I^{n-1/2,-1/2}(\Sigma,\Lambda_1)$.

Let  $\sigma^{(p)}_{{f_1}}({ x},\xi)$ be  the
principal symbol of ${f_1}$  at $({ x},\xi)$ and $\sigma^{(p)}_{w_1}({ y},\eta)$ be the principal symbol of $w_1$ at $({ y},\eta)\in  \Lambda_1$.
Then 
\begin{align}\label{eq for principal symbols}
\sigma^{(p)}_{w_1}({ y},\eta)=
R({y},\eta,{ x},\xi)\sigma^{(p)}_{{f_1}}({ x},\xi)
\end{align}
where  $R=R({ y},\eta,{ x},\xi)$ is an invertible linear operator (or,
a non-zero scalar number if Maslov line bundle structure is omitted).
Moreover, the function $R$ is independent of the coefficient $B$.
 
\end{lemma}
\begin{proof} The lemma follows from the proof of Lemma 3.1 of \cite{KLU2018}, but we recall the essential arguments of the proof.
Since the Schwartz kernel $Q$ of $(\square_{g} + B)^{-1}$ satisfies $Q \in I^{-\frac{3}{2},-\frac{1}{2}}(\Delta_{T^*{\Nextended}}',\Lambda_{{g}})$, then using $f_1\in I^{n+1}(\Sigma)$ we get $Qf_1\in I^{n-\frac{1}{2},-\frac{1}{2}}(\Sigma,\Lambda_1)$. Let $\pi:T^*\Next\to \Next$ be the projection to the base point of a covector.
Considering  the restriction of $w=Qf_1$ in the set  ${\Nextended}\setminus \pi(\Sigma)$, then using the formula (1.4) of \cite{GU1993}, we see that $w|_{{\Nextended}\setminus  \pi(\Sigma)}\in I^{n-\frac{1}{2}}(\Lambda_1)$.

In our case, because the manifold is globally hyperbolic, $\Theta_{x,\xi}\cap \Sigma$ contains only a single point and the formula \eqref{eq for principal symbols} for principal symbols follows from {\color{black} \cite{GU1993}},  Proposition 2.1 and is given precisely by
$$
\sigma^{(p)}_{w_1}({ y},\eta)=
\sigma(Q)(x,\xi;y,\eta)\sigma^{(p)}_{{f_1}}({x},\xi),
$$
where $(x,\xi)\in \Theta_{x,\xi}\cap \Sigma\subset T^*{\Nextended}$.
The function $\sigma(Q)$ does not vanish and one can consider it as a non-zero scalar. Finally,   the principal symbol $\sigma^{(p)}_{w_1}$
does not depend on the coefficient $B$, and thus $R$ is independent of $B$. 
\end{proof}

\begin{figure}[ht!]
\centering

\hspace{-40mm}
\includegraphics[height=2.5cm]{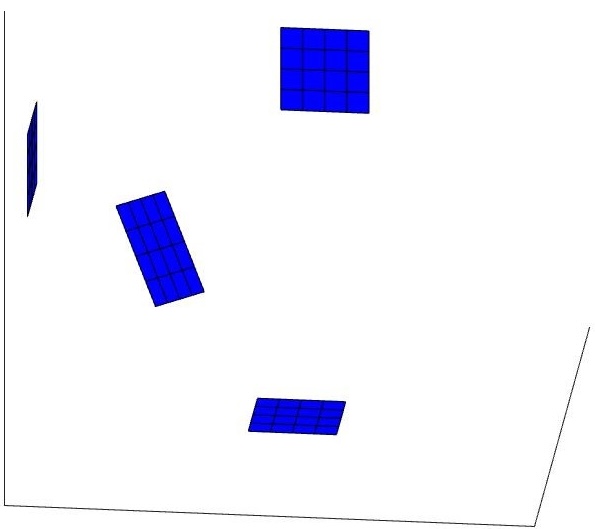}\hspace{2mm}
\includegraphics[height=2.5cm]{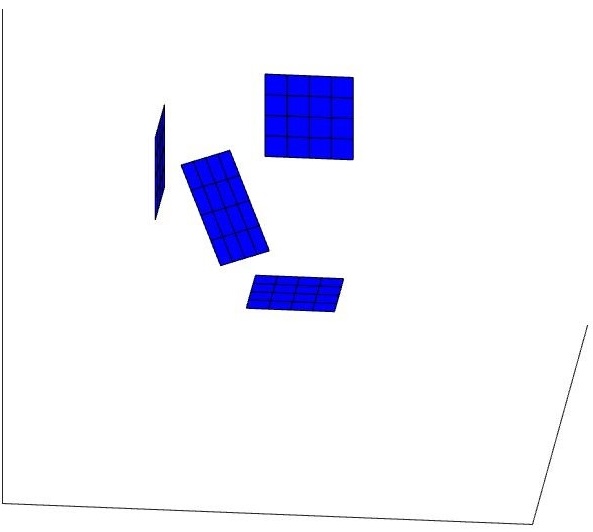}\hspace{2mm}
\includegraphics[height=2.5cm]{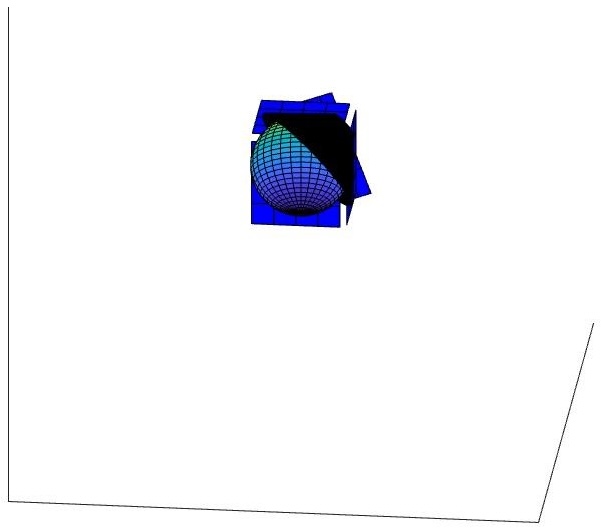}\hspace{2mm}
\includegraphics[height=2.5cm]{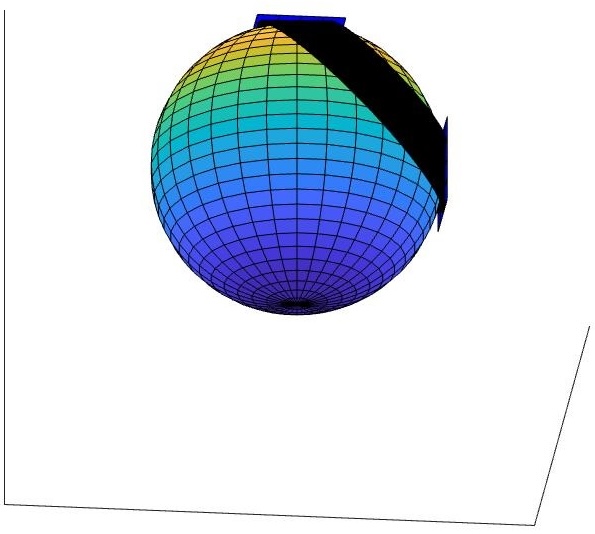}\hspace{-40mm}

\caption{
The directed conormal sources $\e_jf_j$, supported in $N_-$ near the past light-like infinity $\mathcal I^-$,
produce distorted plane waves $u_j$. The left and the center left figures show,
at  the times $t_1$ and $t_2$  the singular supports of such waves. 
When the four waves have
non-linear interaction, and produce a wave 
$\U^{({\kappa})}$, see formula \eqref{eq:4th_order_solutions},
that is similar to 
a spherical wave produced by a point source. The black, conic waves
are produced by 3-interactions.
The center right and the right figures show,
at  the times $t_3$ and $t_4$,  the singular supports the distorted plane waves  and 
the interaction waves.
Using scattering functionals, on can detect the interaction waves
at the future light-like infinity $\mathcal I^+$ and use these data to construct the conformal class of the metric $g$
near  $\mathcal I^+\cup  \mathcal I^-\cup i_0$.
}\label{fig: interaction of waves}
\end{figure}

\subsubsection{ $\kappa$th order interactions}
Let $x_j \in \Wminus$ and $\zeta_j\in L_{x_j}{\Nextended}$, {\color{black} $j=1,2,3,4$, and consider the sources
\beq\label{conormal souces 1}
f_j\in I^{n+1}(\Sigma(x_j,\zeta_j,s_0)),
\eeq
$n\in \R$, $n<-6$. Let $V=\bigcup_{j=1}^4\supp(f_j)$ and assume that $V\subset \Wminus$ and define $\vec{\e}=(\e_1,\e_2,\e_3,\e_4)$ and 
\beq\label{eq: epsilon source}
f_{\vec{\e}}=\sum_{j=1}^4\e_jf_j.
\eeq
 Let $u_{\vec\e}$
 be the solution of 
 \begin{equation}\label{eq:nonlinear wave equation in Penrose 3}
\begin{cases}
&(\square_{\gext}+B)u_{\vec\e} + Au_{\vec\e}^{\kappa}=f_{\vec{\e}},\quad
\text{in } I^-_\Next(p_+),\\
&\supp( u_{\vec\e})\subset J^+(\supp(f)).
\end{cases}
\end{equation}
 }
Moreover, assume {\color{black} that
\begin{equation}\label{eq:source causality}
\supp(f_j)\cap J^+(\supp(f_k)) = \emptyset,\quad\text{for all  $j\neq k$ and $\supp(f_j)\subset  
 \Wminus$},
\end{equation}
so} that the supports of the sources are causally independent.

These sources give rise to the solutions of the linearized wave equation, which we denote by
\begin{equation}\label{eq:linearized_solutions}
u_j := \partial_{\e_j} u_\vece|_{\vece=0} = (\square_{g}+\sct)^{-1}f_j \in
 I({\Nextended}\setminus\{ x_j \}; \Lambda(x_j,\zeta_j,s_0)).
\end{equation}
We will use the following abbreviations: $\partial_\vece^1u_\vece|_{\vece=0} := \p_{\e_1}u_\vece|_{\vece=0}$, $\p_\vece^2u_\vece|_{\vece=0} := \p_{\e_1}\p_{\e_2}u_\vece|_{\vece=0}$, $\partial_\vece^3u_\vece|_{\vece=0} := \p_{\e_1}\p_{\e_2}\p_{\e_3}u_\vece|_{\vece=0}$ and
\[
\p_\vece^4 u_\vece|_{\vece=0} := \p_{\e_1}\p_{\e_2}\p_{\e_3}\p_{\e_4}u_\vece|_{\vece=0}
\]
 and
\beq\label{Dkappa}
D^{\kappa} u_\vece|_{\vece=0} := \p_{\e_1}\p_{\e_2}\p_{\e_3}\p_{\e_4}^{{\kappa}-3}u_\vece|_{\vece=0}.
\eeq

The result of the fourth-order interactions produced by the waves $u_j$ for the non-linear wave equation,
see Figure \ref{fig: interaction of waves}, will be denoted by
\begin{equation}\label{eq:4th_order_solutions}
{\color{black}\U^{({\kappa})}} :={\color{black}D^{\kappa}} u_\vece|_{\vece=0} = (\square_{{g}}+\sct)^{-1} \Source,
\end{equation}
where
\begin{align}\label{eq:4th_order_source_S}
\Source &:={\color{black} -{\kappa} !\cdot A u_1u_2u_3(u_4)^{\kappa-3}.}
\end{align}

Let now $v_{\vec\eps}$ be a solution to $\square_{g} v_{\vec\eps} + Bv_{\vec\eps} = f_{\vec\eps}$, where $f_{\vec\eps}=\sum_{i=1}^4 \eps_i f_i$.
By linearity $v_{\vec\eps}=\sum_{i=1}^4 \eps_iu_i$, where $\square_{g} u_i+Bu_i = f_i$.
This also implies that for sufficiently small $\vec \e$, 
$$
\Vert u_{\vec\eps}\Vert_{H^\muutos{\tau}({\Nextended})} \leq C\sum_{i=1}^4 \eps_i \Vert f_i\Vert_{H^\muutos{\tau}({\Nextended})},
$$
where $\tau>-n-2\ge 4.$
Now $\square_{g} (u_{\vec\eps}-v_{\vec\eps}) + B(u_{\vec\eps}-v_{\vec\eps}) + A
{\color{black} u^{{\kappa}}_{\vec\eps}} = 0$, which yields
\begin{equation}\label{eq:equation for u}
u_{\vec\eps}=v_{\vec\eps}-Q(A{\color{black} u^{{\kappa}}_{\vec\eps}}),
\end{equation}
where $Q$ is the causal inverse of $\square_{g} + B$.

By a direct calculation, one sees that
$$
{\color{black} u^{{\kappa}}_{\vec\eps}} = (v_{\vec\eps}-Q(A{u^{{\kappa}}_{\vec\eps}))^{{\kappa}}} ={\color{black} v^{{\kappa}}_{\vec\eps}} + R,
$$
where $R$ contains all terms which are $O(\eps_1^{k_{1}}\eps_2^{k_{2}}\eps_3^{k_{3}}\eps_4^{k_{4}})$, $k_1+k_2+k_3+k_4>\kappa$ in the space
$H^\muutos{\tau}({\Nextended})$. %
Furthermore, substituting this back to \eqref{eq:equation for u} yields
$$
u_{\vec\eps} = v_{\vec\eps}-Q(A{\color{black} v^{{k}}_{\vec\eps}}) + R.
$$
When we compute the  derivatives of this equation with respect to $\e_j$:s, we obtain, see \cite{LLPT2} for the details on computing the derivatives \eqref{eq:4th_order_solutions}.
This shows that we only need to consider these interactions
{\color{black} of order ${{\kappa}}$}.
However, we need to keep in mind also regions where only three waves interact, since there can be problems in the symbol calculus in this case.

\begin{definition}\label{def:regular intersection}
The geodesics corresponding to $(\vec x,\vec\xi)=((x_j,\xi_j))_{j=1}^4$ intersect and the intersection takes place at $q\in {\Nextended}$ if there are $t_j>0$ such that $q=\gamma_{x_j,\xi_j}(t_j)$ for all $j=1,2,3,4$.
The intersection is \emph{regular} if $t_j \in (0,\rho(x_j,\xi_j))$ and the vectors $\dot\gamma_{x_j,\xi_j}(t_j)\in T_q{\Nextended}$, $j=1,2,3,4$, are linearly independent.
\end{definition}
Now, for $q\in {\Nextended}$ let ${\color{black} \Lambda_{\text{int}}}$ be the Lagrangian manifold
\begin{equation}\label{eq:set_Lambda_q}
\begin{split}
{\color{black} \Lambda_{\text{int}}} := \{
(y,\eta)\in T^*{\Nextended}\mid y&=\gamma_{q,\zeta}(1),\,\eta^\sharp = r\dot\gamma_{q,\zeta}(1),
\
\zeta\in L_q^+{\Nextended},\, r\in \R\setminus\{0\}
\}.
\end{split}
\end{equation}
Then the projection $\pi({\color{black} \Lambda_{\text{int}}})$ of ${\color{black} \Lambda_{\text{int}}}$ on ${\Nextended}$ is the light-cone $\mathcal L^+(q)\subset \Next$ emanating from $q$.
Let us take four points satisfying
\begin{equation}\label{eq:causally unrelated points}
x_j\in \Wminus\quad \hbox{and}\quad x_j\not\in J^+(x_k),\text{ for } j\neq k,\quad j,k=1,2,3,4.
\end{equation}
Let $\xi_j\in L_{x_j}^+{\Nextended}$ and denote $(\vec x,\vec\xi):=(x_j,\xi_j)_{j=1}^4$.
Let
\begin{equation}\label{eq:good set N}
\goodN(\vec x,\vec\xi) := {\Nextended} \setminus \cup_{j=1}^4 J^+(\gamma_{x_j,\zeta_j}(\mathbf{t}_j)),\quad
\text{where } \mathbf{t}_j:= \rho(x_j,\zeta_j).
\end{equation}
Denote also
\begin{equation}\label{eq:K_j and Lambda_j}
K_j {\color{black}(s_0)}:= K(x_j,\xi_j,s_0)=\pi(\Lambda(x_j,\xi_j,s_0)),\quad \Lambda_j := \Lambda(x_j,\xi_j,s_0)
\end{equation}
similarly to  \eqref{eq:set_K} and \eqref{eq:sets_Sigma_Lambda}.
Here we choose $s_0>0$ so small that either
\begin{equation}\label{A}
 {\color{black}K_{1234}=\bigcap_{s_0>0}}\left(\bigcap_{j=1}^4K_j {\color{black}(s_0)}\right)\cap \goodN(\vec x,\vec\xi)
 =\emptyset,\tag{A}
\end{equation}
or
\begin{equation}\label{B}
\begin{split}
& {\color{black}K_{1234}=\bigcap_{s_0>0}}\left(\bigcap_{j=1}^4K_j {\color{black}(s_0)}\right)\cap \goodN(\vec x,\vec\xi),\\
&{\color{black} q = \gamma_{x_j,\xi_j}(t_j)\in K_{1234}},\quad t_j>0 \text{ for all } j=1,2,3,4.
\end{split}
\tag{B}
\end{equation}
Moreover, let us define condition (T) as
\begin{equation}\label{T}
\begin{split}
&\text{the condition $(B)$ is valid with $q\in K_{1234}$}\\
&\text{and there exists } b_j=\dot \gamma_{x_j,\xi_j}(t_j)\in N_qK_j\setminus\{0\}\\
&\text{such that } b_j,\, j=1,2,3,4 \text{ are linearly independent.}
\end{split}
\tag{T}
\end{equation}

Roughly speaking, in case \eqref{A} the geodesics $\gamma_{x_j,\xi_j}$, $j=1,2,3,4$ do not intersect in $\goodN(\vec x,\vec\xi)$ and in \eqref{B} they intersect at a single point $\{q\}\in\goodN$.
We will use distorted plane-waves that propagate on the surfaces $K_j$. Due to the {\color{black}
$\kappa$:th order non-linearity with $\kappa\ge 4$}, we avoid dealing with the $3$-wave interactions of waves as they disappear in the linearized equations. However, while we do not have singularities produced by three waves, the sets where these $3$-wave singularities would propagate can cause some problems.

Due to this, we define the following sets analogously to \cite{KLU2018}.

Let $\X((\vec x,\vec\xi),s_0)\subset L^*{\Nextended}$ be the set of all light-like co-vectors $(x,\xi)$ belonging to the conormal bundles $N^*(K_{j_1} {\color{black}(s_0)}\cap K_{j_2} {\color{black}(s_0)} \cap K_{j_3} {\color{black}(s_0)})$ with $1\leq j_1 < j_2 < j_3 \leq 4$. More precisely, let $K_{j_1,j_2,j_3}(s_0)=K_{j_1}(s_0)\cap K_{j_2}(s_0)\cap K_{j_3}(s_0)$ and 
\begin{equation}\label{X-sets}
\begin{split}
&\X_{j_1j_2j_3}((\vec x,\vec\xi),s_0) =N^*K_{j_1,j_2,j_3}(s_0)\cap L^*{\Nextended}.
\end{split}
\end{equation}
Observe that $K_{j_1,j_2,j_3}(s_0)\cap \goodN(\vec x,\vec\xi)  $
is a smooth surface whose Hausdorff dimension  is $(3+1)-3=1$. 
For each $x\in K_{j_1,j_2,j_3}(s_0)\cap \goodN(\vec x,\vec\xi) $, the set $N_x^*K_{j_1,j_2,j_3}(s_0)\cap L^*{\Nextended}$ 
is of Hausdorff dimension $2$.
At the limit
$s_0\to 0$ the sets $K_j(s_0)$ tend towards the light-like geodesics $\gamma_{x_j,\xi_j}$.
Thus the submanifold
$\X_{j_1j_2j_3}((\vec x,\vec\xi),s_0)\cap T^*\goodN(\vec x,\vec\xi) \subset T^*M$ has Hausdorff dimension  $3$ and when $s_0\to 0$
these submanifolds converge to a submanifold $\X_{j_1j_2j_3}(\vec x,\vec\xi)\cap T^*\goodN(\vec x,\vec\xi) \subset T^*M$ of Hausdorff dimension  $2$.

Recall that $\Theta_{x,\xi}$ is the bicharacteristic of $\square_{g}$ containing $(x,\xi)$.
To define the sets of three wave interactions, we define
\begin{equation*}
    \begin{split}
\H_{j_1j_2j_3}((\vec x,\vec\xi),s_0) &=
\{ (y,\eta)\in T^*{\Nextended}\mid \text{there is } (x,\zeta)\in \X_{j_1j_2j_3}((\vec x,\vec\xi),s_0)\\
&\qquad \text{ such that } x\leq y \text{ and } (y,\eta)\in\Theta_{x,\zeta}\}.
\end{split}
\end{equation*}
Roughly speaking, the sets $\X_{j_1j_2j_3}$ are the light-like directions related to the three wave interactions and $\H_{j_1j_2j_3}$ are  the sets to which $\X_{j_1j_2j_3}$ flow under the bicharacteristic flow.
We also denote
\begin{equation}\label{XX sets}
\begin{split}
\H(\vec x,\vec\xi)& = \bigcap_{s_0>0}\bigg( \bigcup_{1\leq j_1 <j_2<j_3\leq 4} \H_{j_1j_2j_3}((\vec x,\vec\xi),s_0)\bigg).
\end{split}
\end{equation}
We use the above  sets to discard the possible singularities produced by three waves interactions.
Moreover, these sets allow us to define $\Y((\vec x,\vec\xi),s_0) = \pi(\H((\vec x,\vec\xi),s_0))$, where $\pi : T^*{\Nextended}\to {\Nextended}$ is the projection to the base space on the cotangent bundle.
The limiting spaces are defined as
\begin{equation}\label{XYH sets}
\begin{split}
&\Y(\vec x,\vec\xi) = \bigcap_{s_0>0}\Y((\vec x,\vec\xi),s_0),
\qquad \X(\vec x,\vec\xi) = \bigcap_{s_0>0}\Y((\vec x,\vec\xi),s_0).
\end{split}
\end{equation}
We call the sets $\Y(\vec x,\vec\xi)$ the exceptional sets of 3-wave interactions.

{Recall that the  submanifold $\X_{j_1j_2j_3}((\vec x,\vec\xi))\cap T^*\goodN(\vec x,\vec\xi) \subset T^*M$ has the Hausdorff dimension  $2$.
The main point of these sets is that at the limit $s_0\to 0$ the sets $\Y((\vec x,\vec\xi),s_0)\cap \goodN(\vec x,\vec\xi) \subset T^*M$ tend to the set $\Y((\vec x,\vec\xi))\cap \goodN(\vec x,\vec\xi) \subset T^*M$, whose Hausdorff dimension is at most $2$ and  the sets $\H((\vec x,\vec\xi),s_0)\cap T^*\goodN(\vec x,\vec\xi) \subset T^*M$ tend to the set $\H((\vec x,\vec\xi))\cap T^*\goodN(\vec x,\vec\xi) \subset T^*M$, whose Hausdorff dimension is at most $3$.
Later, these sets can be discarded in the reconstruction procedure.} We will not analyze what happens on these exceptional sets. We refer the reader to \cite{FLO} for a study where the three wave interactions are used also to prove uniqueness results for inverse problems.

We will also need small neighbourhoods of the sets $\Lambda_j$.
Recall that on the manifold ${\Nextended}$ we have an auxiliary Riemannian metric $g^+$.
Using $g^+$ we may define the unit cotangent bundle $S^*{\Nextended}$, which further allows us to consider conic $\eps$-neighbourhoods $\Gamma_j(\eps)$ of $\Lambda_j\cap S^*{\Nextended}$.
Let us denote
\begin{equation}\label{eq:gamma_ngbh}
\begin{split}
\tilde\Gamma(\eps)&:=
\left(
\bigcup_{j<k<l} (\Gamma_j(\eps) + \Gamma_k(\eps) + \Gamma_l(\eps))
\right)\\
&\qquad\cup
\left(
\bigcup_{j<k} (\Gamma_j(\eps) + \Gamma_k(\eps))
\right)
\cup
\left(
\bigcup_{j=1}^4 \Gamma_j(\eps)
\right).
\end{split}
\end{equation}
This set contains $\X((\vec x,\vec\xi))$.
Moreover, let
\begin{equation}\label{eq:H_ngbh}
\H(\eps) := \Lambda_{g}'\circ (\tilde\Gamma(\eps) \cap \mathrm{char}(\square_{g}))
\end{equation}
be the Hamiltonian flow-out of $\tilde\Gamma(\eps)$ (given by the canonical relation of $\square_{g}$).
Then $\H(\eps)$ is a neighbourhood of $\H((\vec x,\vec\xi))$.

Next we use a generalization of the analysis obtained in \cite{KLU2018} and \cite{LUW} 

First, to analyze the $\kappa$-th order non-linearity, we use the following result:

\begin{lemma}\label{lem: higher power}

Let $K\subset {\Nextended}$ be a codimension one submanifold. Let $u_j\in I^{\mu_j}(K)$, $j=1,2,\dots,J$,
 $\mu_j<-\frac{3}{2}$. Then  $v=\prod_{j=1}^J u_j(x)$ is a well-defined distribution in $I^{\nu}(K)$ with $\nu=\frac 32(J-1) +\sum_{j=1}^J \mu_j$. 
 Moreover, the principal symbol satisfies
\beq
\sigma(v) = (2\pi)^{-J/2}\sigma(u_1)\ast\sigma(u_2)\ast \dots \ast\sigma(u_J) 
\eeq
where  the convolution is over the fiber variable of $N^*K$, i.e.,
\begin{equation}
\begin{split}
&\sigma(v)(x, \xi) = \\
& \int_{\R^{J-1}} 
\sigma(u_J) (x,\xi-\sum_{j=1}^{J-1}\eta_j)\cdot \prod_{j=1}^{J-1} \sigma(u_j)(x,\eta_j)\,
d\eta_1\dots d\eta_{J-1}
\end{split}
\end{equation}
where $(x, \xi) \in K\times (\R\setminus 0)$ and we identify $K\times (\R\setminus 0)$
with $N^*K$.
Moreover,
the product $v$ satisfies $\WF(v)\subset N^*K $.
\end{lemma}

\begin{proof} The proof follows by iterating Lemma 5.1 in  \cite{LUW}.

\end{proof}

  \begin{lemma}
\label{thm:analytic limits A wave A and B} 
  Let $(\vec x,\vec \xi)=((x_j,\xi_j))_{j=1}^4$  be future pointing light-like vectors  
  such that \eqref{eq:causally unrelated points} is satisfied. 
  Assume also that $s_0>0,K_j,\Lambda_j$ 
are as in \eqref{eq:K_j and Lambda_j}. 
   Let ${ y_0}\in  \goodN(\vec x,\vec \xi)\cap \Wplus$,
see \eqref{eq:good set N}.
    
Let  $n\in \Z_+$  and
${f}_{j}\in \mathcal I^{-n+1}({ \Sigma}(x_j,\zeta_j,{ s_0}))$, $j=1,2,3,4$, 
be sources satisfying \eqref{eq:source causality} and
$u_j= (\square_{g}+\sct)^{-1}f_j$ and $\U^{({{\kappa}})}$ be 
 the wave produced by the  ${{\kappa}}$:th order interaction
given in \eqref{eq:4th_order_solutions}. 
When $n$ is large enough, $s_0$ is small enough,
the following claims hold:

(a) 
When the above condition \eqref{A} is satisfied, {then  $(y_0,\zeta_0)\not \in WF(\U^{({{\kappa}})})$.
Moreover, if  $y_0 \not\in \bigcup_{j=1}^4 \gamma_{x_j,\xi_j}([0,\infty))\cup \Y((\vec x,\vec\xi),s_0)$,
see \eqref{eq:good set N} and \eqref{XYH sets},  the point $y_0$ has a neighborhood $U\subset \Wplus$ such that  $\U^{({{\kappa}})}|_U$ is $C^\infty$-smooth.}

(b)
When the above conditions \eqref{B} and \eqref{T} are satisfied,  the following holds:
\begin{enumerate}[(i)]
\item\label{item:y not in J implies smooth}  If $ y_0\not \in {\mathcal L^+(q)},$
 then ${ y_0}$ has a neighborhood $V$ such that $\U^{({{\kappa}})}|_V$ is $C^\infty$-smooth.
\item\label{item:y in J implies Lagrangian} Assume that $y_0\in  {\mathcal L^+(q)} \setminus \Y((\vec x,\vec\xi),s_0)$.
Also, assume that
$w_0\in L^{*}_{y_0}{\Nextended}$ and $r\in \R$ are such that
 $\gamma_{y_0,w_0}(r)=q$ and denote $ \eta=(\dot \gamma_{y_0,w_0^\sharp}(r))^\flat\in  {\color{black}\Lambda_{1234}}$. 
Then the point $y_0$ has a neighbourhood $V$ such that $\U^{({{\kappa}})}$ in $V$ is a Lagrangian distribution, $\U^{({{\kappa}})}\big|_{V}\in I^{-4n-\frac{1}{2}}({\color{black} \Lambda_{\text{int}}})$.

Moreover, $\eta$ can be written as $\eta=\sum_{j=1}^4 \zeta_j$, where $\zeta_j\in N^*_qK_j$ are linearly independent.
Then the principal symbol of $\U^{({{\kappa}})}|_{V}\in I^{-4n-\frac{1}{2}}( {\color{black}\Lambda_{1234}})$, at the point  $({ y_0},{ w_0})$, {\color{black} is  
\begin{multline}\label{principal symbol of cal U}
\sigma^{p}_{\U^{({{\kappa}})}}(y_0,w_0) =
-{{\kappa}} !(2\pi)^{-3}R(y_0,w_0,q,\eta)\cdot A(q)(\prod_{j=1}^3\sigma^{p}_{u_j}(q,\zeta_j))\cdot \sigma^{p}_{v}(q,\zeta_4)),
\end{multline}
where $v=u_4^{{{\kappa}}-3}$ and} ${\vec \zeta}=(\zeta_j)_{j=1}^4$ and $R({ y_0},{ w_0},q,{  \eta} )$ is given Lemma~\ref{lem: Lagrangian 1 wave}.
\end{enumerate}
\end{lemma}

We remark that as by Lemma~\ref{lem: Lagrangian 1 wave} the principal symbols of $u_j$ does not depend on the coefficient $B$, we see that the principal 
symbol of $\U^{(\kappa)}$ given in \eqref{principal symbol of cal U} does not depend on $B$.
\begin{proof}
In the proof of both cases (A) and (B) we use the fact that we consider observations at the point
${ y_0}\in  \goodN(\vec x,\vec \xi)$ and thus  the point $y_0$ has a neighborhood $V_0$
such that in the chronological past $I^-(V_0)$
the linearized waves $u_j\in \I(K_j)$ are conormal distrubutions associated to smooth
submanifolds $K_j\cap I^-(V_0)$. 

The case (a) follows from Prop. 5.6 in \cite{LUW} (see also
Theorem 3.3 of \cite{KLU2018} about the detailed analysis for 2nd order non-linearity).

Next we prove the case (b). First, we
consider the claim (ii).
As in the proof of Lemma~\ref{lem: Lagrangian 1 wave}, we start by recalling that the causal inverse $Q=(\square_{g} + B)^{-1} \in I^{-\frac{3}{2},-\frac{1}{2}}(\Delta_{T^*{\Nextended}}^\prime, \Lambda_{g})$.
Therefore, for $f_j\in I^{-n+1}(\Sigma_j)$ we have that $u_j:= Qf_j \in I^{-n-\frac{1}{2},-\frac{1}{2}}(\Sigma_j,\Lambda_j)$.

By  Prop. 5.6 in \cite{LUW} (see also  Prop. 3.11, claim (i) in \cite{LUW} on the detailed
analysis when ${{\kappa}}=4$),
 we know that the product $u_1^{{{\kappa}}-3}u_2u_3u_4$ can be expressed as
\begin{equation*}
\begin{cases}
u_1^{{{\kappa}}-3}u_2u_3 {\color{black} v} = w_0 + w_1,\\
w_0|_{\M\setminus \Y((\vec x,\vec\xi),s_0)} \in I^{-4n+1+{\color{black} ({{\kappa}}-4)(1-n)}}(\Lambda_{1234}), \quad w_1\in \mathcal{D}'({\Nextended}),\\
\WF(w_1) \subset 
\tilde\Gamma(\eps),
\end{cases}
\end{equation*}
where $\tilde\Gamma(\eps)$ is given by \eqref{eq:gamma_ngbh}.

By H\"ormander's theorem about propagation of singularities \cite[Theorem 26.1.1]{H4}
$$
\WF(-{{\kappa}}!Q(Aw_1)) \subset \tilde\Gamma(\eps) \cup \H(\eps),
$$
where $\H(\eps)$ is as in \eqref{eq:H_ngbh}.
Noting that $\pi(\H(\eps))$  contains an $\eps$-neighbourhood of $\Y((\vec x,\vec\xi))$ yields for small enough $\eps>0$ that $y_0$ has a neighbourhood $W$ such that $W\cap \H(\eps)=\emptyset$.
Finally, applying $Q$ to $-{\color{black} {{\kappa}}}!Aw_0$ yields
that in the set $\M\setminus \Y((\vec x,\vec\xi),s_0)$
$$
 -{\color{black} {{\kappa}}}!Q(Aw_0)\in I^{-4n+1+{\color{black} ({{\kappa}}-4)(1-n) -\frac{3}{2}},-\frac{1}{2}}(\Lambda_{1234},{\color{black} \Lambda_{\text{int}}}),
$$
where 
{\color{black}
$\Lambda_{\text{int}}= \Lambda_{g}'\circ (\Lambda_{1234}\cap \mathrm{char}(\square_{g}))$ is the Hamiltonian flow-out of $\Lambda_{1234}$. The notation ${\color{black} \Lambda_{\text{int}}}$ refers
to the fact that this Lagranian submanifold is associated to the interaction of waves.}
Hence, assuming $A(q) \neq 0$, we have 
$$
\U^{({{\kappa}})}\big|_{\Wplus\setminus \Y(\vec x,\vec\xi)}\in I
({\color{black} \Lambda_{\text{int}}}).
$$
The claim about the principal symbol $\sigma_{\U^{({{\kappa}})}}^\mathrm{p}$ follows from 
Propositions 3.12(i) and 5.6 in  \cite{LUW}
and Lemma~\ref{lem: Lagrangian 1 wave}.

For claim (i), note that $(y_0,w_0)\in \WF(\U^{({{\kappa}})})$ if either $w_0$ is not light-like and $(y_0,w_0)\in \WF(\Source)$ or $w_0$ is light-like and $(\gamma_{y_0,w_0^\sharp}(s),\gamma_{y_0,w_0^\sharp}(s)^\flat)\in \WF(\Source)$, for some $s\in\R$.
By the above considerations, we know that
$\WF(\U^{({{\kappa}})})\subset {\color{black} \Lambda_{\text{int}}}\cup \tilde\Gamma(\eps)\cup \H(\eps)$.
It follows from \eqref{eq:set_Lambda_q} that $\pi({\color{black} \Lambda_{\text{int}}}) = \mathcal L^+(q)$, so when $y_0\not\in J^+(q)$, we have that $(y_0,w_0)\not\in {\color{black} \Lambda_{\text{int}}}$.
On the other hand, we assumed $y_0\not\in \Y(\vec x,\vec\xi)$.
Taking small enough $\eps>0$ it holds that $(y_0,w_0)\not\in \H(\eps)$.
So $y_0$ is not in the wave-front set of $\U^{({{\kappa}})}$ and hence it has a neighbourhood $V$ where $\U^{({{\kappa}})}$ is smooth.
\end{proof}

\begin{lemma}\label{lem: Eq construction} 
Let $x_j\in  \Wminus$, $j=1,2,3,4$ and   
$\xi_j\in T_{x_j}N_{\ext}$ be future directed light-like vectors, and consider geodesics  $\gamma_{x_j,\xi_j}(\R_+)$.
Also, let $t_0\ge 0$ be such 
that $\gamma_{x_j,\xi_i}([0,t_0])\subset  {{N}}^-$.
Then using the  extended source-to-solution operator $L_{g_{\ext},p_+}$ we can determine a set $S=S(\vec x,\vec \xi,t_0)$ having the following properties:

(i) If all  four geodesics $\gamma_{x_j,\xi_j}([t_0,\infty))$,  $j=1,2,3,4$ intersect at a point $q$ before the first cut point
of any of these geodesics {\mnewtext  and $A(q)\not = 0$}, then  $S=\mathcal E_V(q)$, where $V= \Wplus$.

(ii) If all  four geodesics $\gamma_{x_j,\xi_j}([t_0,\infty))$,  $j=1,2,3,4$ do not intersect before the first cut point
of any of these geodesics {\mnewtext  or they intersect at the point $q$ having
a neighborhood $V_q$ where $A|_{V_q}=0$}, then $S\subset  \Wplus\setminus \mathcal{N}((\vec x, \vec \xi), t_0)$.
\end{lemma}

\begin{proof}

Consider $\vec x=(x_j)_{j=1}^4,$  $x_j\in {{N}}^-$  and $\vec \xi=(\xi_j)_{j=1}^4,$  where
$\xi_j\in T_{x_j}N_{\ext}$ are future directed light-like vectors. Also, let $t_0>0$ be so small
that $\gamma_{x_j,\xi_i}([0,t_0])\subset  {{N}}^-$.
Similarly to  
\cite[Section 3.5]{KLU2018},
we say that a point $y\in  \Wplus$,
   satisfies the singularity {\it detection condition} ($D_{}$)  
with light-like directions $(\vec x,\vec \xi)$, 
   and $t_0,\hat s>0$
  if
  \medskip  
  
  \noindent
(${\bf D}_{}$) {\it For any $s,s_0\in (0,\hat s)$ and   $j=1,2,3,4$
there exists $(x_j^{\prime},\xi_j^{\prime})$
 in the
$s$-neighborhood of $(x_j,\xi_j)$, open sets $B_j\subset B_{g^+}(\gamma_{x_j^{\prime},\xi_j^{\prime}}(t_0),s)$, satisfying $B_j\cap J ^+(B_k)=\emptyset$ for 
$j\not =k$, such that the following is valid: There are  ${f}_{j}\in { I}^{\mu+1}(Y((x_j^{\prime},\xi_j^{\prime});t_0,s_0))$ 
such that $\supp(f_j)\subset B_j$, the wavefront set
of $f_j$  is in the $s_0$-neighborhood of  $(x_j^{\prime},\xi_j^{\prime})$, and  for the solution 
 $u=u_{\vec \epsilon}$ 
 of \eqref{eq:nonlinear wave equation in Penrose 3}
with the source ${f}_{\vec \epsilon}=\sum_{j=1}^4
\epsilon_j{f}_{j}$
it holds that
$\mcu^{({{\kappa}})} = \p_{\eps_1}^{{{k}}-3}\p_{\eps_2}\p_{\eps_3}\p_{\eps_4}    u_{\vec \epsilon}|_{\{\eps_1=\eps_2=\eps_3=\eps_4 = 0\}} $
is not
$C^\infty$-smooth at $y$.}
  \medskip

We see that 
if geodesics $\gamma_{x_j,\xi_j}(\R_+)$ intersect at a point 
$q\in \mathcal{N}((\vec x, \vec \xi), t_0)$ {\mnewtext  where $A(q)\not =0$},
then arbitrarily close to 
 $(x_j,\xi_j)_{j=1}^4$ in $(L^+N_{\ext})^4$ there are  $(x_j',\xi_j')_{j=1}^4$ 
such that the geodesics $\gamma_{x_j',\xi_j'}(\R_+)$ intersect at the point $q$
and also the condition (T) is valid.

Thus Lemma~\ref{thm:analytic limits A wave A and B} implies that 
 the set
\begin{equation}
\begin{split}
S(\vec x,\vec \xi,t_0)
:=&\{y\in \Wplus\mid \hbox{there is  $\hat s>0$ such that }\\
&\qquad\hbox{$y$ satisfies (D) with $(\vec x,\vec \xi)$ and 
$t_0,\hat s$}\}
\end{split}
\end{equation}
has the property that 
\begin{multline*}
S(\vec x,\vec \xi,t_0)\cap \bigg(\mathcal{N}((\vec x, \vec \xi), t_0) \setminus (\Y^{(3)}\cup \bigcup_{k=1}^4 K_j)\bigg)
\\
=
 \mcl^+(q)\cap  \Wplus\cap \bigg(\mathcal{N}((\vec x, \vec \xi), t_0) \setminus (\Y^{(3)}\cup \bigcup_{k=1}^4 K_j) \bigg),
\end{multline*}
 where $\Y^{(3)}=\Y((\vec x,\vec\xi),s_0)$.
Roughly speaking, this means that if  the
 geodesics $\gamma_{x_j,\xi_j}(\R_+)$ intersect at $q$ before their first cut points, then
the linearized waves $v_j=Q_g f_j$ interact at the point $q$ and produce a wave $\mcu^{({{k}})}$  that in the set $\mathcal{N}((\vec x, \vec \xi), t_0)$ may be singular only on the future 
light cone $ \mcl^+(q)$  emanating from $q$. Moreover, at any point $y\in  \mcl^+(q)\cap  \mathcal{N}((\vec x, \vec \xi), t_0)$
the wave $\mcu^{({{k}})}$  is  non-smooth near $y$  if one makes a suitable perturbation to sources $f_j$.

Define 
 $S_{reg} (\vec x,\vec \xi,t_0)$ be the set  
of the points  $y\in S(\vec x,\vec \xi,t_0)$ having a neighborhood $W\subset  \Wplus$  
such that the intersection $W\cap S(\vec x,\vec \xi,t_0)$  
is a non-empty $C^\infty$-smooth 3-dimensional submanifold.  
Moreover, let  $S_{cl} (\vec x,\vec \xi,t_0)$ be the closure of the set $S_{reg} (\vec x,\vec \xi,t_0)$  in $ \Wplus$
and define  $S_{e} (\vec x,\vec \xi,t_0)$ to be the set of those $y\in S_{cl} (\vec x,\vec \xi,t_0)$
for which any  geodesics $\mu_a=(\R\times \{a\})\cap {{N}}^+$, $a\in \mathbb S^3$,   containing $y$  does not intersect $S_{cl} (\vec x,\vec \xi,t_0)$ in the chronological past of $y$.

The proof of Lemma 4.4 of \cite{KLU2018} 
 shows
the following: First, in the case when all four geodesics $\gamma_{x_j,\xi_j}(\R_+)$, $j=1,2,3,4$ intersect at some point 
$q\in \mathcal{N}((\vec x, \vec \xi), t_0)$, the above constructed set  $S_{e} (\vec x,\vec \xi,t_0)$ coincides
with $\mathcal E_V(q)$ with $V= \Wplus$. Second, in the case when all four geodesics $\gamma_{x_j,\xi_j}(\R_+)$ do not intersect 
at any point of
$\mathcal{N}((\vec x, \vec \xi), t_0)$  {\mnewtext  or they intersect at a point near which $A$
vanishes}, the above constructed set  $S_{e} (\vec x,\vec \xi,t_0)$ does not intersect 
$\mathcal{N}((\vec x, \vec \xi), t_0)$. This proves the claim for $S(\vec x,\vec \xi,t_0)=S_{e} (\vec x,\vec \xi,t_0)$.

\end{proof}

\section{Proof of the main theorem}\label{sec: proof of main thm}

\begin{proof} (of Theorem~\ref {main thm for general manifold})
Let us consider two manifolds $(M_1,g_{M_1})$ and $(M_2,g_{M_2})$ with asymptotically 
Minkowskian infinities that are visible in the whole manifold and $(N_1,g_1)$ and $(N_2,g_2)$ be conformal to them.
Let   $(N_{\ext}^1,\tilde g_1)$ and $(N_{\ext}^2,\tilde g_2)$  be the
`non-physical extensions' of $(N_1,g_1)$ and $(N_2,g_2)$ that are obtained
by gluing manifolds ${{N}}^+$ and ${{N}}^-$ to these spaces.


The proof consists of several steps; let us provide a brief overview of its structure. First, we construct neighbourhoods of $\nullinf^+$ and $\nullinf^-$, such that future directed null geodesics do not have conjugate points. This allows us to reconstruct the conformal type of these neighbourhoods of $\nullinf^\pm$. We then study a gauge-like conformal transformation in these neighbourhoods and show that in the neighbourhoods of $\nullinf^\pm$, the source-to-solution maps agree up to smoothing error of order one, at least locally. This allows us to consider source-to-solution maps (locally) in a neighbourhood of $\widehat\mu$. Having established that the source-to-solution maps of $N_\ext^1$ and $N_\ext^2$ agree in a neighbourhood of $\widehat\mu$, we are in the regime of~\cite{KLU2018}, and can finish the proof using local measurements.
\medskip

{\it {\bf Step 1:} Defining neighborhoods of $\mathcal I^-$ and $\mathcal I^+$ with no cut or conjugate points.}

We recall that $\hat \mu(s)$ is the path $\hat \mu(s)=(s,\SP)\in \R\times \mathbb S^3$, $s\in (-\pi,\pi)$.

On the both manifolds $(N_{\ext}^1,\tilde g_1)$ and $(N_{\ext}^2,\tilde g_2)$,
the metric tensor in the set $J^+(i_0)\cap J^-(p_{+2}) \subset N_{\ext}^j$
coincides with the product metric of $\R\times \mathbb S^3$. Let
 $\xi\in L^+_{i_0}N_{\ext}^j$ be a light-like vector at the point $i_0$, normalized so that $\gamma_{i_0,\xi}(\pi)=\NP$. Then,
  the 
cut locus functions on the manifolds $N_{\ext}^1$ and $N_{\ext}^2$ satisfy $\rho(i_0,\xi)=\pi$.
 Since the cut locus function $\rho(x,\xi)$ 
is lower semi-continuous and ${{{N}}}^+$ is isometric to a subset of
$\R\times \mathbb S^3$,  for any $h>0$ there is a neighborhood $V_h\subset L^+ N_{\ext}^j$
of  the point $(i_0,\xi)$ such that  $\rho(y,\eta)>\pi-h$ for all $(y,\eta)\in V_h$. 
Thus, if $h>0$ is small enough and $s_0^-<0$ is so close to zero that 
\[
J^+(p_0^-)\cap {{N}}^-\subset V_h,\quad p_0^-=\hat \mu(s_0^-)
\]
then the light-like geodesics $\gamma_{y,\eta}$ emanating from the points 
$y\in J^+(p_0^-)\cap {{N}}^-\subset  N_{\ext}^j$ and $\eta\in L^+ N_{\ext}^j$
have no cut points in the set $J^-(p_{+2}) \subset N_{\ext}^j$.
We denote (see Figure 9 below)
\[
W_j(s_0^-):=I^+_{N_{\ext}^j}(p_0^-)\cap I^-_{N_{\ext}^j}(p_{+{}}).
\]  
and
\[
Y_j(s_0^-):=W_j(s_0^-)\setminus  ({{N}}^+\cup {{N}}^-).
\]

Next, for fixed $j\in\{1,2\}$, let us consider $x\in \mathcal I^-\cap J^+(p_{-2})\subset {N_{\ext}^j}$ and light-like vectors
$\xi\in L^+_xN_{\ext}^j$. We define 
\ba
& &h_0(x,\xi,r)=f_{\hat \mu}^+(\gamma_{x,\xi}(r)),\\
& &F_0(x,\xi)=h_0(x,\xi,\rho(x,\xi)).
\ea

The function $h_0$ is continuous and $r\to h_0(x,\xi,r)$
is  non-decreasing for all $(x,\xi)\in L^+N$. As $\rho(x,\xi)$ is lower semi-continuous,
we see that if $(x_n,\xi_n)\to (x,\xi)$ as $n\to \infty$, then
\begin{equation*}
\lim_{n\to \infty}h_0(x_n,\xi_n,\rho(x_n,\xi_n)) =
h_0(x,\xi,\lim_{n\to \infty} \rho(x_n,\xi_n))\ge h_0(x,\xi,\rho(x,\xi)).
\end{equation*}
Thus,  $F_0$ is also  a  lower semi-continuous function.

Next we make two observations.

First, let $(x,\xi)$ be a future-directed light-like vector in the tangent space of $ \mathcal I^-$ or
$x=i_0$ and $\xi\in L_{i_0}^+N_{\ext}^j$.
As  ${{{N}}}^-$ and ${{{N}}}^+$ are isometric to subsets of
$\R\times \mathbb S^3$, we see that $F_0(x,\xi)=h_0(x,\xi,\rho(x,\xi)))>0$.

Second, if $x\in  \mathcal I^-$ and $\xi\in L_x^+N_{\ext}^j$  is not tangent to $ \mathcal I^-$,
then for all $r>0$ we have $y=\gamma_{x,\xi}(r)\in I^+({\mathcal I^+})$,
and thus $f_{\hat \mu}^+(y)>0$. Hence,  if $x\in  \mathcal I^-\cup \{i_0\}$ and $\xi\in L_x^+N_{\ext}^j$, then $F_0(x,\xi)=h_0(x,\xi,\rho(x,\xi))>0$.
By using compactness of $\hbox{cl}(\mathcal I^-\cap J^+(p_{-2}))$,
we see that 
\beq\nonumber
s_{{0}}^+:=\frac 14 \min\{ h_0(x,\xi,\rho(x,\xi)) & | &  x\in \hbox{cl}(\mathcal I^-\cap J^+(p_{-2})),
\\ \label{s0 plus}
&& (x,\xi)\in L^+N,
\ \|\xi\|_{g_+}=1\}>0.
\eeq
Then, every point $x\in \hbox{cl}(\mathcal I^-\cap J^+(p_{-2}))$
has a neighborhood $B_x\subset N_{\ext}^j$,
such that  for all $y\in B_x$ and $\eta\in L_y^+N_{\ext}^j$ it holds that
$h_0(y,\eta,\rho(y,\eta))>2s_{{0}}^+$. Let $p_{{0}}^+=\hat \mu(s_{{0}}^+)$.
Next, let us cover the compact set
 $ \hbox{cl}(\mathcal I^-\cap J^+(p_{-2}))$
with a finite number of open sets ${B}_k={B}_{x_k}$, $k=1,2,\dots,K$, see Fig.\ \ref{Fig 9.}.
Then, using compactness of $({{N}}^-\cap J^+(p_{-2}))\setminus 
\bigcup_{k=1}^K {B}_k\subset {N_{\ext}^j}$ and that  ${{{N}}}^+$ is isometric to a subset of
$\R\times \mathbb S^3$, we see that
\ba
s_{00}^-:=\max \{f_{\hat \mu}^+(x)&|&  x\in ({{N}}^-\cap J^+(p_{-2}))\setminus 
\bigcup_{k=1}^K {B}_k
\}<0.
\ea
Let  $p_{00}^-=\hat \mu(s_{00}^-)$. Then we see that 
for all $z\in {N_{\ext}^j}$ satisfying
\beq\label{z points}
z\in  ({{N}}^-\cap J^+(p_{-2}))\setminus I^-(p_{00}^-)\subset \bigcup_{k=1}^K {B}_k
\eeq
and $\zeta\in L^+_z{N_{\ext}^j}$, it holds that $h_0(z,\zeta,\rho(z,\zeta))>2s_{{0}}^+$. 
Then by \eqref{s0 plus} and \eqref{z points}, the light-like
geodesics $\gamma_{z,\zeta}$ do not have cut points
in the set $J^-(p_{{0}}^+)\subset {N_{\ext}^j}$.

We denote
\ba
& &X_j(s):=(I^-_{N_{\ext}^j}(\hat \mu(s))\cap I^+_{N_{\ext}^j}(p_{-{}}))\setminus ({{N}}^+
\cup {{N}}^-)\subset J^-(p_{{0}}^+),
\ea
so that
$$
X_j(s_{{0}}^+):=(I^-_{N_{\ext}^j}(p_{{0}}^+)\cap I^+_{N_{\ext}^j}(p_{-{}}))\setminus ({{N}}^+
\cup {{N}}^-)\subset J^-(p_{{0}}^+)
$$
see \eqref{s0 plus}, and
let
 $$ {{{Z}}}_j=({{N}}^-\cap I^+_{N_{\ext}^j}(p_{-{}}))\setminus  J^-_{N_{\ext}^j}(p_{{00}}^-).
 $$
Note that by \eqref{z points}, $ {{{Z}}}_j\subset \cup_{k=1}^K {B}_k.$
\medskip

{\it {\bf Step 2:} Construction of the conformal type of the  neighborhoods of $\mathcal I^-$ and $\mathcal I^+$.}
 
By using sources $f_k$ that are conormal distributions supported in 
${{{Z}}}_j$,
where $j=1,2$,
and the interaction of waves,
observed  in ${{N}}^+$, we apply Lemma~\ref{lem: Eq construction}. 
We see that the light-like geodesics emanating from points $(x_k,\xi_k)\in L^+ {{{Z}}}_j$, 
$x_k\in {{{Z}}}_j\subset  {{N}}^-$, 
$k=1,2,3,4$ do not have cut points on
$J^-(p_{{0}}^+)\subset {N_{\ext}^j}$.
Let now ${\tilde U}= I^-_{N_{\ext}^j}(p_{{0}}^+)\cap {{N}}^+$ and $S\subset {\tilde U}$ be the set constructed in
Lemma  \ref{lem: Eq construction}
 with the initial directions  $(x_k,\xi_k)_{k=1}^{4}$.
 If $S\cap I^-_{N_{\ext}^j}(p_{{0}}^+)\not = \emptyset$,
Lemma   \ref{lem: Eq construction}
implies that the geodesics $\gamma_{x_k,\xi_k}([0,\infty))$,  $k=1,2,3,4$ 
have to intersect at some point $q\in I^-_{N_{\ext}^j}(p_{{0}}^+){\mnewtext \cap \overline N}$
and the set $S$  satisfies 
$S=\mathcal E_{{{N}}^+}(q)$.
{\mnewtext  In the case when  $S\cap I^-_{N_{\ext}^j}(p_{{0}}^+) = \emptyset$,
 the four geodesics  do not have
common intersection points in the set  $ I^-_{N_{\ext}^j}(p_{{0}}^+){\mnewtext  \cap N}$}. 

Moreover, we see that for all $q\in X_j(s_{{0}}^+)$ there exists $(x_j,\xi_j)$ satisfying
 the above conditions such that the geodesics
 $\gamma_{x_j,\xi_j}$, $j=1,2,3,4$ intersect at $q$.

By computing the first Frechet derivatives  of 
source-to-solution maps, we obtain the source-to-solution map for the linearizated
wave equation. This map determines the earliest light observation sets $\mathcal E_U(q)$, where $U=N_+$,
for the points $q\in N_-$. Clearly, as $(N_+,g_{ext}|_{N_+})$ is a known Lorentzian manifold, we can determine also the sets
$\mathcal E_U(q)$ for  $q\in N_+$.

Using  the source-to-solution map for the linearizated
wave equation we can alsp determine the intersection $\gamma_{x,\xi}\cap {{N}}^+$ of
the set ${{N}}^+$ and the light-like geodesics  $\gamma_{x,\xi}$ emanating from
the points $x\in {{N}}^-$. Using the knowledge of the geodesic segments $\gamma_{x,\xi}\cap {{N}}^+$, $j=1,2,3,4,$  we can determine when $(x_j,\xi_j)\in L^+({{N}}^-)$ are such that the geodesics
 $\gamma_{x_j,\xi_j}$, $j=1,2,3,4$ have a common intersection point in the set ${{N}}^+$.
{\mnewtext  Finally, as the metric of ${{N}}^-$ coincides with that of $\R\times \mathbb S^3$,
we can determine when $(x_j,\xi_j)\in L^+({{N}}^-)$ are such that the geodesics
 $\gamma_{x_j,\xi_j}$, $j=1,2,3,4$ have a common intersection point in the set ${{N}}^-$.
 Summarizing, we have analyzed the cases when the geodesics
 $\gamma_{x_j,\xi_j}$, $j=1,2,3,4$ have either a common intersection point in $I^-_{N_{\ext}^j}(p_{{0}}^+){\mnewtext \cap \overline N}$ or ${{N}}^-$ or  ${{N}}^+$ 
 and also in the case when common intersection points in $I^-_{N_{\ext}^j}(p_{{0}}^+)$ do not
 exist.
 Note that some of the cases may happen at the same time but when only the first case
 takes place, we know that the geodesics have a common intersection point in the set $X_j(s_{{0}}^+).$}
 
 The above observations imply we can use the source-to-solution map to determine the light observation sets 
 $\{\mathcal E_{\tilde U}(q)\subset N_{\ext}^j\mid q\in X_j(s_{{0}}^+)\}$, $j=1,2$, that is, we see that
$$
\{\mathcal E_{\tilde U;N_{\ext}^1}(q)\subset N_1\mid q\in X_1(s_{{0}}^+)\}
=\{\mathcal E_{\tilde U;N_{\ext}^2}(q)\subset N_2\mid q\in X_2(s_{{0}}^+)\},
$$
where $\mathcal E_{\tilde U;N_{\ext}^j}(q)$ is the light observation set, {\mnewtext  that is, the intersection of
the light cone emanating from the point $q$ and the $\tilde U$}  on the manifold
${N_{\ext}^j}$.

To continue the construction, we observe that  for all $0<s<s_{{0}}^+$ and $q\in  X_j(s_{{0}}^+)$,  it holds that $q\in  X_j(s_{{0}}^+)\setminus  X_j(s)$ if and
only if $f^+_{\hat \mu}(q)>s$, {\mnewtext  that is equivalent to that $S=\mathcal E_{\tilde U;N_{\ext}^1}(q)$
satisfies $S\cap \hat \mu(s,s_{+2})\not =\emptyset$}. Thus,  
\begin{multline}\label{s0 s sets}
\{\mathcal E_{\tilde U;N_{\ext}^1}(q)\mid q\in X_1(s_{{0}}^+)\setminus  X_1(s)\} 
=\{\mathcal E_{\tilde U;N_{\ext}^2}(q)\mid q\in 
X_2(s_{{0}}^+)\setminus  X_2(s)\}.
\end{multline}
Next we use \cite[Theorem 1.2]{KLU2018} to reconstruct
the conformal class of a set from the collection of the light observation
sets of its points.
Observe that the closure of the set $ X_1(s_{{0}}^+)\setminus  X_1(s)$ is a compact
subset of  $ I^-_{N_{\ext}^j}(\hat \mu(s_{+}))\setminus  J^-_{N_{\ext}^j}(\hat \mu(s))$.
Thus, using \cite[Theorem 1.2]{KLU2018},  with the observation set $\tilde U$ that is a neighborhood of the time-like path $\hat \mu\cap \tilde U$, the formula \eqref{s0 s sets} implies that there is a conformal diffeomorphism
\beq\label{s0 s sets B1}
\Psi:X_1(s_{{0}}^+)\setminus  X_1(s)\to X_2(s_{{0}}^+)\setminus  X_2(s).
\eeq
By taking the union of these sets over $s>0$, we see that 
 there is a conformal diffeomorphism
\beq\label{s0 s sets B2}
\Psi:X_1(s_{{0}}^+)=X_1(s_{{0}}^+)\setminus  X_1(0)\to X_2(s_{{0}}^+)= X_2(s_{{0}}^+)\setminus  X_2(0).
\eeq

{\mnewtext  Observe that the images of the geodesics $\gamma_{x_0,\xi_0}$, where $x_0\in N^-$
and $\xi_0$ is light-like,
on $\mathcal E_{\tilde U;N_{\ext}^2}(X_1(s_{{0}}^+))$,
that is, 
$$
\{\mathcal E_{\tilde U;N_{\ext}^2}(q)\in \mathcal E_{\tilde U;N_{\ext}^2}(X_1(s_{{0}}^+)) \mid q\in 
\gamma_{x_0,\xi_0}([0,\T(x,\xi)))\cap \tilde U\}
$$
is
the set of those $\mathcal E_{\tilde U;N_{\ext}^2}(q)\in \mathcal E_{\tilde U;N_{\ext}^2}(X_1(s_{{0}}^+))$ that are produced in the above construction
with directions $((x_j,\xi_j))_{j=1}^4,$ where $(x_1,\xi_1)=(x_0,\xi_0)$.}
By using light-like geodesics that intersect $\mathcal I^-$ transversally, we can glue the sets $X_j(s_{{0}}^+)$ together and 
${{N}}^-$ 
and see that 
 there is a conformal diffeomorphism
\beq\label{s0 s sets c}
\Psi:\mathcal W_1^-\to \mathcal W_2^-,\quad \mathcal W_j^-:=I^-_{N_{\ext}^j}(p_{{0}}^+)\cap I^+_{N_{\ext}^j}(p_{-}). 
\eeq

{\mnewtext  Above, we have essentially reconstructed a neighborhood of $\nullinf^-$.
Next we make similar considerations in a neighborhood of $\nullinf^+$}.

Next we apply the observation we made above that  in the set 
$W_j:=J^+(p_0^-)\cap J^-(p_{+2})\subset N_{ext,j}$ the light-like geodesics emanating from $ W_j\cap {{N}}^-$ have no cut points.
Using conormal sources supported
in $W_j\cap {{N}}^-$  and the corresponding
4-tuples $(x_j,\xi_j)_{j=1}^4$, where $(x_j,\xi_j)\in L^+ (W_j\cap {{N}}^-)$, we see as above,
using Lemma~\ref{lem: Eq construction},
that 
\begin{multline}\label{s0 miinus s sets}
\{\mathcal E_{{{N}}^+;{N_{\ext}^1}}(q)\subset {{N}}^+\mid q\in Y_1(s_0^-)\setminus 
J^-_{{N_{\ext}^1}}(\hat \mu(s))
\}
\\
=
\{\mathcal E_{{{N}}^+;{N_{\ext}^2}}(q)\subset {{N}}^+\mid q\in 
Y_2(s_0^-)\setminus 
J^-_{{N_{\ext}^2}}(\hat \mu(s))
\}
\end{multline}
for all $0<s<s_{+{}}$. Then, we see as above,
using \cite{KLU2018}, Theorem 1.2, with the observation set ${N}^+$ that is a neighborhood of the time-like path $\hat \mu$,
that there is a conformal diffeomorphism
\beq\label{s0 miinus s sets B}
\Phi:Y_1(s_0^-)\setminus 
J^-_{N^1}(\hat \mu(s))\to 
Y_2(s_0^-)\setminus 
J^-_{N^2}(\hat \mu(s)).
\eeq
Taking union over sets $s>0$ 
we see 
that there is a conformal diffeomorphism 
\beq\label{s0 miinus s sets B2}
\Phi:Y_1(s_0^-)\to 
Y_2(s_0^-).
\eeq
Let $(y,\zeta)\in L^+{N}^+$ and consider the set of
the points $q\in Y_j(s_0^-)$ whose light observation
set $\mathcal E_{N^+}(q)$ contains the light-like geodesic 
$\gamma_{y,\zeta}([0,r_0])\subset {{{N}}}^-$,
that is, the set 
\ba
\Gamma_{y,\zeta}&:=&\{\mathcal E_{{{N}}^+}(q) 
\in \mathcal E_{{{N}}^+}(Y_j(s_0^-))\ |\ \gamma_{y,\zeta}([0,r_0])\subset
\mathcal E_{{{N}}^+}(q)\}.
\ea
We see as in \cite[Lemma 2.6]{KLU2018} that
\ba
\Gamma_{y,\zeta}=\{\mathcal E_{{{N}}^+}(q) \mid q\in \gamma_{y,\zeta}(\R)\cap Y_j(s_0^-)\}.
\ea
Using this, we can construct the images of the light-like geodesics $ \gamma_{y,\zeta}(\R)
\cap({{{N}}}^-\cup  Y_j(s_0^-))$,
 in the map $\mathcal E_{N^+}$. Using such geodesics that intersect $\mathcal I^+$ transversaly,
 we can glue the sets $Y_j(s_0^-)$ together with 
${N}^+$.
As this construction can be done on both manifolds $N_{\ext}^j$, $j=1,2,$ we see that 
 there is a conformal diffeomorphism
\beq\label{s0 s sets c2}
\Psi:\mathcal W_1^+\to \mathcal W_2^+,\quad \mathcal W_j^+:=I^-_{N_{\ext}^j}(p_{+})\cap I^+_{N_{\ext}^j}( p_{{0}}^-).
\eeq
By combining the conformal diffeomorpshims
\eqref {s0 s sets c} and \eqref {s0 s sets c2},
we see that there is a conformal diffeomorphism
\beq\label{s0 s sets combined}
\Psi:\mathcal W_1\to \mathcal W_2,\quad \mathcal W_j:=
\mathcal W_j^+\cup \mathcal W_j^-\subset N_{\ext}^j.
\eeq

\medskip

{\it {\bf Step 3:} Conformal transformation of the  neighborhoods of $\mathcal I^-$ and $\mathcal I^+$ and source-to-solution maps.}

\noindent
\begin{minipage}{0.68\textwidth}
By the above, there exists a function $\gamma\in C^\infty(\mathcal W_1)$
such that
\beq
 g_{1} = e^{2\gamma}\Psi^*g_{2},\quad \hbox{on }\mathcal W_1.
 \eeq
As the metric tensor in the set ${N}^\pm \subset N_{\ext}^j$
coincides with the product metric of $\R\times \mathbb S^3$
we have that  $\gamma=0$ on $({{{N}}}^-\cup  {{{N}}}^+)\cap \mathcal W_1$.

Next, we will consider an extension of the  source-to-solution map for a non-linear wave equation with an additional term in the zeroth order term $B$. To do that, we introduce some notations.
Let  $V_1,V_2\subset \Next$ be relatively compact open sets, 
and $K\subset V_1$ be a compact  set.

Let
$f\in H^k_0(K)$, $\norm{f}_{H^k(V_3)} < \eps$,
where $\eps=\e_{K}>0$ is small enough.

Then there exists a unique solution to 
\begin{equation}\label{eq:nonlinear wave equation in Penrose 2 cp2}
\begin{cases}
&(\square_{\gext}+B)w + Aw^{\kappa}=f,\quad
\text{in } I^-_\Next(p_+),\\
&\supp( w)\subset J^+_\Next(\supp(f))
\end{cases}
\end{equation}
and we define the source-to-solution map 
\ba
L_{g,B,A;V_1,V_2}(f)=w|_{V_2}.
\ea
\end{minipage}
\hfill
\begin{minipage}{0.31\textwidth}
\begin{center}
    \includegraphics[height=7cm]{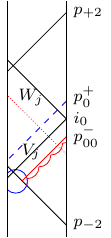}
    \captionsetup{width=.9\linewidth}
    \captionof{figure}{Subsets $W=I^+(p_0^-)\cap I^-(p_{+2})$ and $V=I^-(p_{{0}}^+)\cap I^+(p_{-2})$  of $ \Next$ in the Penrose  compactification.}\label{Fig 9.} \label{Fig 10}
\end{center}
\end{minipage}

Moreover, we define the domain  $\mathcal D_{V_1,V_2}=\mathcal D(L_{g,B,A;V_1,V_2})$ of the map $L_{g,A,B;V_1,V_2}$
to be the union
\ba
\mathcal D(g,A,B;V_1,V_2)=\bigcup_{K\subset \subset V_1}\{f\in  H^k_0(K)\mid 
 \norm{f}_{H^k(V_1)} <\e_{K}\},
\ea
where the union is taken over the subsets $K \subset V_1$ and $\e_K>0$
are sufficiently small numbers.

Next we consider a  transformation of the conformal class of the metric with the function $e^{2\gamma}$.
That is, we consider the function $\underline u=e^{-\gamma} u$ 
on the manifold $N_{\ext}^1$ where  $u$ satisfies 
the equation
\beq\label{new wave eq1}
(\square_{g_1}+B_1+{\mntext D_1})u+A_1u^\kappa=f\quad \hbox{on }\mathcal W_1,
\eeq
where $f$ is supported in $({{{N}}}^+\cup {{{N}}}^-)\cap \mathcal W_1$.
We observe that as 
$\gamma=0$ in $({{{N}}}^+\cup {{{N}}}^-)\cap \mathcal W_1$,
the equation \eqref{new wave eq1}
implies 
\beq\label{new wave eq2}
(\square_{\underline g_1}+\underline B_1+{\mntext \underline  D_1}+\underline q_1)\underline u+\underline A_1\underline u^\kappa=f,
\eeq %
where $\underline g_1=e^{2\gamma }g_1$ and
{\mntext $\underline B_1=e^{-2\gamma }B_1$,} {\mntext $\underline D_1=e^{-2\gamma }D_1$,} $\underline A_1=e^{(\kappa-3)\gamma }A_1$ and $\underline q_1=\frac{1}{6} (R_{\underline g_1}  - e^{2\gamma}R_{g_1})$ 
is a smooth function that
vanishes on  $({{{N}}}^+\cup {{{N}}}^-)\cap \mathcal W_1$. 

Next, we consider the principal symbols of the observed waves using \cite[Theorem 6.2]{LUW} (more precisely, we use
 \cite[Prop. 4.3 and 4.5]{LUW}). To this end, we use the operators $f\mapsto (\square_g+B)^{-1}f$. 
 We observe that the principal symbol
 of the operator $(\square_g+B)^{-1}$ coincides with that of $\square_g^{-1}$.

To consider the derivatives of the source-to-solution operators 
that correspond to the zeroth order terms $B_1$ and $\tilde B={\mntext \underline  B_1}+{\mntext \underline  D_1}+\underline q$, we recall 
certain details in our earlier considerations.
Let us consider the direction $(\vec x,\vec \xi)=((x_j,\xi_j))_{j=1}^4$ and $s_0$ that satisfy either the above condition (A)
or the conditions (B) and (T).
Let $f_j$ be the conormal sources associated to direction $(\vec x,\vec \xi)=((x_j,\xi_j))_{j=1}^4$, given in \eqref{conormal souces 1}. For such sources,
 we define the linearized solutions, see
\eqref{eq:linearized_solutions}, using the coefficients $B_1$ 
and $\tilde B={\mntext \underline  B_1}+{\mntext \underline  D_1}+\underline q$. 
These linearized solutions are
\beq\label{eq:linearized_solutions q} 
&& u^{B_1}_j = (\square_{g}+B_1)^{-1}f_j \in
 I({\Nextended}\setminus\{ x_j \}; \Lambda(x_j,\zeta_j,s_0)),\quad \hbox{and}\\
 && u^{{\mntext \tilde B_1}}_j = (\square_{g}+{\mntext \tilde B_1})^{-1}f_j \in
 I({\Nextended}\setminus\{ x_j \}; \Lambda(x_j,\zeta_j,s_0)).
\eeq
We see that the principal symbols of $ u^{B_g}_j$ and $ u^{{\underline B}}_j$ on $\Lambda(x_j,\zeta_j,s_0)$ coincide.
Similarly to \eqref{eq:4th_order_solutions}, we define  the waves
produced by the interaction of linearized waves,
\beq\label{eq:4th_order_solutions q}
&&\U^{(\kappa,B_1)} = (\square_{{g}}+B_1)^{-1} \Source^{B_1},\quad \U^{(\kappa,{\mntext \tilde B_1})} = (\square_{{g}}+{\mntext \tilde B_1})^{-1} \Source^{{\mntext \tilde B_1}},
\eeq
where
\begin{align}\label{eq:4th_order_source}
\Source^{B_1}&:= -{\kappa} !\cdot A_1 u_1^{B_1}u_2^{B_1}u_3^{B_1}(u_4^{B_1})^{\kappa-3},\quad 
\Source^{{\mntext \tilde B_1}} &:= -{\kappa} !\cdot \underline A_1 u_1^{{\mntext \tilde B_1}}u_2^{{\mntext \tilde B_1}}u_3^{{\mntext \tilde B_1}}(u^{{\mntext \tilde B_1}}_4)^{\kappa-3}.
\end{align}
Observe that 
 \ba
 & &\U^{(\kappa,B_1)}|_{V_+}= D^\kappa|_0L_{g,B_1,A;V_-,V_+}[f_1,f_2,f_3,f_4],\\
 & &\U^{(\kappa,{\mntext \tilde B_1})}|_{V_+}= D^\kappa|_0L_{\underline g,\underline A,{\mntext \tilde B_1};V_-,V_+}[f_1,f_2,f_3,f_4]
 \ea
 are the $\kappa$-th order (Fr\'echet) derivatives of the maps $f\to L_{g,B_1,A;V_-,V_+}(f)$
 and $f\to L_{\underline g,{\mntext \tilde B_1},\underline A;V_-,V_+}(f)$,
 evaluated 
 at the point $f=0$, see \eqref{Dkappa}.
As pointed out after the claim of Theorem~\ref{thm:analytic limits A wave A and B},
 the restrictions of the functions $\U^{(\kappa,{\mntext \tilde B_1})}$ and $\U^{(\kappa,B_1)}$ to the domain
$\goodN(\vec x,\vec \xi)$  are Lagrangian distributions associated to 
the same Lagrangian manifold
 $ \Lambda_{1234}$  and  their principal symbols
 are the same. 
  Below, we refer to this property by saying that the
Fr\'echet derivatives $D^\kappa|_0L_{g,B_1,A;V_-,V_+}$
and $D^\kappa|_0L_{\underline g,{\mntext \tilde B_1},\underline A;V_-,V_+}$ are the same
up to a smoothing error of order one.
This means that adding a potential $\underline q$ does not change the construction
of the metric $g$ or the coefficient $A$ of the non-linear term.

Also, our assumption that the source-to-solution maps 
$L_{g_1,B_{1},A_1;{{{N}}}^-,{{{N}}}^+}$ and 
$L_{g_2,B_{2},A_2;{{{N}}}^-,{{{N}}}^+}$
are the same, imply that when
the pair $(V_-,V_+)$ 
is either 
\beq\label{V plus minus sets}
({{{N}}}^-\cap \mathcal W_1^-,{{{N}}}^+\cap \mathcal W_1^-),\quad\hbox{or}\quad 
({{{N}}}^-\cap \mathcal W_1^+,{{{N}}}^+\cap \mathcal W_1^+),
\eeq
then 
\beq\label{L maps the same with q} 
\nonumber
L_{g_2,B_{2},A_2;V_-,V_+}(f)=L_{g_1,B_{1},A_1;V_-,V_+}(f)=L_{\underline g_1,{\mntext \tilde B_1},\underline A_1;V_-,V_+}(f).
\eeq
for all $f\in \mathcal D(L_{g_1,B_{1},A_1;V_-,V_+})$.
As above $(\mathcal W_2,g_2)$ and $(\mathcal W_1,\underline g_1)$ are isometric,
we will next  show that also $A_2$ and $\underline A_1$ coincide in these sets. 

By the above considerations, we have that 
the map 
\beq\label{s0 s sets combined conformald}
\Psi:(\mathcal W_1,\underline g_1)\to (\mathcal W_2,g_2)
\eeq
is a  isometric diffeomorphism.
Moreover, as $$L_{g_2,B_{2},A_2;V_-,V_+}=L_{\underline g_1,{\mntext \tilde B_1},\underline A_1;V_-,V_+}$$
and the
maps $D^\kappa|_0L_{\underline g_1,{\mntext \tilde B_1},\underline A_1;V_-,V_+}$
and $D^\kappa|_0L_{\underline g_1,{\mntext \tilde B_1},\underline A_1;V_-,V_+}$
are the same
up to a smoothing error of order one,
the
maps $D^\kappa|_{f=0} L_{\underline g_1,{\mntext \tilde B_1},\underline A_1;V_-,V_+}(f)$
and $D^\kappa|_{f=0} L_{g_2,B_{2},A_2;V_-,V_+}(f)$ of source-to-solution maps on
isometric Lorentzian manifolds $(\mathcal W_1,\underline g_1)$
and $(\mathcal W_2,g_2)$
 are the same
up to a smoothing error of order one.

\medskip

{\it {\bf Step 4:} Construction of $A$ in  neighborhoods of $\mathcal I^-$ and $\mathcal I^+$ and modified source-to-solution maps.}

{\mnewtext  The results in \cite{LUW} imply that when the metric of a Lorentzian manifold
is given, the source-to-solution operator for the non-linear wave equation determines
the coefficient $A$ of the non-linear term uniquely. Moreover, the construction
of $A$ does not depend on the possible zeroth order term $q$ in the equation. More precisely,}
using \cite[Prop. 4.3 and 4.5]{LUW}
and the formula \eqref{L maps the same with q}, 
when $(V_-,V_+)$ is the pair
$({{{N}}}^-\cap \mathcal W_1^-,{{{N}}}^+\cap \mathcal W_1^-),$
 we see that the equation $\underline A_1=\Phi^*A_2$ holds on
the set $\mathcal W_1^-.$ Similarly, using
the formula \eqref{L maps the same with q} 
when $(V_-,V_+)$ is the pair
$({{{N}}}^-\cap \mathcal W_1^+,{{{N}}}^+\cap \mathcal W_1^+),$
 we see that the equation $\underline A_1=\Phi^*A_2$ holds on
on the set $\mathcal W_1^+.$  Thus,
$\underline A_1=\Phi^*A_2$ holds 
on the set $\mathcal W_1.$ 
 Note that we do not know whether the function $\underline q_1$
is zero or not, but as this function does not change the principal
symbol of the source-to-solution map, this will not cause issues
in our considerations below.

Above,
we have shown that the sets $(V_1,\underline g_1)$ and $(V_2,g_2)$,
where we denote $V_j:=I^-(p_{{0}}^+)\cap I^+(p_{-2})\subset {N_{\ext}^j}$,
 $j=1,2$, are isometric
 and thus we can identify these sets below and  denote those by $V$.
 With this identification, the metric tensor $g_2$ and $\underline g_1$ 
 as well as the coefficients $A_2$ and $\underline A_1$ coincide on $V$.
 As  $V$ is a globally hyperbolic manifold, derivatives of the
  source-to-solution maps,  $D^\kappa|_{0} L_{\underline g_1,{\mntext \tilde B_1},\underline A_1;V,V}$ and
  $D^\kappa|_{0} L_{g_2,B_{2},A_2;V,V}$ are the same up to an order 1 smoothing error.

Similarly,  as 
  we have shown that  sets $(W_1,\underline g_1)$ and $(W_2,g_2)$,
  where $W_j=I^+(p_{{0}}^-)\cap I^-(p_{+2})\subset {N_{\ext}^j}$, $j=1,2$ are isometric,
we  identify the sets $W_j$, $j=1,2$,
and  denote them by $W$. 
As above, we see that  the
  source-to-solution maps  
  $D^\kappa|_{0} L_{\underline g_1,{\mntext \tilde B_1},\underline A_1,W,W}$ and
  $D^\kappa|_{0} L_{g_2,B_{2},A_2,W,W}$ are the same up to an order 1 smoothing error.

\medskip

{\it {\bf Step 5:} Reconstruction of near field measurements, that is, finding a family of the source-to-solution maps
in a neighborhood $\hat \mu$.}

 Next we use the path  $\mu_{mod}=\hat \mu\cap J^+_{N_{\ext}}(p_{-})\cap J^-_{N_{\ext}}(p_{+})$ and let
 $U$ be a neighborhood of $\mu_{mod}$ such that  $\hbox{cl}(U)\subset {{N}}^-\cup W$.
 Using the first derivatives of the above constructed source-to-solution maps (i.e., the linearized source-to-solution maps) we see that the causality
relations $R^<_{U,g_j}=\{(x,y)\in U\times U\mid x<_{(j)}y\}$ in the set $U$, where $<_{(j)}$ is the causality relation of the Lorentzian manifold $(N_{\ext},g_j)$, coincide for $j=1,2$.

Recall that the zeroth order term $q$ does not change principal symbols of 
the derivatives of the source-to-solution maps.
Let us next denote 
\beq\label{new S-to-S maps}
& &D^\kappa|_{0} L_{g_j,B_{j},A_j,V_1,V_2}[f_1,f_2,f_3,f_4] \\
\nonumber
&=& \p_{\e_1}\p_{\e_2}\p_{\e_3}\p_{\e_4}^{{\kappa}-3}
L_{g_j,B_{j},A_j,V_1,V_2}(\e_1f_1+\e_2f_2+\e_3f_3+\e_4f_4)\bigg|_{\vece=0},
\eeq
see \eqref{Dkappa}. 
Thus the above shows that the source-to-solution maps 
\beq\label{new S-to-S maps1}
(f_1,f_2,f_3,f_4)\to
D^\kappa|_{0} L_{\underline g_1,{\mntext \tilde B_1},\underline A_1,V_1,V_2}[f_1,f_2,f_3,f_4] 
\eeq
and 
\beq\label{new S-to-S maps2}
(f_1,f_2,f_3,f_4)\to
D^\kappa|_{0} L_{g_2,B_{2},A_2,V_1,V_2}[f_1,f_2,f_3,f_4] 
\eeq
 are the same up to  smoothing error of order 1 when 
 $(V_1,V_2)=(V,V)$,  $(V_1,V_2)=(W,W)$, or  $(V_1,V_2)=({{N}}^-,{{N}}^+\cap I^-(p_{+2})$.

We observe that any point $p\in U$ has an open neighborhood $O_p^{(1)}\subset U$ 
and any point $q\in J^+(p)\cap U$ has a neighborhood $O^{(2)}_{p,q}\subset U$ 
such that for some of the above pairs $(V_1,V_2)$ 
it holds that $O_p^{(1)}\subset V_1$ and $O^{(2)}_{p,q}\subset V_2$.
Thus, the above considerations lead to the following  conclusion:
We have shown that the metric tensors $g_j$, $j=1,2,$ and the coefficients $A_j$
of the non-linear term coincide up to a conformal transformation in a neighborhood $U$ of 
 the path  $\mu_{mod}$. Moreover,  any points $p\in U$  and $q\in J^+(p)\cap U$
and for all $f_1,f_2,f_3,f_4\in H^k_0(U)$ supported in an open neighborhood $O_p^{(1)}\subset U$
 we can determine the principal symbol of $D_\kappa|_{\vec \e=0} u^{f_{\vec \e}}$ in ${O^{(2)}_{p,q}}$.
 For points $q\in  U\setminus J^+(p)$ this principal symbol is zero, and we can find the principal symbol of 
 $D_\kappa|_{\vec \e=0} u^{f_{\vec \e}}$ in the whole set $U$. Hence, for the functions 
 $f_1,f_2,f_3,f_4$ supported in $O_p^{(1)}\subset U$,
 the maps
 \beq\label{map 1}
 (f_1,f_2,f_3,f_4)\to
D^\kappa|_{0} L_{\underline g_1,{\mntext \tilde B_1},\underline A_1,O_p,U}[f_1,f_2,f_3,f_4]
\eeq
and 
 \beq\label{map 2}
(f_1,f_2,f_3,f_4)\to
D^\kappa|_{0} L_{g_2,B_{2},A_2,O_p,U}[f_1,f_2,f_3,f_4] 
\eeq
 are the same up to  smoothing error of order 1.
Thus for such sources $f_j$ 
the principal symbols of $D_\kappa|_{\vec \e=0} (u^{f_{\vec \e}}|_U)=D^\kappa|_{f=0} L_{g_j,B_{j},A_j,U,U}(f)[f_1,f_2,f_3,f_4]$ are the same for $j=1,2$. 
\medskip

{\it {\bf Step 6:} Reconstruction of the manifold.}

{\mnewtext  Recall  that $U\subset {{N}}^-\cup W$ is a neighborhood of the connected path $\mu_{mod}$.
This makes it possible to apply the results of \cite{KLU2018,LUW} for the inverse problem where
the source-to-solution map is studied in the case when the set $U_{\i}$, where sources are supported, and the
the set $U_{\out}$, where the waves are observed, are the same neighborhood $U=U_{\i}=U_{\out}$   of a connected time-like path.}
Thus the fact that the maps \eqref{map 1} and \eqref{map 2} are the same up to  smoothing error of order 1 implies that we can determine the principal symbols of 
 functions $D_\kappa u^{f_{\vec \e}}|_{\vec \e=0}$  in $U$, where
 $f_{\vec \e}$ are linear combinations of conormal sources $f_k$, $k=1,2,3,4$, that are all supported 
 in some open set $O_p.$  Moreover, the union of these open sets $O_p$ covers the whole set $U$.

The above observation can be used   in the proof of Theorem 4.5 of \cite{KLU2018}
where in all steps
one needs to consider only sources $f_{\vec \e}$ that are supported in an arbitrarily small
neighborhoods $O_p\subset U$. In particular, for $x_0\in O_p$ and a light-like geodesic $\gamma_{x_0,\xi}$ 
emanating from $x_0$,
we see using Lemma  \ref{thm:analytic limits A wave A and B} that for all $q\in \gamma_{x_0,\xi}([0,\rho(x_0,\xi)])\cap N$,
for which we have  $a(q)\not =0$, we can determine the  light-observation set $\mathcal E_{U}(q)$.
For $q\in \gamma_{x_0,\xi}\cap N_-$ the first Frechet derivatives of the source-to-solution map 
$L_{g,B,A,O_p,U}$
determines the  light-observation set $\mathcal E_{U}(q)$ and as the Lorentz metric on $N_-$ 
is is given to us, we can determine $\mathcal E_{U}(q)$  for the points  $q\in \gamma_{x_0,\xi}\cap N_+$.
Using these observations and the proof of Theorem 4.5 of \cite{KLU2018}, we can determine
 the conformal type of 
 the space-time $I^+(p_{-})\cap I^-(p_{+})$.
 Again, we can 
perform a conformal transformation to change the quadruple
$(N_{\ext}^1,g_1,B_{1},A_1)$ to $(N_{\ext}^1,\underline g_1,{\mntext \tilde B_1},\underline A_1)$ so that there is an isometric diffeomorphism
\ba
\Psi:(I^+_{N_{\ext}^1}(p_{-})\cap I^-_{N_{\ext}^1}(p_{+}),\underline g_1)\to  (I^+_{N_{\ext}^2}(p_{-})\cap I^-_{N_{\ext}^2}(p_{+}),g_2).
\ea
After this, we can  
use the analysis of the principal symbols of the $\kappa$:th derivatives of source-to-solution maps, see
 \cite[Theorem 6.2]{LUW}, and prove that $\underline A_1=\Psi^*A_2$ on 
 $I^+_{N_{\ext}^1}(p_{-})\cap I^-_{N_{\ext}^1}(p_{+})$.
 By letting $p_+\to i_+$ and $p_-\to i_-$,
 we see that after performing a conformal transformation to the triple
$(N_{\ext}^2,g_2,A_2)$ there is an isometry
$\Psi:(N_{\ext}^1,\underline g_1)\to (N_{\ext}^2,g_2)$ and $\underline A_1=\Psi^*A_2$ on $N_{\ext}^1$.
  This proves the claim.
\end{proof}

{\mnewtext  As Schwartz class perturbations of the Minkowski space satisfy the assumptions
of Theorem~\ref {main thm for general manifold}, we see that Theorem~\ref{main thm for Minkowski} follows from
 Theorem~\ref {main thm for general manifold}. 
 Theorem \ref{main thm for general manifold generalized}
 follows by 
applying a conformal trasformation \eqref{Q0 and a0} that changes $(M,g)$ to the space-time $(M,g_0)$ 
with an asymptotically Minskowskian infinity and 
the wave equation \eqref{modified equation with Q} to \eqref{conformally Minkowskian wave equation}. 
Then, using only the restricted scattering data, we apply the  proof of 
Theorem~\ref{main thm for general manifold}  using only those points $p_{-}$ for which the sets $ \mathcal W_j^-=I^-_{N_{\ext}^j}(p_{{0}}^+)\cap I^+_{N_{\ext}^j}(p_{-})$ satisfy 
$ \mathcal W_j^-\subset S(R(t_1^*))$, and obtain the claim of {Theorem} \ref{main thm for general manifold generalized}.}

 \subsubsection{On the further generalizations}
 The techniques developed in this paper can be generalized to certain metric tensors that are non-smooth
 at the spatial infinity $i_0$ of the Penrose compactification. This situation is encountered in several physical models with a positive total 
  Arnowitt–Deser–Misner
(ADM) mass
 of the space-time, see \cite[pp.\ 276-278 and 281-285]{Wa1984}. However, 
 for such metric tensors the existence of solutions to the direct problem is a difficult and deep problem, see \cite{Joudioux1,Joudioux2}. This generalization
 is outside the scope of this paper and will be considered  elsewhere, but we will below  discuss in brief  how this
 generalization can be done. First,  let us assume $(M,g_M)$ is a space-time whose conformal compatification
  $(N,g_N)$ is such that the set $N$ is the equal to $\hat N\subset \R\times \mathbb S^3$ used to define the Penrose conformal compactification. Moreover,
  assume that the metric $g_N$ on  $N$ can be extended to $N_+$ and $N_-$ in a way that makes $N_{\ext}=N\cup N_-\cup N_+$
  a globally hyperbolic manifold with a $C^0$-metric $g_{\ext}$ that is $C^\infty$-smooth in $N_{\ext}\setminus i_0$ and $\partial N$ is a light cone
 emanating from $i_0$. Moreover,  we assume that $(M,g_M)$ is a glued Schwartchild space-time in the sense of {Joudioux}, see \cite[Def.\ 1.3]{Joudioux1}. This implies that $(M,g_M)$ is a  Chrusciel-Delay/Corvino-Schoen type space, the conformally compactified metric has a $C^\infty$-smooth extensions to the neighborhoods of the points 
 $i_+,$ and $i_-$, and the spatial infinity $i_0$ has a neighborhood
 where the metric is isometric to the Schwartchild metric. The initial data sets for such space-times satisfy the constraint equations  obtained by 
 Chrusciel-Delay \cite{Chrusciel} and Corvino-Schoen
 \cite{Corvino}. Moreover, we need to assume that $i_0$
 has no light-like cut points in the conformal infinities 
 $ {\mathcal I}^+\setminus \{i_+\}$ and ${\mathcal I}^-\setminus\{ i_-\}$.
  Second, consider the
 non-linear cubic wave equation $(\square_{g_{\ext}}+B_{g_{\ext}})u+au^3=0$, where $a(x)<0$ in $N$
 and assume that $a(x)$ can be smoothly extended by zero to $(N_-\cup N_+)\setminus i_0$. Then,
 the  Cauchy problem for the wave equation 
 \begin{equation}
    \label{cubic} (\square_{g_{\ext}}+B_{g_{\ext}})u+au^3=f,
\end{equation} 
in $N_{ext}$ 
with the compactly supported smooth sources $f$ exists and is unique, see \cite[Prop. 2.1]{Joudioux2}, and moreover, the scattering operator is globally 
defined, see \cite[Thm. 4.4]{Joudioux1}. However, to study the inverse problem with minimal data we define scattering functionals in smaller domains. To do this,
let
 $\mathcal D'_{\Gamma}(\mathcal I^-)$, see \cite[Sec. 8.2]{H1}
be the set of distributions on $\mathcal I^-$ whose wave front set is a subset of a conic set $\Gamma\subset T^*\mathcal I^-$.
Let $\pi:T^*\mathcal I^-\to I^-$ be the projection of a covector   to its base point,
and $\Gamma_n\subset T^*\mathcal I^-,$ $n\in\mathbb Z_+$ be a family of growing conic open subsets whose union is the complement
of the light-like directions in $T^*\mathcal I^-$ and for which the sets $\pi(\Gamma_n)$ are relatively compact.
Moreover, let $q_n\in \mathcal I^-$ be such that $\pi(\Gamma_n)\subset J^+(q_n)$. Then,   the scattering functionals $S_{n,p}:h\to u|_{\nullinf^+(p)}$, where $n\in\mathbb Z_+$ and  $p\in \mathcal I^+,$ are defined for
$h\in \mathcal D(S_{n,p})=\{h\in 
 H^5(\mathcal I^-)\cap \mathcal D'_{\Gamma_n}(\mathcal I^-)\mid \|h_-\|_{H^5(\mathcal I^-)}<\e_{n,p}\}$, where $\e_{n,p}>0$
 are sufficiently small. Note that such boundary values $h$ cause waves whose wave front sets propagate to $N^{\mathrm{int}}$ but not along $\mathcal I^+$. To define the domains of the source-to-solutions operators, let 
$v=v^{h}$  be the solution  of the Goursat problem for the linear wave equation in $N_-$ with the reversed causality,
\beq\label{new linear problem1}
  & & \square_{g_{\ext}} v+ B_{g_{\ext}}v=0,\quad\text{in } N_-,
  \\ \nonumber
  & &v|_{\nullinf^-}=h.
  \eeq
Then, the source-to-solution operators $L_{n,p}:f\to u^f|_{I^-(p^+)\cap N_+}$, where $n\in \mathbb Z_+$ an $p^+\in N_+$, for the equation \eqref{cubic}
 are defined analogously to \eqref{L maps defined} for $f\in \mathcal D(L_{n,p^+})$, where  
 \ba
 \mathcal D(L_{n,p})&=&\{(\square_{g_{\ext}}+B_{g_{ext}})(\psi v^h)\in H^5(N_-): h\in \mathcal D(S_{n,p}),\ \psi\in C^\infty(N_{\ext}),
 \\ & &\quad \qquad \qquad J^+(\supp(\psi))\cap \mathcal I^-\subset J^+(q_n), \ \psi=1\hbox{ near }\supp(h)\}.
 \ea
 We note that the sets $\mathcal D(L_{n,p^+})$ are defined in somewhat less natural way than the domains of the source-to-solution operators
 for the asymptotically Minkowskian space-times, but it contains
 sources which produce waves whose wave front set propagates to the
 unknown space-time $N$. Next, let $u^f$ denote the solution of the Cauchy problem 
 \eqref{eq:nonlinear wave equation in Penrose 2} with $\kappa=3$ and the source $f\in \mathcal D(L_{n,p^+})$. 
Then, the non-smoothness at $i_0$ produces singularities in the wave $u^f$, but these singularities propagate along the light cone $\mathcal I^+$
 emanating from $i_0$. Thus, the non-linear interaction does not affect to these singularities as the coefficient $a(x)$ vanishes on this light-cone. Hence
 the singularities  emanating from $i_0$  do not cause observable singularities in $N_+^{\mathrm{int}}$. 
To reconstruct the light observation sets for the wave equation  \eqref{cubic} with
a cubic non-linearity, we need to replace the analysis done in the proof of Lemma \ref{thm:analytic limits A wave A and B} 
for the $\kappa$-th order non-linearity with that arguments similar to those in proof of Theorem 6.1 in
 \cite{LUW}, that is,
 one needs to consider the singularities produced by the 5th order interaction. Alternatively,
 one can use more advanced, recent techniques developed in  \cite{hintz2024inversenonlinearscatteringmetric}
 to analyze directly the 3rd order interactions. 
 Then, by using the modified version of Lemma  \ref{thm:analytic limits A wave A and B},
 one can proceed as in the proof of Theorem \ref{main thm for general manifold}
 and reconstruct the conformal type of the manifold $\mathcal W$ that is defined in \eqref{s0 s sets combined} and contains
 the sets $V$ and $W$ (See Fig.\ \ref{Fig 10}).
Moreover, as in  the proof of Theorem \ref{main thm for general manifold}
one can see that the source-to-solutions maps \eqref{new S-to-S maps1} and \eqref{new S-to-S maps2}
are the same up to  smoothing error of order 1 for all pairs $(V_1,V_2)=(V,V)$,  $(V_1,V_2)=(W,W)$, or  $(V_1,V_2)=({{N}}^-,{{N}}^+\cap I^-(p_{+2})$.
Let $\hat \mu(s)$ be  the path
$\hat \mu(s)=(s,SP)$ and $U$ a neighborhood of $\hat\mu((0,s_{+2}))$ in $N_+$.
Let $r<0$ be such that $J^+(\hat \mu(r))\cap I^-(p_{+2})\subset W$. 
After these more difficult steps where we have constructed the space-time near the light-like and space-like infinities, let us consider how one can
determine  the earliest light observation sets $ \mathcal E_U(q)$
for the points $q\in J^+(p_-)\cap J^-(p_+)$.
To do this, we make three observations: First, as the Lorentzian manifold $W$ is already constructed and thus known, we can
 determine the sets $\mathcal E_U(q)$ for all $q\in W$. Second,
by using the first derivatives of the  source-to-solution maps \eqref{new S-to-S maps1}, we can
determine  $\mathcal E_U(q)$ for all $q\in N^-\cap I^+(p_{-2})$.
Third,
let $x_j\in N_-$, $j=1,2,3,4$ be points in a sufficiently small neighborhood of the point $\hat \mu(s)$,
$s_{-2}<s\le r$ and   
$\xi_j\in T_{x_j}N_{\ext}$ be future directed light-like vectors and $t_0>0$ be sufficiently small.
Then, we can use the above discussed version of Lemma  \ref{thm:analytic limits A wave A and B} for the wave equation with cubic non-linearity
and the proof of Lemma \ref{lem: Eq construction} to construct the set $S(\vec x,\vec \xi,t_0)$ 
given in Lemma \ref{lem: Eq construction}.
Note that these sets are either of the form $\mathcal E_U(q)$, where $q=\gamma_{x_1,\xi_1}(t)$ with $t>t_0$,
or sets that does not intersect $\mathcal N(\vec x,\vec \xi,t_0)$, c.f.\  Lemma \ref{lem: Eq construction}.

Then, we can apply a layer stripping method based on induction to
reconstruct the space-time: 
First, we note that as $J^+(\hat \mu(r))\cap J^-(p_-)\subset W$, we already have found the collection of the light observation sets
$\{\mathcal E_U(q): \ q\in J^+(\hat \mu(r))\cap J^-(p_-)\}$. Then, by using the above three observations and the proof of Theorem \ref{main thm for general manifold},
we can 
apply the induction  steps
in the proof of Theorem 4.5 in \cite{KLU2018} to recurrently reconstruct the sets
$\{\mathcal E_U(q): \ q\in J^+(\hat \mu(r_j))\cap J^-(p_+)\}$ with a sequence of values $r_j$, $j=1,2,\dots,J$,
where $r_0=r$, $r_{j+1}\in (r_j-\e,r_j)$, $r_J=s_-$, and $\e>0$ is small enough. Hence we can find  
the collection
$\{\mathcal E_U(q): \ q\in J^+(\hat \mu(s_-))\cap J^-(p_+)\}$.  
    Then, similarly to the proof of Theorem \ref{main thm for general manifold},
we can reconstruct the conformal type of the space-time $J^+(\hat \mu(s_-))\cap J^-(p_+)=J^+(p_-)\cap J^-(p_+)$.
By letting $s_-\to -\pi$ and $s_+\to \pi$, we find the conformal type of $N$

 As noted above, we will present the details of these constructions elsewhere, and do consider in this paper the
 space-times that are asymptotically Minkowskian and smooth near $i_0$, 
  in particular, as in the this case the direct problem can be solved without assuming that
  the spacetime is isometric to a given, a priori known space-time (e.g., the Schwartchild or the Minkowski space)
  in a neighborhood of $i_0$.

\appendix
\section{Energy inequality for a nonlinear wave equation}\label{sec:App_energy_inequality}

In this section we prove an a-priori energy inequality for the wave equation
$$
\square_g \psi + B\psi+ A\psi^\kappa=F.
$$
At this stage, the number $\kappa$ can be any positive integer greater than or equal to $1$.
We will prove the energy inequality in a Sobolev space $H^k$.

Below, we use time functions ${\bf t}_j:N_j\to \R$ defined in \eqref{time function tj} 
and the relatively compact sets $W$ and $W_0$ defined in \eqref{eq:domain_for_wave_eq} and \eqref{eq:domain_for_wave_eq B}
with $T_1<0<T_2$ chosen so that so that $\overline{\mathbb D_0}\subset W_0^{\mathrm{int}}.$
Note that the future part of the boundary of $W$ is the smooth  surface 
$$
\Sigma_{f}:= \{x\in N_1 \mid t_1(x) > T_1\}\cap \{x\in N_2 \mid t_2(x) = T_2\}
$$
and the past of the boundary is a subset of the union of $\nullinf^-$ and the Cauchy surface $\Sigma_{T_1}$, where $\nullinf^-$ is as in~\eqref{def:null_infinities}.
Let us denote the light-like part of the boundary of $W$ by
$$
\nullinf^-(T_1) := \nullinf^- \cap\{x\in N_1\mid {\bf t}_1(x)>T_1\}.
$$
Now $W$ is foliated by the space-like surfaces
$$
\Sigma_t:=\{ x\in W\mid {\bf t_2}(x)=t\}. 
$$
Note that the surfaces $\Gamma:=\{{x\in N_1\mid\bf t_1}(x)=T_1\} \cap \{x\in N_2\mid{\bf t_2}(x)=t\}$ are not necessarily smooth but we are going to consider the initial and boundary values which will imply that the solution $\psi$ vanishes identically near $\{x\in N_1\mid{\bf t_1}(x)=T_1\}$. Thus the analysis of the solution $\psi$ near $\Gamma$ does not pose a problem.


We denote 
\begin{equation*}
Z^k_0(W)=\{F\in H^k(I^-(W))\mid F(x)=0\hbox{ for }x\in J^-(
(\nullinf^-\cup \Sigma_{T_1})\cap W\}.    
\end{equation*}
We are now ready to state the main result of this section.

\begin{proposition}\label{prop:energy_inequality}
Suppose $\psi \in  H^{k+1}(W)$, where $k+1>2$ satisfies\
$$
\square_g \psi + B\psi + A\psi^\kappa=F,\quad\text{on } W
$$
where $A,B\in C^k(W)$,
$F\in Z^k_0(W)$ and $\kappa\geq 1$.
Assume $\supp(F) \subset I^+(W)$ is compact and that there exists an open neighbourhood $U$ of $\{x\in N_1\mid{\bf t_1}(x)=T_1 \}\cup\{i_0\}$ such that $U\cap\supp(F)=\emptyset$ and $\psi\equiv 0$ in $U$. 
Then there is $\eps>0$ such that if
$$
\Vert F \Vert_{H^{k}(W)}
+
\Vert \psi\Vert_{H^{k+1}(\nullinf^-(T_1))}<\eps,
$$
then for any $\ell=0,\ldots,k$,
\begin{equation}\label{eq: energy ineq on Cauchy}
\Vert \p_t^\ell \psi \Vert_{H^{k-\ell+1}(\Sigma_{t})}
+
\Vert \p_t^{\ell+1}\psi \Vert_{H^{k-\ell}(\Sigma_{t})}
\leq C \Big(\Vert F \Vert_{H^{k}(W)}
+
\Vert \psi\Vert_{H^{k+1}(\nullinf^-(T_1))}
\Big)  
\end{equation}
for $T_1\leq t\leq T_2$.
If $\kappa=1$ we do not need to assume the smallness of the norms of $F$ and the initial values.
Moreover, similar results follow when $W$ is replaced by $W_0$.
\end{proposition}

\begin{remark}\label{rem: energy over W}
Letting  $\nabla {\bf t}_2$ be the smooth time-like normal vector field along $\Sigma_t$ induced by ${\bf t}_2$, summing over $\ell=0,\ldots,k$, and integrating \eqref{eq: energy ineq on Cauchy} over $t\in[T_1,T_2]$ yields
\begin{align*}
\Vert \psi \Vert_{H^{k+1}(W)}^2
& 
\leq 
C \sum_{\ell=0}^k\int_{T_1}^{T_2}
\left(
\Vert \p_t^\ell \psi \Vert_{H^{k-\ell+1}(\Sigma_{t})}^2
+
\Vert \p_t^{\ell+1}\psi \Vert_{H^{k-\ell}(\Sigma_{t})}^2
\right)|\nabla t_2|
dt
\\
&\leq
C (T_2-T_1)(\ell+1) \Big(\Vert F \Vert_{H^{k}(W)}^2
+
\Vert \psi\Vert_{H^{k+1}(\nullinf^-(T_1))}^2
\Big).   
\end{align*}
\end{remark}

To work in Sobolev spaces, we will construct a finite collection of vector fields on $W$ that span the tangent spaces of $W$ at all points.
Consider a finite open cover of the compact set $\overline{W}$ by coordinate charts $\{(U_j,\varphi_j)\}_{j=1}^J$.
Let $p\in \overline{W}$. Then $p\in U_j$ for some $j\in\{1,\ldots,J\}$. Let $V_p$ be an open set with compact closure such that $\overline{V}_p\subset U_j$.
Now the sets $V_p$ form an open cover of $\overline{W}$, so there is a finite subcover $\{V_k\}_{k=1}^K$.
By construction $V_k\subset U_{j_k}$ for some $j_k$, so also the charts $(V_k,\varphi_{j_k})$ form an atlas of $\overline{W}$.
Let $\chi_k\in C^\infty(\Next)$ be such that
$$
\chi_k(x)=
\begin{cases}
1,& x\in V_k,\\
0,& x\not\in U_{j_k}.
\end{cases}
$$
Let $\p_i$, $i=1,2,3,4$, denote the coordinate vector fields of $U_{j_k}$.
We can now define
$$
X_i = \chi_k(x) \p_i,
$$
so that $X_i=\p_i$ in $V_k$ and $X_i=0$ in $W\setminus U_{j_k}$.
Doing this for all of the finitely many sets $V_k$, $k=1,\ldots,K$ yields (at most) $4K$ vector fields $X_i$, such that
\begin{equation}\label{eq:coordinates_on_R}
T_p W = \mathrm{span}\{ X_{i_1},X_{i_2},X_{i_3},X_{i_4} \mid \text{for some } 1\leq i_1<i_2<i_3<i_4\leq 4K\}   
\end{equation}
for all $p\in W$. We will let $X_{i_1}=\p_t$ be the vector along the time-direction, that is, $X_{i_1}=\p_t$.
It can be shown that
$$
\Vert \nabla^j f \Vert_{L^2(W)}\lesssim \sum_{l=0}^j\sum_{i=1}^{4K}\Vert X_i^l f\Vert_{L^2(W)} 
$$

\begin{proof}[Proof of Proposition~\ref{prop:energy_inequality}]
Let $\psi\in C^\infty(\Next)$.
We will prove the energy inequality up to $t\in [T_1,T_2]$. Let
$$
W_t = \{ x\in N_1\mid {\bf t}_1(x)>T_1\} \cap \{ x\in N_2\mid {\bf t}_2(x)<t\} \setminus {{{N}}}^-
$$
and note $W=W_{T_2}$.
The space-like future boundary of $W_t$ is given by $\Sigma_t$.
To obtain estimates in the Sobolev spaces it suffices to consider  $\psi\in C^\infty(W)$, and to find estimates in $L^2$ in terms of the vector fields $X_i$.
Let $\p_i \in T_pW$. Then a direct calculation using Leibniz rule shows that the following equalities hold for the commutators:
\begin{equation}\label{eq:commutator_relations}
\begin{split}
&[\p_i^k,\square ]\psi =
 \sum_{|\alpha|\leq k+1} a_\alpha \p^\alpha \psi,\quad [\p_i^k ,B]\psi = \sum_{|\alpha|< k} b_\alpha \p^\alpha \psi,\\
&\p_i^k (A\psi^\kappa) = 
\sum_{j_1+\ldots+j_{\kappa+1}=k}
\binom{k}{j_1,\ldots,j_{\kappa+1}} \Big(\partial_i^{j_{\kappa+1}}A\Big)
\prod_{l=1}^{\kappa} \partial_i^{j_l}\psi\\
& = \kappa A\psi^{\kappa-1} \p_i^k\psi+
\sum_{\substack{j_1+\ldots+j_{\kappa+1}=k\\j_{\kappa+1}\geq 1}}
\binom{k}{j_1,\ldots,j_{\kappa+1}} \Big(\partial_i^{j_{\kappa+1}}A\Big)
\prod_{l=1}^{\kappa} \partial_i^{j_l}\psi,
\end{split}
\end{equation}
where the coefficient functions $a_\alpha$ are smooth and $b_\beta\in C^{k-|\alpha|}(W)$.
We will now consider
$$
P^{k-\ell}\p_t^\ell\psi := \sum_{|\alpha|\leq k-\ell} X^\alpha\p_t^\ell \psi,\quad 0\leq \ell\leq k,
$$
which is a differential operator of order $k$, where $X\in\Gamma \Sigma_t$ are vector fields along the space-like surfaces $\Sigma_t$ and $\p_t$ the vector field of ${\bf t}$. To simplify the notation, in the sequel we will denote
\begin{equation}\label{eq: simplify notation of derivatives}
    \psi_{\ell,k} \vcentcolon = P^{k-\ell}\p_t^\ell\psi.
\end{equation}
Then, since $\square \psi + B\psi + A\psi^\kappa=F$, using the standard multi-index notation for $\alpha$, we have
\begin{align*}
P^{k-\ell}\p_t^\ell F&=P^{k-\ell}\p_t^\ell(\square \psi + B\psi + A\psi^{\kappa}) \\
&= \square (\psi_{\ell,k}) + \mathcal{R}[\psi],
\end{align*}
where $\mathcal{R}[\psi]$ denotes the lower order terms:
\begin{align*}
\mathcal{R}[\psi]&:=
\sum_{|\alpha|\leq k+1} a_\alpha \p^\alpha \psi
+ B\psi_{\ell,k} + \sum_{|\alpha< k} b_\alpha \p^\alpha \psi\\
&\qquad
+ A\psi^{\kappa-1} \tilde P^k\psi 
+ \sum_{i=0}^3\sum_{\substack{j_1+\ldots+j_{\kappa+1}=k\\j_{\kappa+1}\geq 1}}
c_{j_1,\ldots,j_{\kappa+1}} \prod_{l=1}^{\kappa} \partial_i^{j_l}\psi
\end{align*}
for some continuous functions $a_\alpha$, $b_\beta$ and $c_\delta$ and a differential operator $\tilde P^k$ of order $k$ with smooth coefficients.
Since 
the space $H^k$ is an algebra, when $k>n/2$, we have
\begin{equation}\label{eq:sobo_integral}
\int_{W_t} \mathcal{R}[\psi]^2 \d V_\gext \lesssim 
\Vert \psi \Vert_{H^{k+1}(W_t)}^2
+
\Vert \psi \Vert_{H^{k+1}(W_t)}^{2\kappa}.
\end{equation}


To work with the wave equation and Stokes' theorem, it is convenient to define the \emph{energy-momentum tensor} associated to a smooth function $\psi$ to be the $(0,2)$-tensor field
\begin{equation}\label{eq:energymomentumtensor}
Q[\psi] := d\psi \otimes d\psi -\frac{1}{2}\gext^{-1}(d\psi,d\psi)\cdot \gext.
\end{equation}
It can be shown that
$$
\div(Q[\psi]) = (\square_\gext \psi)d\psi,
$$
where the divergence is defined via $\div(V)=(\nabla^i V)_i$.
We use the following lemma of Aretakis~\cite{Aretakis}:
\begin{lemma}\label{lemma:posdef}
If $V_1,V_2$ are future-directed time-like vector fields, then the energy-momentum tensor is positive-definite, that is
$$
Q[\psi](V_1,V_2) \geq C \sum_{j=0}^3 (\partial_j \psi)^2,\quad C>0.
$$
\end{lemma}
Let us then contract $Q[\psi_{\ell,k}]$ by a vector field $V$ of $\Next$ and 
take the divergence as
\begin{equation}\label{eq:divJ}
\div(Q[\psi_{\ell,k}]V) = \div(Q[\psi_{\ell,k}])V +   \frac{1}{2}Q[\psi_{\ell,k}]_{ij} \pi^{ij}_V
\end{equation}
where $(\nabla V)^{ij} = (g^{ki} \nabla_k V)^j = (\nabla^i V)^j$
and 
$\pi^{ij}_V = (\nabla^i V)^j + (\nabla^j V)^i$ is the deformation tensor of $V$.
Integrating \eqref{eq:divJ} over the domain $W$ in $\Next$, by Stokes' theorem, 
\begin{multline}\label{eq:stokes}
\int_{W_t} ((\square \psi_{\ell,k})(V \psi_{\ell,k})  \frac{1}{2}Q[\psi_{\ell,k}]_{ij} \pi^{ij}_V )\d V_\gext
=
\int_{\p {W_t}} Q[\psi_{\ell,k}](V,n) \d (\p W_t),
\end{multline}
where $n$ is the normal vector to the boundary $\p {W_t}$ and $\d (\p W_t)$ is the induced volume on the boundary.
We remind the reader that the boundary of $W_t$ is not necessarily a smooth surface, particularly near $\Gamma=\{{\bf t_1}(x)=T_1\}\cap \{{\bf t_2}(x)=t\}$, but since $\psi$ vanishes identically near $\Gamma$ Stokes' theorem remains valid. 
We will analyse the form of the volume form on the boundary below.
Let now $V$ be a time-like vector field and let
$$
f(t) \vcentcolon= \int_{\Sigma_t} Q[\psi_{\ell,k}](V,n_t)\d\Sigma_t,
$$
where $n_t$ is the future-directed normal vector of $\Sigma_t$ and $\d\Sigma_t$ is a Riemannian volume form on $\Sigma_t$.
By Lemma~\ref{lemma:posdef}, we know
\begin{equation}\label{eq:fsim}
c\int_{\Sigma_t} \sum_{j=0}^3 (\partial_j \psi_{\ell,k})^2\d\Sigma_t \leq f(t) \leq C \int_{\Sigma_t} \sum_{j=0}^3 (\partial_j \psi_{\ell,k})^2\d\Sigma_t.
\end{equation}
Let
$
\nullinf^-(T_1,t) = \nullinf^- \cap \{ x\in N_1 \mid  T_1 <{\bf t}_1(x) < t\}.
$
By Stokes' theorem and \eqref{eq:stokes} we have
\begin{align*}
&\int_{W_t} (V\psi_{\ell,k})(\mathcal{R}[\psi]-P^{k-\ell}\p_t^\ell F) \d V_\gext + \int_{W_t} Q[\psi_{\ell,k}](\nabla V)\d V_\gext \\
&=
\int_{W_t} \div( Q[\psi_{\ell,k}](V) )\d V_\gext\\
 &=  -\int_{\Sigma_t} Q[\psi_{\ell,k}](V,n)\d\Sigma_t
 + f(T_1) + \int_{\nullinf^-(T_1,t)} Q[\psi_{\ell,k}](V,n) \d\nullinf^-
\end{align*}
where $n$ is the future directed normal vector of $\p {W_t}$ and $\d\nullinf^-$ is a volume element on the null surface $\nullinf^-(T_-,t)$.
Rearranging, we get
\begin{multline}\label{eq:first_energy}
\int_{\Sigma_t} Q[\psi_{\ell,k}](V,n) \d\Sigma_t
= f(T_1) - \int_{W_t} (V\psi_{\ell,k})(\mathcal{R}[\psi]-P^{k-\ell}\p_t^\ell F)\d V_\gext \\
- \int_{W_t} Q[\psi_{\ell,k}](\nabla V)\d V_\gext + \int_{\nullinf^-(T_1,t)} Q[\psi_{\ell,k}](V,n)\d\nullinf^-.
\end{multline}
Because $V$ and $n$ are time-like then in view of Lemma~\ref{lemma:posdef}
$$
\int_{\Sigma_t} |Q[\psi_{\ell,k}](\nabla V)|d\Sigma_t \leq Cf(t).
$$
Recall that $W$ is foliated by the space-like surfaces $\Sigma_t = \{ x\in N_1\mid {\bf t}_2(x)=t\}$ and $\nabla {\bf t}_2$ is a smooth time-like normal vector field along $\Sigma_t$.
By the smooth co-area formula, see \cite{Isaac},
we then find
$$
\int_{T_1}^{t} \int_{\Sigma_t} |Q[\psi_{\ell,k}](\nabla V)|\d\Sigma_t\d t
=
\int_{W_t} |Q[\psi_{\ell,k}](\nabla V)| |\nabla {\bf t}_2|\d V_\gext\leq
C \int_{T_1}^{t} f(s)ds
$$
where we use the compactness of $W$ and that  $c\leq |\nabla {\bf t}_2|\leq C$ in $W$ for some constants $C,c>0$.
On the other hand, by Cauchy-Schwarz inequality and the co-area formula
\begin{align*}
&\int_{W_t} |\mathcal{R}[\psi]\cdot  V \psi_{\ell,k}|\d V_\gext\\
&\leq 
C\int_{T_1}^{t}\Vert \mathcal{R}[\psi] \Vert_{L^2(\Sigma_s)}   \Vert V\psi_{\ell,k} \Vert_{L^2(\Sigma_s)}ds\\
&\leq
C
\int_{T_1}^{t}\Big(\Vert \p_t^\ell\psi \Vert_{H^{k-\ell+1}(\Sigma_s)} 
+
\Vert \p_t^{\ell+1}\psi \Vert_{H^{k-\ell}(\Sigma_s)}
+ \\
&\qquad (\Vert \p_t^\ell\psi \Vert_{H^{k-\ell+1}(\Sigma_s)}
 +
 \Vert \p_t^{\ell+1}\psi \Vert_{H^{k-\ell}(\Sigma_s)})^{2\kappa}
+
 f(s)
\Big)ds,
\end{align*}
where we used \eqref{eq:sobo_integral} and \eqref{eq:fsim}.
Similarly,
$$
\int_{W_t} |P^{k-\ell}\p_t^\ell F\cdot V\psi_{\ell,k}|\d V_g \leq
C \left(\Vert F \Vert_{H^{k}({W_t})} + \int_{T_1}^{t} f(s)ds\right)
$$
It remains to analyse the last integral in the formula  \eqref{eq:first_energy}.
The integral over this null surface can be understood as follows.
Consider a non-vanishing null vector field $n \in \Gamma(T \nullinf^-(T_1))$.
Then one can find a unique function $\lambda:\nullinf^-(T_1) \to \R$ solving the differential equation $d\lambda(n) = 1$ and $\lambda(x) = 1$, when $x\in \Sigma_{T_1}\cap \nullinf^-(T_1)$.
Now, restricting the Lorentzian metric $g$ to the null surface $\nullinf^-(T_1)$ and further restricting to a level set of $\lambda$ shows that $\widetilde g := g\big|_{\nullinf^-(T_1),\lambda}$ is a Riemannian metric on the space-like submanifolds
$$
\nullinf^-(T_1) \cap \{ x\in \nullinf^-(T_1)\mid \lambda(x) = t \}.
$$
Let $\omega$ be the volume form with respect to $\widetilde g$.
Then $\d\nullinf^- := d\lambda \wedge \omega$ is a volume form on $\nullinf^-(T_1)$.
We note that the volume form is uniquely determined after a choice of an affine vector field.
Since $V$ is time-like and $n$ is null, we have
$$
Q[\psi_{\ell,k}]_{ab}V^an^b = (n \psi_{\ell,k})^2 + (Y_1\psi_{\ell,k})^2 + (Y_2\psi_{\ell,k})^2,
$$
where $Y_1,Y_2\in \Gamma ( T\nullinf^-(T_1))$ are space-like and 
\[
\mathrm{span}(V(p),n(p),Y_1(p),Y_2(p))=T_pW.
\]
The information of $\psi$ to transversal directions is absent here and thus on the null surface it is enough to know only derivatives of $\psi$ to directions not transverse to $\nullinf^-$.
It follows from Lemma~\ref{lemma:posdef} by continuity that
$
Q[\psi_{\ell,k}](V,n) \geq 0.
$
Therefore
$$
0\leq \int_{\nullinf^-(T_1)} Q[\psi_{\ell,k}](V,n)\d\nullinf^-
\leq C \Vert \psi \Vert_{H^{k+1}(\nullinf^-(T_1))}.
$$
\subsubsection{Combination of estimates}
We have showed that
\begin{multline}\label{eq:energy1}
f(t)
 \leq f(T_1)
+
C_0 \left(
\Vert F \Vert_{H^{k}(W)}
+
\Vert \psi \Vert_{H^{k+1}(\nullinf^-(T_1))}
\right)\\
+
 C_1 \int_{T_1}^{T_2-\delta} f(s)\d s + C_2 \int_{T_1}^{T_2-\delta} f(s)^{2\kappa} \d s.
\end{multline}
As we want a bound for the Sobolev $H^k$ norm of $\psi$ and the estimate \eqref{eq:energy1} does not include $\Vert \psi \Vert_{L^2(W)}$ on the left-hand side, we need to find suitable estimate for that.

Let $\phi_s$ be the (smooth, global) flow of $\nabla{\bf t}_2$ to the future direction and let $\phi:\R\times N_2 \to N_2$ be the smooth map $\phi(s,x)=\phi_s(x)$.
Let $x\in \Sigma_t$.
By using the fundamental theorem of calculus, we see that
\begin{align}\label{eq:int1}
\psi(\phi_{t_2}(x))^2 
&\leq
2\psi(\phi_{t_1}(x))^2 + 2(t_2-t_1) \int_{t_1}^{t_2} (\p_s\psi(\phi_s(x)))^2 d s
\end{align}
for all $t_1\leq t_2$.

Let $\theta:\Sigma_t \to \R$ be the map
$$
\theta(x) = \{ t\in\R\mid \phi_t(x)\in \Sigma_{T_1}\cup\nullinf^-(T_1)\}.
$$
Note that $\theta$ is well-defined, because $\Sigma_{T_1}\cup \nullinf^-(T_1)$ and $\Sigma_t$ are achronal and each integral curve of $\nabla {\bf t}_2$ starting from $\Sigma_t$ intersects $\Sigma_{T_1}\cup \nullinf^-(T_1)$ once.
Thus, we see that $\theta(x)\leq 0$ for all $x\in \Sigma_t$ and hence by replacing $t_2$ by $0$ and $t_1$ by $\theta(x)$ in \eqref{eq:int1}, we get
\begin{align}\label{eq:int2}
\psi(x)^2 
\leq
2\psi(\phi(\theta(x),x))^2 + 2\theta(\phi(\theta(x),x)) \int_{\theta(x)}^{0} (\p_s\psi(\phi_s(x)))^2 d s.
\end{align}
Here the map $\phi(\theta(x),x)$ is, in fact, one-to-one, because it is the projection from the space-like surface $\Sigma_t$ to the past boundary $\Sigma_{T_1}\cup \nullinf^-(T_1)$ along the integral curves of $\nabla {\bf t}_2$.
Integrating \eqref{eq:int2} over $\Sigma_t$ we obtain%
\begin{equation}\label{eq:L2a}
\begin{split}
\int_{\Sigma_{t}} \psi^2 \d\Sigma_t
&\lesssim
\int_{\Sigma_t} \psi(\theta(\phi(\theta(x),x))) d\Sigma_t
+
\int_{\Sigma_t}
\int_{\theta(x)}^{0} (\p_s\psi(\phi(s,x)))^2 d s
d\Sigma_t\\
&\lesssim
\int_{\nullinf^-(T_1)} \psi(x) d\nullinf^-
+
\int_{\Sigma_t}
\int_{\theta(x)}^{0} (\p_s\psi(\phi(s,x)))^2 d s
d\Sigma_t,
\end{split}
\end{equation}
where we used the fact that $\psi\equiv 0$ on $\Sigma_{T_1}$ and that $\phi(\theta(\cdot),\cdot))^*d\Sigma_t = h_1 d\Sigma_{T_1}$, when $\phi(\theta(x),x))\in \Sigma_{T_1}$ and $\phi(\theta(\cdot),\cdot))^*d\Sigma_t = h_2 d\nullinf^-(T_1)$, when $\phi(\theta(x),x))\in \nullinf^-(T_1)$, for some $h_1,h_2>0$.
This follows from the fact that $\Sigma_t,\Sigma_{T_1}$ and $\nullinf^-(T_1)$ are smooth surfaces and pullback takes top-rank differential form to another top-rank form.
Moreover, by continuity and compactness of $W$, there are $c,C>0$ such that $c \leq h_1,h_2\leq C$.

The last integral in \eqref{eq:L2a} can be estimated by using the co-area formula, 
\begin{equation}\label{eq:L2b}
\begin{split}
  \int_{\Sigma_t}
\int_{\theta(x)}^{0} (\p_s\psi(\phi(s,x)))^2 d s
d\Sigma_t
&\leq C\int_{W_t} |\nabla {\bf t}_2\psi(x)|^2 dV_\gext\leq 
\int_{T_1}^{T_2} f(t) d t.
\end{split}
\end{equation}
Here we also used that for $x_0 = \phi_{t_0}(x)$ we have $\p_s \psi(\phi(s,x))|_{s=t_0} 
= \nabla {\bf t}_2\psi(x_0)$, since $\nabla {\bf t}_2$ is the infinitesimal generator of $\phi_t$.

Combining \eqref{eq:energy1}, \eqref{eq:L2a} and \eqref{eq:L2b} we have that
\begin{multline}\label{eq:3}
f(t) 
\lesssim
\Vert F \Vert_{H^{k}(W)}
+
\Vert \psi\Vert_{H^{k+1}(\nullinf^-(T_1))}
+
\int_{T_1}^{T_2} f(t)dt
+
\int_{T_1}^{T_2} f(t)^{2\kappa}dt.
\end{multline}

To finish the proof we use the following nonlinear Gr\"onwall inequality,
see e.g.  \cite{Willet}.

\begin{proposition}\label{prop:gronwall}
Assume $u_0>0$ and $p\geq 0$, $p\neq 1$ and $v,w,u$ are non-negative continuous functions. Then
\begin{equation}\label{eq:ineqassumption}
u(t) \leq u_0 + \int_0^t \big(v(s)u(s) + w(s)u^p(s)\big) ds
\end{equation}
implies
\begin{equation}\label{eq:gronwall}
u(t) \leq e^{\int_0^t v(s)ds}
\left[
u_0^{1-p} + (1-p)\int_0^t w(s) e^{(p-1)\int_0^s v(r)dr} ds
\right]^\frac{1}{1-p}.
\end{equation}
If $u_0=0$ then, since the above holds for all $u_0=u_*>0$, by taking limit $u_0\to 0$ the above inequality still holds, provided the limit exists.
\end{proposition}
We apply Proposition~\ref{prop:gronwall} to $f(t)$, $v=w=1$ and $p={2\kappa}$, with
$$
u_0 = \Vert F \Vert_{H^{k}(W)}
+ \Vert \psi\Vert_{H^{k+1}(\nullinf^-(T_1))}.
$$
Thus
$$
f(t) \leq C\left(
u_0 + \int_{T_1}^t f(s) + f^p(s) ds
\right)
$$
implies
\begin{align*}
f(t) 
&\leq
C'e^{C(t-T_1)} u_0
\left[
1 - u_0^{p-1}(e^{C(t-T_1)}-1)
\right]^\frac{1}{1-p},
\end{align*}
for all $t\in[T_1,T_2]$, if the initial data $u_0$ is chosen so small that
$$u_0^{p-1}(e^{C(T_2-T_1)}-1) \leq c$$ for some $c<1$.
This concludes the proof of Proposition~\ref{prop:energy_inequality} for the set $W$.
The analogous result in $W_0$ is obtained by removing the boundary condition on $\nullinf^-$
and covering the whole space $W_0$ by local coordinate neighborhoods.
\end{proof}


\begin{proposition}\label{prop:energy_inequality uniqueness}
Suppose $\psi_j \in  H^{k+1}(W)$, $j=1,2$, where $k+1>2$ satisfy
\ba
& &\square_g \psi_j + B\psi_j + A\psi_j^\kappa=F,\quad\text{on } W,\\
& &\psi_j (x)=0,\quad \hbox{for }x\in U,\\
& &\psi_j (x)=h,\quad \hbox{for }x\in \nullinf^-(T_1),
\ea
where $A$ and $B$ are smooth,
 $U$ is an open neighbourhood  of $\{x\in N_1\mid{\bf t_1}(x)=T_1 \}\cup\{i_0\}$ and 
 $\supp(F) \subset I^+(W)$ is compact such that $U\cap\supp(F)=\emptyset$.
Then, $\psi_1=\psi_2$.
\end{proposition}
\begin{proof}
The claim follows readily by applying Proposition~\ref{prop:energy_inequality}
 to the difference $w=\psi_1-\psi_2$ and observing that by 
the Sobolev embedding theorem, $w\in C(\overline W)$.
Since 
\[
s_1^\kappa - s_2^\kappa = (s_1-s_2)\sum_{p=1}^{\kappa-1} s_1^p s_2^{\kappa-1-p} =\vcentcolon p_\kappa(s_1,s_2)\cdot (s_1-s_2),
\]
we see that $w$ satisfies 
\ba
& &\square_g w + Bw + Ap_\kappa(\psi_1,\psi_2)\cdot w =0,\quad\text{on } W,\\
& &w=0,\quad \hbox{for }x\in U,\\
& &w=0,\quad \hbox{for }x\in \nullinf^-(T_1).
\ea
   As $Ap_\kappa(\psi_1,\psi_2)$ is a continuous function, we see using \cite{Nicolas}
   that $w=0$. Hence, $\psi_1=\psi_2$.
\end{proof}

\bibliography{references} 
\bibliographystyle{abbrv}

\end{document}